\documentclass[12pt]{amsart}

\usepackage{amssymb,latexsym}

\usepackage[ansinew]{inputenc}
\usepackage{graphicx}

\numberwithin{equation}{section}

\def\cal{\mathcal}

\def\Bbb{\mathbb}

\def\C{{\Bbb C}}

\def\R{{\Bbb R}}
\def\Z{{\Bbb Z}}

\def\<{\left<}
\def\>{\right>}

\def\PSet{\mbox{\rm I\kern-.22em P}}

\def\S{{\cal S}}

\def\P{{\bf P}}
\def\p{{\bf p}}

\def\I{{\cal I}}

\def\F{{\cal F}}

\def\T{{\bf T}}

\def\({( \hspace{-0.335em}(}
\def\){) \hspace{-0.335em})}
\def\supp{{\rm supp\,}}

\def\fu2{\frac{n}{2}}

\def\l({\left(}
\def\r){\right)}
\def\be{\begin{enumerate}}
\def\ee{\end{enumerate}}

\def\size{{\rm size}}

\def \diam {{\rm diam}}

\def \dist {{\rm dist}}

\newtheorem{theorem}{Theorem}[section]
\newtheorem{definition}[theorem]{Definition}
\newtheorem{lemma}[theorem]{Lemma}
\newtheorem{corollary}[theorem]{Corollary}

\newtheorem{thm}{Theorem}[section]

\newtheorem{remark}[thm]{Remark}

\begin{document}

\title{Modulation invariant bilinear T(1) theorem}
\author[\'A. B\'enyi]{\'Arp\'ad B\'enyi}
\address{Department of Mathematics, Western Washington University, Bellingham, WA 98225}
\email{{\tt arpad.benyi@wwu.edu}}
\author[C. Demeter]{Ciprian Demeter}
\address{Department of Mathematics, UCLA, Los Angeles CA 90095-1555}
\email{demeter@math.ucla.edu}
\author[A.R. Nahmod]{Andrea R. Nahmod}
\address{Department of Mathematics, University of Massachusetts, Amherst, MA 01003}
\email{{\tt nahmod@math.umass.edu}}
\author[C.M.  Thiele]{Christoph M. Thiele}
\address{Department of Mathematics, UCLA, Los Angeles, CA 90095}
\email{{\tt thiele@math.ucla.edu}}
\author[R.H. Torres]{Rodolfo H. Torres}
\address{Department of Mathematics, University of Kansas, Lawrence, KS 66045}
\email{{\tt torres@math.ku.edu}}
\author[P. Villarroya]{Paco Villarroya}
\address{Department of Mathematics, UCLA, Los Angeles, CA 90095}
\email{{\tt pvilla@math.ucla.edu}}


\date{\today}

\begin{abstract}

We prove a  T(1) theorem for bilinear singular integral operators (trilinear forms)
with a one-dimensional modulation symmetry.

\end{abstract}

\maketitle

\section[]{Introduction}
The $T(1)$ Theorem is a criterion that gives necessary and
sufficient conditions for the $L^2$ boundedness of non-convolution
singular integral operators. It arose as a culmination of decade
long efforts to understand the Cauchy integral operator on a
Lipschitz graph and the related Calderón commutators. In the
original statement of the theorem, proved by G. David and J.L. Journé
(\cite{DJ}), the necessary and sufficient conditions are expressed by the requirement that some properly defined functions $T(1)$ and $T^*(1)$ belong to $\rm BMO$ - hence the name of the theorem - together with the so called weak
boundedness property. This latter condition requires  the $L^2$ bounds when
tested weakly on a restricted class of bump functions
$$
|\langle T(\varphi_{x,R}),\varphi_{y,R}\rangle |\lesssim R,
$$
where $\varphi_{x,R}(t)=\varphi(R^{-1}(t-x))$.

In (\cite{ST}), E. Stein reformulated the necessary and sufficient
conditions into what he called the restricted boundedness property. This amounts to the existence of $L^2$ bounds when strongly tested on
the same class of bump functions, that is
$$
\| T(\varphi_{x,R})\|_{2}\lesssim R^{1/2}
$$
and likewise for $T^*$. Both forms of the  $T(1)$ theorem 
will be used in this paper and further developed in a bilinear 
 version to study certain modulation invariant bilinear singular integrals.

A basic operator in the study of the Cauchy integral on a Lipschitz
graph is Calder\'on's first commutator. This operator can be written
as a superposition of bilinear singular integral operators of the
form
$$
T_{\alpha }(f_{1},f_{2})(x)={\rm p.v.}\int_{\mathbb R
}f_{1}(x-t)f_{2}(x-\alpha t)\frac{dt}{t}
$$
with parameter $\alpha \notin \{0,1\}$, called bilinear Hilbert
transforms. One of Calderón's early attempts to bound his commutator
was to show the boundedness of the bilinear Hilbert transforms from
$L^2\times L^\infty $ to $L^2$. However, he gave up on this approach
and proved bounds on the commutator by different means
\cite{calderon1}, \cite{calderon2}.

The bounds for the bilinear Hilbert transform conjectured by
Calderón remained an open problem for more than 30 years. M. Lacey and C. Thiele proved \cite{LT1},\cite{LT2}, that the bilinear Hilbert transforms
are bounded from $L^{p_1}\times L^{p_2}$ to $L^{p}$ for
$1<p_{1},p_{2}\leq \infty $, $p^{-1}=p_{1}^{-1}+p_{2}^{-1}$ and
$2/3<p<\infty $. Appropriate control on the growth of the constants
associated with these bounds as $\alpha$ approaches the forbidden values
$\{0,1\}$, was established in \cite{Th1}. This step was necessary in order to complete Calder\'on's program of estimating the commutator as superposition of bilinear Hilbert transforms. Thiele's results were strengthened to uniform bounds in some range of exponents by L. Grafakos and X. Li \cite{GL1}, \cite{GL2}.

The main feature that distinguishes the bilinear Hilbert Transform from its classical linear counterpart is the fact that the former has modulation invariance. A similar property is shared by Carleson's maximal operator, which controls convergence of the Fourier series, \cite{carleson}. The resolution of both problems resides in the representation of these operators in a wave packet frame that is itself invariant under modulation.

The natural question regarding bounds on more general bilinear singular
integral operators than the bilinear Hilbert transform, where the
kernel $1/t$ is replaced by  more general Calder\'on-Zygmund kernels was first
addressed by J. Gilbert and A. Nahmod (see \cite{GN1}, \cite{GN2}, \cite{GN3}).
They proved bounds for the class of kernels $K(t)$ which are
$x$-independent. Uniform bounds in $\alpha$ for these operators
were then shown by C. Muscalu, T. Tao and C. Thiele (\cite{MTT2}).

The purpose of the current article is to address the case of
kernels $K(x,t)$ that have both  $t$ and $x$ dependence; this corresponds to
the non-convolution case in the classical linear theory. The results we obtain, in particular Theorem \ref{main}, are  different
 in nature from the
bilinear or multilinear $T(1)$ theorems of M. Christ and J.L. Journ\'e \cite{CJ}
and L. Grafakos and R. Torres \cite{GT} since, as we shall soon describe, we treat operators associated with far more singular kernels. The main new feature that distinguishes the operators we analyze is that they have modulation invariance in a certain direction. We seek a theory for them analogous to the one involved in the classical $T(1)$ Theorem. We will  have to incorporate, however,  time-frequency techniques that reflect the modulation invariance of the bilinear operators treated. 

\bigskip

\noindent{\bf Acknowledgments:} C. Demeter,  A.R. Nahmod, C.M.  Thiele and R.H.
Torres were supported in part by NSF under grants DMS-0556389, DMS 0503542, DMS 0400879,
and DMS 0400423, respectively. P. Villarroya was supported  in part by grant MTM2005-08350-C03-03 and EX2004-0510.

\newpage

\section{The main theorem, applications, and road map of the proof}

\subsection{Modulation invariant bilinear $T(1)$ Theorem.\\}
\label{sec2.1}

We start with a few definitions and examples that will lead us into formulating the classical (linear) $T(1)$ Theorem in its dual version and its bilinear counterpart. Throughout the whole paper we will restrict our attention to the one dimensional case. 

\begin{definition}
A function $K:\R\times (\R \setminus \{0\})\to R$ is called a 
Calder\'on-Zygmund kernel if for some $0<\delta\le 1$ and some
constant $C_{K}$ we have
\begin{equation}
\label{e.e1Cal1}
|K(x,t)|\le C_{K} |t|^{-1}
\end{equation}
\begin{equation}
\label{e.e1Cal2}
|K(x,t)-K(x',t')|\le C_{K} \|(x,t)-(x',t')\|^\delta |t|^{-1-\delta}
\end{equation}
whenever $\|(x,t)-(x',t')\|\le |t|/2$, where $\|\cdot\|$ denotes the euclidian norm.
\end{definition}

\begin{definition}
A bilinear form $\Lambda$, defined on the product of Schwartz spaces
$$\Lambda:\S(\R)\times \S(\R)\to \C$$
is said to be associated with a standard Calder\'on-Zygmund kernel $K$
if for some $\beta=(\beta_1,\beta_2)$ and for all Schwartz functions $f_1,f_2\in {\cal S(\R)}$ whose supports are
disjoint\footnote{The requirement that the supports are disjoint is sufficient -due to \eqref{e.e1Cal1}- to guarantee the convergence of the integral; it is also necessary in general, as it is easily seen by working with $K(x,t)=|t|^{-1}$. }, we have
$$\Lambda(f_1,f_2)=\int_{\R^2} f_1(x+\beta_1 t)f_2(x+\beta_2 t) K(x,t)
\, dx \, dt.$$
If the form is continuous on ${\cal
S}(\R) \times {\cal S}(\R)$ then it will be referred to as a bilinear Calder\'on-Zygmund form.
\end{definition}

The above representation of $\Lambda$ is convenient for the
formulation of the trilinear forms that we will study. Note,
however, that the simple change of variables $x':=x+\beta_2t,
t':=x+\beta_1t$ gives the more classical representation
\begin{equation}
\label{SKREP}
\Lambda (f_1, f_2)=\int_{{\R}^2}f_1(t')f_2(x')\tilde K(x',
t')\,dt'dx'
\end{equation}
 where
$$\tilde K(x, t)=\frac{1}{|\beta_1-\beta_2|}K(\frac{\beta_1x-\beta_2t}{\beta_1-\beta_2},
\frac{t-x}{\beta_1-\beta_2})$$ 
satisfies
$$|\tilde K (x, t)|\leq C_K|x-t|^{-1}$$
and
$$|\tilde K(x, t)- \tilde K (x', t')|\le C_K(2+\|\beta\|)^{\delta}
\|(x,t)-(x',t')\|^{\delta}|x-t|^{-1-\delta}$$ whenever $c_\beta
\|(x,t)-(x',t')\| \leq |x-t|$ and $c_\beta=2(1+\|\beta\|)\geq 2$.
 
We will use the notation $\langle \cdot | \cdot \rangle$ to denote the pairing of a distribution with a test function, which we take to be linear in both entries.  We reserve the notation $\langle \cdot , \cdot \rangle$ to denote the usual Hilbert space inner product of $L^2$, conjugate linear in the second entry.
With this notation we can associate to $\Lambda$ the linear dual operators $T$ and $T^*$, continuous from ${\cal S} (\R)$ to ${\cal S'} (\R)$,
given by  $$\Lambda(f_1,f_2)= \langle T(f_1) |  f_2 \rangle = \langle T^*(f_2) |  f_1\rangle .$$
We see from \eqref{SKREP} that the Schwartz  kernel of $T$ restricted away from the diagonal of $\R^2$  agrees with  the function $\tilde K$, as usually stated for linear Calder\'on-Zygmund operators. When convenient in our computations, and without loss of
generality, we will often assume this more classical representation \eqref{SKREP} for $\Lambda$.

\begin{definition}
A trilinear form $\Lambda$ defined on ${\cal S}(\R)\times {\cal
S}(\R) \times {\cal S}(\R)$, is said to be associated with a standard
Calder\'on-Zygmund kernel $K$ if for some $\beta=(\beta_1,\beta_2,\beta_3)$ and for all functions
$f_1,f_2,f_3\in {\cal S(\R)}$ such that the intersection of the
three supports is empty, we have
\begin{equation}\label{interep}
\Lambda(f_1,f_2,f_3)=\int_{{\mathbb R}^2} \prod_{j=1}^3
f_j(x+\beta_j t) K(x,t) \, dx \, dt.
\end{equation}
If the form is continuous on ${\cal S}(\R)\times {\cal
S}(\R) \times {\cal S}(\R)$ then it will be referred to as a trilinear Calder\'on-Zygmund form.
\end{definition}

Now the trilinear form $\Lambda$ is associated to the bilinear dual operators given by
$$
\Lambda (f_1,f_2,f_3)= \langle T_3(f_1,f_2) |  f_3\rangle = \langle T_1(f_2,f_3) |  f_1\rangle =\langle T_2(f_1,f_3) |  f_2\rangle,
$$
but  unlike the bilinear case, $K$ is no longer the restriction of the Schwartz kernel of $T_3$.
In the sequel, we shall assume that $\beta_1,\beta_2,\beta_3$ are pairwise
different, as otherwise the trilinear form reduces to a combination
of a pointwise product and bilinear form.  If needed, by a simple change of variables and 
appropriately modifying the constants involved in the definition of a Calder\'on-Zygmund kernel, 
we can assume $\beta$ to be of unit length and perpendicular to $\alpha =(1,1,1)$. 
Let $\gamma$ be a unit vector 
perpendicular to $\alpha$ and $\beta$, the sign of $\gamma$ being of
no importance. The condition that no two components of $\beta$ are
equal is equivalent to no component of $\gamma$ being zero. The
integral representing  $\Lambda$ 
 for functions with disjoint supports satisfies the {\it modulation symmetry along the direction of $\gamma $}:
\begin{equation}\label{modulationinv}
\Lambda(f_1,f_2,f_3)=\Lambda (M_{\gamma_1\xi } f_1, M_{\gamma_2\xi }f_2, M_{\gamma_3\xi } f_3)
\end{equation}
for all $\xi \in \mathbb R$. Here modulation is defined as $M_\eta
f(x)= e^{2\pi i \eta x} f(x)$. Note, however, that the kernel representation does
not guarantee the modulation invariance (\ref{modulationinv}) for
arbitrary triples of Schwartz functions $f_1,f_2,f_3$. 

\begin{definition}
A trilinear Calder\'on-Zygmund  form $\Lambda$ associated with a standard
kernel $K$  for some $\beta$ is said to have
modulation symmetry in the direction $\gamma$,  with $\gamma$ of unit length and perpendicular to the
plane generated by $\beta$ and $\alpha$, if  \eqref{modulationinv}
is satisfied for all functions
$f_1,f_2,f_3\in {\cal S(\R)}$.
\end{definition}

Let us look at typical examples of these operators given in pseudodifferential form.
Consider again the bilinear Hilbert transform
$$T(f_1, f_2)(x)= {\rm p.v.}  \int_{\mathbb R}f_1(x-t)f_2(x+t)\frac{dt}{t},$$
or equivalently
$$T(f_1, f_2)(x)=\int_{{\mathbb R}^2} {\rm sign} (\xi - \eta) \widehat f_1(\xi)\widehat
f_2(\eta)e^{2\pi ix\cdot (\xi+\eta)}\,d\xi d\eta.$$
More generally, one can consider operators of the form
$$T(f_1, f_2)(x)=\int_{{\mathbb R}^2}m(\xi - \eta)\widehat f_1(\xi)\widehat
f_2(\eta)e^{2\pi ix\cdot (\xi+\eta)}\,d\xi d\eta,$$ where $m$ is a
 multiplier satisfying the classical conditions
$$ |m^{(n)}(u)|\leq C |u|^{-n},\;n\le N.$$
Undoing the Fourier transforms of $f_1$, $f_2$, one arrives to the kernel
representation of $T$, namely
$$T(f_1, f_2)(x)= {\rm p.v.}  \int_{\mathbb R}f_1(x-t)f_2(x+t) K(t) \, dt,$$
where $K$ is a classical CZ kernel of convolution type and $\widehat
K (u) = m(u)$. These bilinear operators fall under the scope of the
more general boundedness results stated in \cite{GN1} and
\cite{MTT1}.

To introduce
$x$-dependent kernels consider now bilinear operators of the form
$$T(f_1, f_2)(x)=\int_{{\mathbb R}^2} \sigma (x, \xi - \eta)\widehat
f_1(\xi)\widehat f_2(\eta)e^{2\pi ix(\xi+\eta)}d\xi d\eta,$$
where
$\sigma (x, u)$ is a symbol in the H\"ormander class $S^0_{1,0}$, so that
$$|\partial^\mu_x\partial_{\xi,\eta}^\alpha \sigma (x, \xi -\eta)|\leq C_\alpha (1+|\xi -\eta|)^{-|\alpha|}.$$
Then, undoing again the Fourier transforms,  we arrive to the following integral representation of $T$
valid at least for functions with disjoint support:
$$T(f_1, f_2)(x)=\int_{{\mathbb R}^2}  K (x, x-y) \delta (z-2x+y) f_1(y)f_2(z)\,dydz$$
$$= \int_{\mathbb R}  K (x, t) f_1(x-t)f_2(x+t)\,dt,$$
where $K(x,x-y)=({\cal F}^{-1} \sigma)(x,x-y)$ and the inverse Fourier transform  is taken in the second variable.
It is well-known that such a $K$ is a  Calder\'on-Zygmund kernel (and with $\delta = 1$).
These bilinear operators give rise then to trilinear forms of the type (\ref{interep}) with
$\beta = (-1,1, 0)$.

In the previous example, the Schwartz  kernel of the bilinear operator $T$ is 
given by  $k(x,y,z)= K (x, x-y) \delta(z-2x+y)$ and hence it is too singular to
fall under the scope of other multilinear $T(1)$ theorems in \cite{CJ} and  \cite{GT}, which 
essentially apply to pseudodifferential operators of the form
$$T(f_1, f_2)(x)=\int_{{\mathbb R}^2} \sigma (x, \xi,\eta)\widehat
f_1(\xi)\widehat f_2(\eta)e^{2\pi ix(\xi+\eta)}d\xi d\eta,$$
where
$ \sigma (x, \xi,\eta)$ satisfies the classical Coifman-Meyer estimates
$$|\partial^\mu_x\partial_{\xi,\eta}^\alpha \sigma (x, \xi,\eta)|\leq C_\alpha (1+|\xi|+|\eta|)^{-|\alpha|},$$
so that the  (restricted) Schwartz kernels satisfy
$$|\partial^\alpha k(x,y,z)| \leq C_\alpha  (|x-y|+|y-z|+|z-x|)^{-(|\alpha|+1)}.$$

\begin{definition}
An $L^p$- normalized bump function $\phi:\R^m\to\C$ is said to be $C$-adapted of order $N$ to a box $I:=I_1\times\ldots\times I_m$ if 
$$|\partial^{\alpha}\phi(x)|\le C\prod_{m'=1}^{m}|I_{m'}|^{-1/p-\alpha_{m'}}\chi_{I}^N(x),$$ 
for each $0\leq  |\alpha| \le N$.
\end{definition}
We will use the notation 
$$\chi_I(x)=\left(1+\left\|\left(\frac{x_1-c(I_1)}{|I_1|},\ldots,\frac{x_m-c(I_m)}{|I_m|}\right)\right\|^2\right)^{-1/2},$$
with $c(I)$ denoting the center of the interval $I$. Often times we will simply call a function $L^p$- adapted to $I$ (or $L^p$- adapted to $I$ of some order $N_0$), if it is $L^p$- normalized and $C_N$- adapted of each order $N$ (or of order $N_0$),  for some $C_N$  whose value will not be specified. 
When no $L^p$ normalization will be mentioned for a bump, it will be implicitly understood that the normalization is taken in $L^2$.

The implicit bounds hidden in the notation $a\lesssim b$ that we shall use, will be allowed to depend on the constants of adaptation and on fixed parameters like $\gamma$, $\delta$, $\alpha$ or $C_{K}$. The notation $A\approx B$ will mean that $A\lesssim B$ and $B\lesssim A$.

Before we state our main result we recall the classical $T(1)$ theorem in a form useful for our purposes.

\begin{theorem}[Linear $T(1)$ theorem]\label{classicalt1}
Assume $\Lambda$ is a Calder\'on-Zygmund form on ${\cal
S}(\R)\times {\cal S}(\R)$. Then $\Lambda$ extends to a bounded
bilinear form on $L^2(\R)\times L^2(\R)$ if and only if there exists
an $N$ such that
\begin{equation}\label{linearrestricted}
|\Lambda(\phi_I,f)|\lesssim\|f\|_2
\end{equation}
$$|\Lambda(f, \phi_I)|\lesssim \|f\|_2$$
for any interval $I$, any $L^2$- adapted bump function $\phi_I$ of order
$N$ which is supported in $I$ and any Schwartz function $f$.
If these equivalent conditions are satisfied, then the bilinear form
extends also to a bounded form on $L^p(\R)\times L^{p'}(\R)$ for
$1<p<\infty$ and $1/p+1/p'=1$.
\end{theorem}

The condition (\ref{linearrestricted}) and its symmetric form are
called the restricted boundedness conditions. In Lemma
\ref{compactclassical} we will see a slightly stronger result,
namely that it is sufficient to test the restricted boundedness
condition only for those $f$ supported in $I$.

\begin{remark}
\label{equivTcondbmo}
There are a few other equivalent formulations of the $T(1)$ theorem.  We just recall, as mentioned in the introduction, that the boundedness of $T$ is also equivalent to  the weak boudedness property (that is $\Lambda(\varphi_I,\varphi_I)\lesssim 1$ for all $\varphi_I$ which are $L^2$- adapted to $I$) together with the fact that some appropriately defined functions $\Lambda(1,.),\Lambda(.,1)$ are in BMO. 
\end{remark}

We now state our new result for trilinear forms.

\begin{theorem}[Main theorem]\label{main}

Assume $\Lambda$ is a trilinear  Calder\'on-Zygmund form on ${\cal
S}(\R)\times {\cal S}(\R)\times {\cal S}(\R)$ associated with a  kernel $K$
with parameter $\delta $, and with modulation
symmetry \eqref{modulationinv} in the direction of $\gamma$ (with
$\gamma_i\neq 0$).

Then $\Lambda$ extends to a bounded form
$$|\Lambda(f_1,f_2,f_3)|\lesssim \prod_{j=1}^3 \|f_j\|_{p_j}$$
for all exponents $2\le p_1,p_2,p_3\le \infty$ with
$$\frac 1{p_1}+\frac 1{p_2}+\frac 1{p_3}=1$$
if and only if there is an\footnote{As in the case of the classical $T(1)$ Theorem, the value of $N$ is not important. Once the theorem holds for some $N$, it also holds for any larger value. This observation will be used repeatedly throughout the argument.} 
$N$ such that the following three
estimates hold for all intervals $I$, all $L^2$- adapted functions
$\phi_I$ and $\psi_I$ of order $N$ which are also  supported in $I$,
and all Schwartz functions $f$
\begin{equation}\label{restrictedbounded}
|\Lambda(\phi_I,\psi_I,f)|\lesssim |I|^{-1/2}\|f\|_2,
\end{equation}
$$|\Lambda(\phi_I,f, \psi_I)|\lesssim |I|^{-1/2}\|f\|_2,$$
$$|\Lambda(f, \phi_I,\psi_I)|\lesssim |I|^{-1/2}\|f\|_2.$$

\vskip 10pt

Moreover, if these equivalent conditions hold, then for
$\sum_j\alpha_j=1$ the following holds: If $0\le \alpha_i
<\min(1/2+\delta,1)$ for $1\le i\le 3$, then
$$|\Lambda(f_1,f_2,f_3)|\lesssim \prod_{i=1}^3 \|f_i\|_{1/\alpha_i}\ .$$
If $\max(-\delta ,-1/2)<\alpha_j<0$ for only one index $j$ and $0\le
\alpha_i <\min(1/2+\delta,1)$ for the other two indices, then the
dual operator $T_j$  satisfies
$$\|T_j((f_i)_{i\neq j})\|_{1/(1-\alpha_j)}\lesssim \prod_{i\neq j}\|f_i\|_{1/\alpha_i}.$$
\end{theorem}

To summarize, there is an a priori estimate for the form $\Lambda$
and a tuple $\alpha$ of reciprocals of exponents provided that
$\sum_j\alpha_j=1$ and $\max(-\delta
,-1/2)<\alpha_j<\min(1/2+\delta,1)$ for all $j$. Interestingly, the range of exponents $\alpha$ for which the theorem guarantees boundedness is the same for each $\delta\in[\frac12,1]$, while the range shrinks for $\delta<\frac12$, as $\delta$ approaches 0. We do not know if this range is optimal. Note also that for $\delta\ge 1/2$ in Theorem \ref{main} we recover the same range in which the bilinear Hilbert transform is known to be bounded. 
\\

It is worthwhile noting that the necessity of condition
(\ref{restrictedbounded}) and its symmetric counterparts is clear,
as such conditions follow from the claimed estimates applied to special test
functions. We call these conditions the {\it (trilinear) restricted boundedness
conditions}. One can see that it is also enough to  test the conditions  for $C^\infty$ functions $f$ supported
in an interval containing $I$ of length $C|I|$, where $C>0$
 is a
universal constant. See Lemma~\ref{equivalentrestricted} below for more details.

Note  also that, formally, if $\Lambda$ is a
trilinear Calder\'on-Zygmund form satisfying the conditions of the main theorem, then
$$\Lambda(f_1,f_2,1)$$ is a bilinear Calder\'on-Zygmund form satisfying the conditions of the
classical $T(1)$ Theorem. We will make this reduction to bilinear forms
rigorous in the next section. However, not every bounded bilinear
Calder\'on-Zygmund form can be obtained this way, since for the
trilinear form to be bounded more conditions need to be satisfied.
We will see concrete examples in Section~\ref{sec7}.

\subsection{ An application of the main theorem.\\}

Of course, the relevance of  Theorem \ref{main} is that it applies to operators with $x$-dependent kernels. We present one application to the bilinear pseudodifferential operators mentioned before.\\

  Consider again the trilinear form
$$\Lambda (f_1,f_2,f_3) = \int_{\R} T_3(f_1,f_2)(x) f_3(x)\,dx$$
$$=\int_{{\R}^3}  \sigma (x, \xi - \eta)\widehat
f_1(\xi)\widehat f_2(\eta) f_3(x) e^{2\pi ix(\xi+\eta)}d\xi d\eta dx$$
with $\sigma$ in $S^0_{1,0}$. Note that this form has modulation symmetry in the
direction $\gamma = (1,1,-2)/\sqrt{6}$ for all triples $f_1$, $f_2$, $f_3$, not just the ones with
disjoint supports.  To check the first of the restricted bounded conditions  we may
assume $f$ is supported in $CI$ and compute
$$|\Lambda (\phi_I, \psi_I, f)| \lesssim \|\widehat \phi_I \|_{L^1} \|\widehat \psi_I \|_{L^1}
\|f \|_{L^1} \lesssim \|\widehat \phi_I \|_{L^1} \|\widehat \psi_I \|_{L^1} |I|^{1/2}\|f \|_{L^2}$$
$$\lesssim |I|^{-1/2} \|f \|_{L^2}.$$
Here, we used that $\widehat \phi_I$ and $\widehat \psi_I$ are $L^2$-normalized and adapted to intervals of length $|I|^{-1}$. More precisely, $\phi_I$ can be written as
$\phi_I (x)= |I|^{-1/2} \phi_0((x-x_0)/|I|)$ where $\phi_0$ is adapted to and supported in the unit
interval centered at the origin. It follows easily now that
$$\|\widehat \phi_I \|_{L^1}=|I|^{-1/2} \|\widehat \phi_0 \|_{L^1}\leq C |I|^{-1/2}$$
where $C$ depends only on finitely many derivatives of $\phi_0$.  The same estimate applies to  $\psi_I$. To
obtain the other restricted boundedness conditions, write
$$\Lambda (\phi_I, f, \psi_I) = \int_{\R} T_2(\phi_I,\psi_I)(x) f(x) \, dx$$
and
$$\Lambda ( f, \phi_I, \psi_I) =\int_{\R} T_1(\phi_I,\psi_I)(x) f(x) \, dx.$$
It was proved by \'A. B\'enyi, A. Nahmod and R. Torres \cite{BNT}, that $T_1$ and $T_2$  can be computed
from $\sigma$ and they admit pseudodifferential representations of the form\footnote{\, $T_1(f,g)=T^{*1}(g,f)$ and $T_2(f,g)=T^{*2}(f,g)$ in the notation of \cite{BNT}.}
$$T_2(f,g)(x)=\int_{{\R}^2}  \sigma_2 (x, \xi,\eta)\widehat
f(\xi)\widehat g(\eta)  e^{2\pi ix(\xi+\eta)}d\xi d\eta$$
and
$$T_1(f,g)(x)=\int_{{\R}^2}  \sigma_1 (x, \xi,\eta)\widehat
g(\xi)\widehat f(\eta)  e^{2\pi ix(\xi+\eta)}d\xi d\eta$$
where $ \sigma_2 (x, \xi,\eta)$  and $ \sigma_1(x, \xi,\eta)$  satisfy
$$|\partial^\mu_x\partial_{\xi,\eta}^\alpha \sigma_2 (x, \xi,\eta)|\lesssim (1+|2\xi+\eta|)^{-|\alpha|}$$
and
$$|\partial^\mu_x\partial_{\xi,\eta}^\alpha \sigma_1 (x, \xi,\eta)|\lesssim (1+|\xi+2\eta|)^{-|\alpha|}.$$
The computations done with $T_3$ can now be repeated with $T_2$ and $T_1$. It follows that $\Lambda$ and the $T_j$
have the boundedness properties of Theorem~\ref{main} with $\delta=1$.

Similar examples of forms  can be obtained by starting with a bilinear operator $T_3$
given by a symbol of the form $\sigma_\theta(x,\xi,\eta)=\sigma(x,\eta- \xi \tan \theta)$  for
$\sigma$ in $S^0_{1,0}$ and $\theta \not= -\pi/4,, 0, \pi/2$ (we make the convention $\sigma_{\pi/2}(x,\xi,\eta)=\sigma(x,\xi)$).
In the three forbidden cases when $\theta= -\pi/4,, 0, \pi/2$ the trilinear forms correspond again to a combination of
a pointwise product and a bilinear form.  See \cite{BNT} for more details.\\

\subsection{Plan of the proof.\\}

The rest of this article is structured as follows. In Section \ref{preliminaries},
after some basic reductions, we present some equivalent
formulations of the main theorem. In particular, in analogy to the
classical  $T(1)$ theorem, we give a meaning to the functions
$T_j(1,1)$ for $j=1,2,3$ (see Lemma~\ref{uniformBMO}  below) and show that the
restricted boundedness conditions imply that these functions are in
$\rm BMO$. We also observe that the restricted boundedness conditions
imply a certain weaker  one (\ref{weakbounded}), which together with the conditions
$T_j(1,1) \in$ BMO is all what will be used  to prove the main theorem. Hence 
this set of conditions is also necessary and sufficient to obtain the bounds on the trilinear form.

In Section~\ref{canccond1} we establish some bounds on the action of  trilinear  Calde\'on-Zygmund forms satisfying the alluded weak continuity \eqref{weakbounded} and some special cancellation conditions on bumps functions.
These are almost orthogonality type conditions.
The proof of the Theorem \ref{main} then splits into two steps.
First, one proves the theorem under the special cancellation
condition that $T_j(1,1)=0$ for all $j$.  This step is done in
Sections \ref{canccond} and \ref{sec6}. In the former the
problem is reduced to a time-frequency model form, which
 is  then estimated  in the latter.

The second step in the proof of the theorem is to construct  for each
given $\rm BMO$ function $b$, forms
$\Lambda_j$, $j=1,2,3$, which are associated to Calder\'on-Zygmund kernels, have  a given modulation symmetry, satisfy the bounds of the theorem, and are such that the corresponding
dual operators satisfy
$T_j(1,1)=b$ and 
$T_i(1,1)=0$ for $i \neq j$
By analogy again with the classical T(1) Theorem, we call these special forms modulation invariant  paraproducts. 
Then, Theorem \ref{main} can be always reduced to the case with special cancellation by
subtracting from the original form $\Lambda$ three  paraproducts with the same modulation symmetry. The paraproducts are discussed in Section \ref{sec7}.

\section{Alternative formulations and the role of $\rm BMO$}\label{preliminaries} 

We begin this discussion by a lemma
that implies the strengthening of the classical Theorem~\ref{classicalt1} 
that was mentioned after the statement of the theorem. This will be used in the proof of Lemma \ref{equivalentrestricted}.

From now on, for each box $I$ in $\R^n$ and each $R>0$, $RI$ will denote the box with the same center as $I$ and sidelengths $R$ times larger than those of $I$.

\begin{lemma}\label{compactclassical}
Assume $\Lambda$ is a continuous bilinear form on ${\cal
S}(\R)\times {\cal S}(\R)$ that is associated with a standard
Calder\'on-Zygmund kernel $K$. Assume there is an $N$ such that the
restricted boundedness condition
$$|\Lambda(\phi_I,f)|\le C_0\|f\|_2$$
holds for all bump functions $\phi_I$ of order $N$ adapted to and
supported in $I$ and all Schwartz functions $f$ supported in $I$.
Then, 
$$|\Lambda(\phi_I,f)|\le C\|f\|_2$$
 (with a possibly larger constant $C$) also  holds
for all bump functions $\phi_I$ of order $N$ adapted to and
supported in $I$ and all Schwartz functions $f$ (not necessarily supported in $I$).
\end{lemma}

\begin{proof}
Without loss of generality, we may assume  that $I$ is centered at the origin. 
Also, by the continuity of $\Lambda$ on 
${\cal S}(\R)\times {\cal S}(\R)$, it is enough to prove the conclusion for 
all $f$ in ${\cal S}$ with compact support.  Consider first  the case when $\phi_I$ has mean
zero. Let $R$ be a  large constant chosen later depending on $\beta$. Decompose now in a smooth way
 $f=f_1+f_2$,  with $f_1$ supported in
$2RI$ and $f_2$ supported outside $RI$.  Since $(2R)^{-(1/2+N)}\phi_I$ is adapted
to and supported in $2RI$, the hypotheses of the Lemma give then
$$|\Lambda(\phi_I,f_1)|\le C_0 (2R)^{(1/2+N)} \|f\|_2.$$
To estimate 
 $\Lambda(\phi_I, f_2)$ we use the kernel representation 
with the change of variables explained in the introduction to get
$$
\Lambda(\phi_I, f_2)=
\int \phi_I(u) f_2(v) \tilde{K}(u,v)\, du\, dv,
$$
where $\tilde K$ satisfies $$|\tilde{K}(u,v)-\tilde{K}(u',v')|\leq
C_{\beta_1, \beta_2}|u-v|^{-1-\delta }\|(u,v)-(u',v')\|^\delta
$$
for $c_\beta \|(u,v)-(u',v')\| <|u-v|$ and $c_\beta\geq 2$. The
domain of integration is given by $|u|<|I|/2$ and $|v|>R|I|/2$,
 which imply that $|u-v|>(R-1)|I|/2>c_\beta|u|$ if $R$ is large enough.
 Thus, using the
mean zero of $\phi_I$ we can write
$$
\Lambda(\phi_I, f_2)=\int \phi_I(u) f_2(v) (\tilde{K}(u,v)-\tilde{K}(0,v))\, du\, dv
$$
and  obtain
$$
|\Lambda(\phi_I,f_2)|\leq C_{\beta_1, \beta_2}\int|\phi_I(u)|
|f_2(v)| |u|^{\delta }|u-v|^{-(1+\delta )}\, du\, dv
$$
$$
\leq C_{\beta_1, \beta_2}|I|^\delta \int_{|u-v|>c|I|} |\phi_I(u)| |f_2(v)|
|u-v|^{-(1+\delta )}\, du\, dv
$$
$$
\leq C_{\beta_1, \beta_2} |I|^\delta \|\phi_I\|_2 \| f_2\|_2
\int_{|x|>c|I|} |x|^{-(1+\delta )}\, dx \leq C_{\beta_1,
\beta_2} \| f\|_2.
$$
This proves 
$$|\Lambda(\phi_I,f)|\le C_1\|f\|_2$$
 for an appropriate constant $C_1$ under the additional assumption that $\phi_I$ has mean zero.

 To treat the general case, define $\phi_{2^jI}(x)=2^{-j/2}\phi_{I}(2^{-j}x)$ for $j>0$ and 
 observe that $2^{-(1/2+N)}(\phi_{2^jI} - 2^{-1/2}\phi_{2^{j+1}I})$ is adapted to and
 supported in $2^{j+1}I$ and has mean 
 zero.  Let now $k\geq 0$ be the smallest integer such that the support of $f$ is contained in $2^{k+1}I$,
 and write 
 $$
 \Lambda(\phi_I, f)= \sum_{j=0}^k  2^{-j/2}\Lambda(\phi_{2^jI} - 2^{-1/2}\phi_{2^{j+1}I},  f) + 
2^{-(k+1)/2}\Lambda(\phi_{2^{k+1}I},  f).
$$
In the first $k+1$ terms the bumps have mean zero, while the last term can be controlled by 
the hypothesis. Thus,
$$
|\Lambda(\phi_I, f)| \leq (C_1 2^{1/2+N} \sum_{j=0}^k  2^{-j/2} + C_0 2^{-(k+1)/2})\| f\|_2
\leq C \| f\|_2,
$$
where $C$ is independent of $k$.
\end{proof}

We continue by studying the relationship between bilinear and trilinear forms.
By the Schwartz kernel theorem, the trilinear form $\Lambda$ can be
represented by a tempered distribution in $\R^3$, which we shall also denote by
$\Lambda$, so that
$$\Lambda(\phi_1,\phi_2,\phi_3)=\Lambda(\phi_1\otimes \phi_2\otimes \phi_3).$$
In this way, $\Lambda (\phi)$ has a meaning for any $\phi\in
\S(\mathbb R^3)$ not necessarily a tensor product.

Moreover, the modulation invariance of $\Lambda$ implies that the distribution
is supported on the orthogonal complement of $\gamma$. Even
stronger, it is given by a two dimensional distribution $\Lambda_*$
applied to the
restriction $ \left. \phi \right|_{\gamma^\perp}:\gamma^\perp\to\C$ of the test function $\phi$ to $\gamma^\perp$. This can be seen as follows.
Consider first a function $f$ in $C^\infty_0(\mathbb R^3)$ and pick another function 
$\varphi$ in $C^\infty_0(\mathbb R^3)$  with $\varphi \equiv 1$ on a neighborhood 
of the support of $f$. By the  linearity, 
continuity, and modulation invariance of $\Lambda$, and writing $\xi=t\gamma+\xi^*$ with 
$\xi^*$ in $\gamma^\perp$, we get 
$$
\Lambda(f)= \Lambda(\varphi f)= \Lambda(\int \varphi(x) e^{2\pi ix\xi} \widehat f(\xi) d\xi)
= \Lambda(\int_{\mathbb R} \int_{\gamma^\perp}  \varphi(x) e^{2\pi ix \gamma t}   e^{2\pi ix \xi^*} \widehat f(t\gamma+\xi^*) dt d\xi^*)
$$
$$
= \int_{\mathbb R} \int_{\gamma^\perp} \Lambda(\varphi(x) e^{2\pi ix \gamma t}e^{2\pi ix \xi^*} )  \widehat f(t\gamma+\xi^*) dt d\xi^*
= \int_{\mathbb R} \int_{\gamma^\perp} \Lambda(\varphi(x) e^{2\pi ix \xi^*} )  \widehat f(t\gamma+\xi^*) dt d\xi^*
$$
$$
= \Lambda( \int_{\mathbb R} \int_{\gamma^\perp} \varphi(x) e^{2\pi ix \xi^*}  \widehat f(t\gamma+\xi^*) dt d\xi^*)
= \Lambda( \int_{\mathbb R^3} \varphi(x) e^{2\pi ix^* \xi}  \widehat f(\xi)  d\xi)
= \Lambda(\varphi f_{*}),
$$
where $f_{*}(x)=f(x^*)$. (Note that $\varphi f_{*}$ is still in  $C^\infty_0(\mathbb R^3)$).  In particular,
if two functions $f_1$ and $f_2$ in $C^\infty_0(\mathbb R^3)$ agree on $\gamma^\perp$,
then $\Lambda (f_1-f_2)=0$.
Also, since $C^\infty_0(\mathbb R^3)$ is dense in $\S(\mathbb R^3)$, a simple limiting argument shows that
$\Lambda (f) =0$ for all $f \in \S(\mathbb R^3)$ with $\supp f \cap \gamma^\perp = \emptyset$ and so
$\supp \Lambda$ is contained in $ \gamma^\perp$. 

Now, let $\tilde \psi$ be in $C^\infty_0(\mathbb R)$,
 $\tilde \psi \equiv 1$ in a neighborhood of zero, and define $\psi(x)= \psi (t\gamma + x^*) =\tilde \psi (t)$. We have already seen that for $f \in C^\infty_0(\mathbb R^3)$, $\Lambda (f) $ depends only on 
$ \left. f \right|_{\gamma^\perp}$ so 
$$\Lambda (f) = \Lambda (\psi f) = \Lambda (\psi f_{*})$$
and a limiting argument gives the same result for $f \in \S(\mathbb R^3)$. This also allows to define a 
distribution in $\S'(\gamma^\perp)$ by 
$$\Lambda_*(g)=\Lambda (\psi \tilde{g})$$
for all $g \in {\cal S}(\gamma^\perp)$, where $\tilde{g}:\R^3\to\C$ satisfies $\tilde{g}(t\gamma + x^*)=g(x^*)$. The definition is clearly independent of the choice of $\psi$ and we have $$\Lambda (f)=\Lambda_*(\left. f \right|_{\gamma^\perp}).$$

Clearly, by a density argument,  the kernel representation of $\Lambda$ continues to hold when the
test function $\phi$ is no longer a tensor product but still  has support disjoint from the span of $(1,1,1)$.

Moreover, 
$$\Lambda(\phi)=\int \phi(x\alpha+t\beta) \,K(x,t) \, dxdt =\int \phi(x+\beta_1t,x+\beta_2t,x+\beta_3t)\, K(x,t) \, dxdt$$
and  the integral
is absolutely convergent  as long as function $\phi$ just vanishes on the span of $(1,1,1)$, but provided  that $\Lambda$ satisfies a weak  boundedness property (\ref{weakbounded}) below, which is implied by the restricted boundedness conditions.  To verify this, the reader  may adapt to the case of trilinear forms the arguments used in, for example,  \cite{To} for bilinear ones.

There are three distinct 
bilinear forms that we can consider now
$$\Lambda_1(\phi_2,\phi_3)=\Lambda_*(\left. (1\otimes \phi_2\otimes \phi_3)\right|_{\gamma^{\perp}})$$
$$\Lambda_2(\phi_1,\phi_3)=\Lambda_*(\left. (\phi_1\otimes 1 \otimes \phi_3)\right|_{\gamma^\perp})$$
$$\Lambda_3(\phi_1,\phi_2)=\Lambda_*(\left. (\phi_1\otimes \phi_2 \otimes 1)\right|_{\gamma^\perp})$$
Observe that the functions on the right hand side are in ${\cal S}(\gamma^\perp)$, since none of the components of $\gamma$ is zero. Moreover in case of disjointly supported functions
$\phi_2$, $ \phi_3$ we obtain a kernel representation for $\Lambda_1$ with the same kernel $K$. In fact, since $\psi (1\otimes \phi_2\otimes \phi_3) \in \S(\mathbb R^3)$ and obviously has the same values
as $\psi (1\otimes \phi_2\otimes \phi_3)_{*} $ on $\gamma^\perp$,
$$\Lambda_1(\phi_2,\phi_3)= \Lambda_*(\left. (1\otimes \phi_2\otimes \phi_3)\right|_{\gamma^{\perp}}) = \Lambda (\psi (1\otimes \phi_2\otimes \phi_3)_{*} )=
\Lambda (\psi (1\otimes \phi_2\otimes \phi_3))$$$$=
\int  \psi(x\alpha+t\beta)\phi_2(x+\beta_2t)\phi_3(x+\beta_3t)K(x,t)\, dxdt=
\int \phi_2(x+\beta_2t)\phi_3(x+\beta_3t)K(x,t)\, dxdt.$$
Similarly for $\Lambda_2$ and $\Lambda_3$.\\

In the sequel we will use the following notation. 

\begin{definition}
For $a>0, v\in \R^d$ the
{\it $L^p$-normalized dilation operator $D_a^p$} and the {\it translation operator $\tau_v$} are defined via
$$D_a^p f(x)=a^{-d/p}f(a^{-1}x),\, \tau_v f(x)=f(x-v),$$
for all functions $f$ defined on $\R^d$. We will also sometimes write 
$$D_a  f(x)= D_a^\infty f(x)= f(a^{-1}x).$$
\end{definition}

\begin{lemma}\label{equivalentrestricted}
The restricted boundedness conditions for $\Lambda$ are equivalent
to the restricted boundedness conditions for the $\Lambda_i$'s, if
one is willing to have a loss in the order of the bump functions and
the constants involved in defining restricted boundedness.
\end{lemma}
\begin{proof}
We assume first that $\Lambda$ satisfies the restricted
boundedness conditions, and prove restricted boundedness of
$\Lambda_1$ ($\Lambda_2$ and $\Lambda_3$ can be done analogously).
Let $\Phi:\R\to [0,1]$ be smooth with
\begin{equation}
\label{Phi}
\begin{array}{ll}
\Phi(x)=1 & |x|\le 1\\
\Phi(x)=0 & |x|\ge 2.
\end{array}
\end{equation}
It is easy to check that if $f$ is supported in the interval $I$ (Lemma \ref{compactclassical} allows us to restrict attention to this case) then
$$\left. (1\otimes \phi_I\otimes f)\right|_{\gamma^{\perp}}
= \left. (\tau_{c(I)}D_{C|I|}\Phi\otimes \phi_I\otimes f)\right|_{\gamma^{\perp}},$$
where $C$ is a sufficiently large constant
depending only on $\gamma$.
To obtain restricted boundedness for $\Lambda_1$, we simply apply the restricted boundedness
of $\Lambda$ for the enlarged interval $CI$.

 For the converse, we first show, as claimed in the introduction,
 that it is enough to show the restricted boundedness estimates for
$\Lambda(\phi_{I}\otimes \psi_I\otimes f)$ for every $C^\infty$  function
$f$ supported in $CI$, where $C$ is a universal constant depending on $\gamma$. 
To see this, and without loss of generality, we may assume $I$ centered at the origin. The region 
$|x+\beta_1t| < |I|/2$ and $|x+\beta_2t| < |I|/2$ is a parallelogram centered at the origin in the $(x,t)$ 
while $|x+\beta_3t|> C|I|/2$ is the complement of a strip along the line $x+\beta_3t =0$. 
Recall that the components of $\beta$ are pairwise distinct, so  if $C$
is large enough depending only on $\beta$, then the two regions do not intersect.
Decompose now in a smooth way $f=f_{CI}+ (f-f_{CI})$, where $f_{CI}$ coincides with $f$ on $CI$, is zero outside $2CI$ and satisfies $\|f_{CI}\|_2\le 2\|f_{I}\|_2$
Note that $\phi_I\otimes \psi_I\otimes (f-f_{CI})(x\alpha+t\beta)\equiv 0$, so we must  have 
$\Lambda(\phi_I\otimes \psi_I\otimes (f-f_{CI}))=0$, and the claim follows.

 Assume now that all the $\Lambda_i$ satisfy restricted boundedness
conditions of order $N$.
For simplicity of notation also assume that $I$ is
of length $1/C$ and centered at the origin and that $\phi_I,\psi_I$ are bump functions adapted of order $N'\gg N$ and supported in $I$. Let $f$ be a smooth function supported in $CI$. Let $g$ be adapted of order $N'$ and supported on $\{|u|<1/(2C)<1/2,\;|v|<1/2\}$ and so that $g(u,v)\psi_I(u)f(v)=\psi_I(u)f(v)$ on 
the same region. Write $\tilde \phi(u,v) = \phi_I(-\gamma_1^{-1}(\gamma_2 u+\gamma_3 v))$ and note that $\tilde \phi g$ is adapted of order $N'$ to $[-1/2,1/2] \times [-1/2,1/2]$. Expanding
$\tilde \phi g$ in Fourier series on $[-1/2,1/2] \times [-1/2,1/2]$ we get 
$$ \phi_I(x+\beta_1t)\psi_I(x+\beta_2t)f(x+\beta_3t)$$$$ = \tilde \phi(x+\beta_2t, x+\beta_3t) g(x+\beta_2t, x+\beta_3t)\psi_I(x+\beta_2t)f(x+\beta_3t)$$
$$=\sum_{m_2,m_3}c_{m_2,m_3}
\psi_I(x+\beta_2t)e^{2\pi i m_2 (x+\beta_2t)} f(x+\beta_3t)e^{2\pi i m_3 (x+\beta_3 t)},$$
with 
$$|c_{m_2,m_3}|\lesssim(1+\max(|m_2|,|m_3|))^{-N'}.$$
Now, the function
$$(1+|m_2|)^{-N}\psi_I(x)e^{2\pi i m_2 x}$$
is a bump function of order $N$ adapted to the interval $I$. The factor
$(1+|m_2|)^{-N}$ is needed to offset the loss of powers of $m_2$ when taking
derivatives of the exponential factor. Applying the restricted boundedness of $\Lambda_1$ proves that
$$\Lambda(\phi_{I}\otimes \psi_I\otimes f)=\sum_{m_2,m_3}c_{m_2,m_3}(1+|m_2|)^{N}\Lambda_1((1+|m_2|)^{-N}M_{m_2}\psi_I,M_{m_3}f)\lesssim \|f\|_2.$$
\end{proof}

Assuming Theorem \ref{main} is true, we have thereby seen the following
corollary of Theorems \ref{classicalt1} and \ref{main}.

\begin{theorem}
Let $\Lambda$ be a trilinear Calder\'on-Zygmund form with modulation symmetry in
direction $\gamma$. Then $\Lambda$ extends to bounded trilinear form
with exponents as in Theorem \ref{main} if and only if $\Lambda_i$
for $i=1,2,3$ are bounded in $L^2\times L^2.$
\end{theorem}

The lemma below will allow us to state the restricted boundedness
property in a slightly more general way. We need to consider bumps which may no longer be compactly supported but are still concentrated around appropriate intervals.

\begin{lemma}\label{ctos}

Assume $\Lambda$ is a  Calder\'on-Zygmund trilinear form with  modulation symmetry in  the direction $\gamma$ and 
that satisfies the restricted boundedness conditions of Theorem \ref{main} for some $N$.
Then for all intervals $I$, all $L^2$- normalized bump
functions $\phi_{I\times I}$ $L^2$- adapted to $I\times I$ of order $N'>>N$  and all functions 
$f\in {\cal S(\R)}$ we have the estimate
$$|\Lambda(\phi_{I\times I} \otimes f)|\lesssim |I|^{-1/2}\|f\|_2$$
and the symmetric inequalities. \end{lemma}

\begin{proof}
For simplicity of notation we shall assume $I$ is centered at the origin
and of length $1$. 
We take a lacunary decomposition related to $I$:
let $\Phi$ be as in (\ref{Phi}) and define
$$\phi_0= (\Phi\otimes \Phi)\phi_{I\times I}$$
$$\phi_k= (\Phi_k \otimes \Phi_k-\Phi_{k-1}\otimes \Phi_{k-1})\phi_{I\times I}$$
for $k>0$, where $\Phi_k=D^\infty_{2^{k}}\Phi$.
Note that the functions $\phi_k$ add up to $\phi_{I\times I}$, so by
using the continuity of $\Lambda $ in ${\cal S}(\R^3)$ we obtain
\begin{equation}
\label{e.e1}
\Lambda(\phi_{I\times I}\otimes f)=\sum_{k \ge 0}\Lambda(\phi_k\otimes f).
\end{equation}

Note also that  $2^{5k}\phi_k $
is $L^2$-adapted to and supported in $2^{k+2}(I\times I)$. We write
$$\phi_k(x_1,x_2)=\widetilde{\phi_k}(x_1,x_2) \Phi_k(x_1) \Phi_k(x_2)$$
where $\widetilde{\phi_k}$ means the periodization of $\phi_k$ from the square $2^{k+2}(I\times I)$.
By performing a windowed Fourier series, we obtain
$$\phi_k(x_1,x_2)=\sum_{m_1,m_2\in \mathbb Z} c_{k,m_1,m_2}\Phi_k(x_1)\Phi_k(x_2)e^{2\pi i 2^{-(k+2)}(m_1x_1+m_2x_2)}$$
where the coefficients $c_{k,m_1,m_2}$ are rapidly decaying in the sense that
\begin{equation}
\label{e.e2}
|c_{k,m_1,m_2}|\lesssim 2^{-5k} (1+\max(|m_1|,|m_2|))^{-2N-5}
\end{equation}
Denoting by
$\Phi_{k,m}(x)=\Phi_k(x)e^{2\pi i 2^{-(k+2)}mx}$
we have that the functions $|m|^{-N}2^{-k/2}\Phi_{k,m}$ are $L^2$- normalized, adapted of order $N$ and supported in $2^{k+2}I$, uniformly in $m$ and $k$.

By \eqref{e.e1} we have
\begin{align*}
|\Lambda(\phi\otimes & f)|\le \sum_{k\geq 0}\sum_{m_1,m_2\in \mathbb Z}
|c_{k,m_1,m_2}| |\Lambda (\Phi_{k,m_1} ,\Phi_{k,m_2},f)|\\&\le \sum_{k\geq 0}\sum_{m_1,m_2\in \mathbb Z}|m_1|^{N}|m_2|^N2^k|c_{k,m_1,m_2}\Lambda (|m_1|^{-N}2^{-k/2}\Phi_{k,m_1} ,|m_2|^{-N}2^{-k/2}\Phi_{k,m_2},f)|
\end{align*}
Now the estimate of the lemma follows by applying \eqref{e.e2} and the restricted boundedness condition to each summand on the right hand side
with interval $2^{k+2}I$.
\end{proof}

We will now give a rigorous definition of $T_i(1,1)$ as a distribution modulo constants
and prove that the restricted boundedness property implies that they are
elements of BMO. The approach is similar 
to the linear case and we will follow some of the arguments in \cite{ST}. See also \cite{To} and \cite{GT}
for similar linear and multilinear definitions.
We start with the following lemma. 

\begin{lemma}\label{definebmo}
Assume that $\Lambda$ is a trilinear Calder\'on-Zygmund form with modulation
symmetry in the direction $\gamma$. Let $\Phi$ be as in \eqref{Phi}. For every $C_0^\infty$ function $f$ 
 with mean zero the limit
\begin{equation}
L(f)=
\lim_{k\to \infty} \Lambda ( D^\infty_{2^k}(\Phi \otimes \Phi) \otimes f) =
\lim_{k\to \infty}  \langle T_3(D^\infty_{2^k}\Phi, D^\infty_{2^k}\Phi)|f\rangle \label{L}\end{equation}
exists. Moreover, if $\operatorname{supp} f  \subset (-2^{k_0},2^{k_0})$, then for sufficiently large $k$ (depending on $k_0$ and $\gamma$) we have the error bound
\begin{equation}
\label{deferrorb12}
|L(f)- \Lambda(D^\infty_{2^k}(\Phi \otimes \Phi)\otimes f)|\lesssim 2^{-\delta (k-k_0)} \|f\|_{L^1}
\end{equation}
where $\delta$ is the parameter in the Calder\'on-Zygmund property of
the kernel $K$ and the implicit constant is independent of $k_0, k$ and $f$.

\end{lemma}

\begin{proof}
Assume  $\operatorname{supp} f \subset (-2^{k_0},2^{k_0})$.
For all $j>k\gg  k_0$ write 
$$D^\infty_{2^j}(\Phi \otimes \Phi)-D^\infty_{2^k}(\Phi \otimes \Phi)=
 \sum_{l=1}^{j-k} D^\infty_{2^{l+k}}(\Phi \otimes \Phi)-D^\infty_{2^{l+k-1}}(\Phi\otimes \Phi)$$
and let $\psi_l= D^\infty_{2^{l+k}}(\Phi\otimes\Phi)-D^\infty_{2^{l+k-1}}(\Phi\otimes\Phi)$. We will estimate
$|\Lambda(\psi_{l}\otimes f)|$ to prove that the sequence 
$\Lambda ( D^\infty_{2^k}(\Phi \otimes \Phi)\otimes f)$ is Cauchy, as well as to estimate the error bound.
Since the support of $\psi_l\otimes f$ is disjoint from the span of $(1,1,1)$, we can use the kernel
representation of $\Lambda$
$$
\Lambda(\psi_l\otimes f)
=\int \psi_l(x+\beta_1t,x+\beta_2t)f(x+\beta_3 t) K(x,t)\, dxdt.
$$
Similarly to what we did in Section \ref{sec2.1} for bilinear forms, a simple change of variables allows us 
to write the above integral in the form
$$
\int \psi_l (u,v)f(w) \tilde{K}(u,w)\, dudw
$$
where $v=-\gamma_2^{-1}(\gamma_1 u+\gamma_3w)$ and $\tilde K$ satisfies the classical Calder\'on-Zygmund
estimates
\begin{equation}
\label{CCZZ1}
|\tilde{K}(u,w)|\lesssim |u-w|^{-1},
\end{equation}
\begin{equation}
\label{CCZZ2}
|\tilde{K}(u,w)-\tilde{K}(u',w')|\lesssim |u-w|^{-(1+\delta )}\|(u,w)-(u',w')\|^{\delta }
\end{equation}
for $c_\beta \|(u,w)-(u',w')\| <|u-w|$.

On the support of $\psi_l$, 
$$2^{l+k-1}\leq \|(u,v)\| \leq 2^{l+k+1}$$
and for $w \in \supp f$, $|w|<2^{k_0}$. Note that  if  $|u|<  (1+ |\gamma_2^{-1}\gamma_1|)^{-1} 2^{k+l-2}$ and $k\gg k_0$, then
$$\|(u,v)\| \leq |u| (1+ |\gamma_2^{-1}\gamma_1| )+ |\gamma_2^{-1}\gamma_3|2^{k_0} < 2^{k+l-1}.$$
Thus, in the integral representation
 of  $\Lambda(\psi_l\otimes f)$  we may assume 
$$2^{k+l} \lesssim |u| \leq 2^{k+l+1}.$$

Using the mean zero property of $f$ we obtain
$$
\Lambda(\psi_l\otimes f) = \int [\psi_l(u, v) \tilde{K}(u,w)-\psi_l(u,-\gamma_2^{-1}\gamma_1u) \tilde{K}(u,0)]f(w) \, dwdu.
$$
We now write the term in square
brackets as
$$
(\psi_l(u,v)-\psi_l(u,v_0))\tilde{K}(u,w)
+ \psi_l(u,v_0) (\tilde{K}(u,w)-\tilde{K}(u,0))
$$
and estimate each term by its supremum norm on the domain of
integration. Clearly, we have $|\psi_l|\le 2$.
As the derivative of $\psi_l$ is $O(2^{-k-l})$, we have
$$
|\psi_l(u,v)-\psi_l(u,v_0)| \lesssim 2^{-k-l}|w|\lesssim 2^{-k-l}2^{k_0}.
$$
For $\tilde{K}$ we have from the Calder\'on-Zygmund estimates 
$$|\tilde{K}(u,w)|\lesssim |u-w|^{-1}\lesssim (2^{k+l} -2^{k_0})^{-1}\lesssim 2^{-k-l}$$
 and
$$
|\tilde{K}(u,w)-\tilde{K}(u,0)|\lesssim |u|^{-(1+\delta )}|w|^{\delta }\lesssim 2^{(-k-l)(1+\delta)}2^{k_0\delta},
$$
Hence we can estimate
$$
|\Lambda(\psi_l\otimes f)|
\lesssim (2^{-2(k+l)} 2^{k_0}+ 2^{-(k+l)(1+\delta)}2^{k_0\delta})\int_{|u|<2^{k+l+1}}\int |f(w)| \, dudw$$$$
\lesssim  2^{-(k+l-k_0)\delta} \|f\|_{L^1}.
$$
Summing in $l$ finishes the proof of the lemma.
\end{proof}

\begin{lemma}\label{uniformBMO} Assume that $\Lambda$ satisfies the restricted boundedness conditions. Then the linear functional given by 
(\ref{L})  can be extended to all the Hardy space $H^1$ and
defines an element of BMO that we will denote by $T_3(1,1)$.
\end{lemma}

\begin{proof}
We will show first  that the functions $T_3(D^\infty_{2^k}\Phi, D^\infty_{2^k}\Phi)$ are uniformly bounded in BMO. 
To simplify the notation, let $\phi_k=D^\infty_{2^{k}}(\Phi\otimes\Phi)$, and 
let $c_\gamma$ be a large constant (depending only on $\gamma$) whose value will become clear later. We will show that the mean oscillation of $T_3(\phi_k)$ on some arbitrary interval $I$ is $O(1)$.

If $I$ is an interval  with $c_\gamma|I| \geq 2^{k}$, then by Lemma \ref{ctos}
$$\|T_3(\phi_k)\|_{L^2(I)}\leq \|T_3(\phi_k)\|_{L^2}\lesssim  2^{k/2}\lesssim |I|^{1/2},$$
which implies that the mean oscillation of $T_3(\phi_k)$ on $I$ is $O(1)$. 

Assume next that $I$ is an interval centered at the point $c(I)$ and such that $c_\gamma |I| < 2^{k}$. Let 
$M$ be the smallest integer so that  $[-2^k,2^k] \subseteq
 c_\gamma 2^M I$  and write 
$$
\phi_k= \sum_{l=0}^{M-1} \psi_{k,I,l},
$$
where 
$$\psi_{k,I,0}= (\tau_{(c(I),c(I))}D^\infty_{c_\gamma|I|}(\Phi \otimes \Phi ))\,
\phi_k$$
and 
$$\psi_{k,I,l}=( \tau_{(c(I),c(I))}D^\infty_{c_\gamma2^{l+1}|I|}(\Phi \otimes \Phi ) - 
\tau_{(c(I),c(I))}D^\infty_{c_\gamma2^{l}|I|}(\Phi \otimes \Phi))
\,
\phi_k.$$

By Lemma \ref{ctos} , for any function in $f$ in $L^2$ we can write (with a small abuse of notation)
$$
\langle T_3(\phi_k) | f\rangle = \Lambda(\phi_k \otimes f) =  \sum_{l=0}^{M-1} \Lambda (\psi_{k,I,l} \otimes f) =  \sum_{l=0}^{M-1} \langle T_3 (\psi_{k,I,l}) | f \rangle$$

Since $c_\gamma|I| \le 2^{k}$, 
 $|I|^{-1} \psi_{k,I,0}$ is $L^2$- adapted to $I\times I$
and we can estimate again using Lemma \ref{ctos},
$$\|T_3(\psi_{k,I,0})\|_{L^2(I)} \lesssim |I|^{1/2},$$ which implies that the mean oscillation of $T_3(\psi_{k,I,0})$ on $I$ is $O(1)$.
In the other terms $T_3 (\psi_{k,I,l})$ can be represented  on $I$ by absolutely convergent integrals
which, after the usual change of coordinates,  take the form
$$
T_3(\psi_{k,I,l})(w)= \int  \psi_{k,I,l}(u,v)\tilde{K}(u,w) \, du,
$$
with $v=-\gamma_2^{-1}(\gamma_1 u+\gamma_3w)$. 
Observing that $v-c(I) = -\gamma_2^{-1}(\gamma_1 ( u-c(I))+\gamma_3(w-c(I)))$,  we may proceed as in the previous lemma to verify that the above integral may be restricted to the region where 
$$c'_\gamma c_\gamma 2^l |I| \leq |u-c(I)| \leq c_\gamma 2^{l+1} |I|$$
for some small constant $c'_\gamma$ depending only on $\gamma$.
We take now $c_\gamma$ large enough so that 
$$|u-c(I)|\ge \max\{10\times 2^l,c_{\beta}\}|I|,$$
where $c_{\beta}$ is associated with $\tilde{K}$ as in Lemma \ref{definebmo}.

Define $v_0=-\gamma_2^{-1} (\gamma_1 u+\gamma_3 c(I))$ and the constant
$$C_{k,I,l}= \int \psi_{k,I,l}(u,v_0)\tilde{K}(u,c(I)) \, du$$
 and  write 
$$
|T_3(\psi_{k,I,l})(w) - C_{k,I,l} | \leq $$$$ \int  |\psi_{k,I,l}(u,v) - \psi_{k,I,l}(u,v_0)| \, |\tilde{K}(u,w)| \, du +
\int  |\psi_{k,I,l}(u,v_0)| \,| \tilde{K}(u,w) - \tilde{K}(u,c(I)|\, du. 
$$

Take now some $w\in I$. 
We can apply \eqref{CCZZ2} to obtain that the second integral above is bounded (up to a multiplicative constant) by
$$
\int_{|u-c(I)|> 10\times 2^l |I|} |I|^\delta |u-c(I)|^{-(1+\delta)}\,du \lesssim 2^{-l\delta}.$$ 

To estimate the first integral note that 
$$\|\frac{\partial\psi_{k,I,l}}{\partial v}\|_{\infty}\lesssim \frac{1}{\min(c_\gamma2^l |I|, 2^k)},\;\;|v-v_0|\lesssim|I|,\;\;|\tilde K(u,w)| \lesssim (2^l|I|)^{-1},$$
and  the region of integration has length $O(\min(c_\gamma2^l |I|, 2^k))$.
By putting these things together, the first integral above is easily seen to be $O(2^{-l})$. This proves that the mean oscillation of $T_3(\psi_{k,I,l})$ over the interval $I$ is $O(2^{-l\delta})$. Finally, by the triangle inequality we conclude that $\|T_3(\phi_k)\|_{BMO}\lesssim 1$.

To conclude the proof we can now use the $H^1-BMO$ duality. In fact, since the unit ball of the dual of 
a Banach space is weak$^*$-compact, we can extract a subsequence of $T_3(\phi_{k_j})$ so that
$\lim_{j \to \infty} \langle T_3(\phi_{k_j}), f \rangle = \langle T_3(1,1)|f \rangle$ 
for some $T_3(1,1)\in BMO$ and all $f\in H^1$. However, the previous lemma shows that the  sequence on the left converges to $L(f)$  for $C^\infty$ functions $f$ with compact support and mean zero. Since such functions are dense in $H^1$, it follows that the continuous functional induced by $T_3(1,1)$ on $H^1(\R)$ extends $L$ to all $H^1$  and that $T_3(1,1)$ is the unique limit (in BMO) of the sequence $T_3(\phi_{k})$.
\end{proof}

The following result shows that the definition of  $T_3(1,1)$ in Lemma \ref{definebmo} 
is independent of the choice of function $\Phi$ in a very general sense. 

\begin{lemma}
\label{altdefT1neededlater}
Let $f$ be a Schwartz function supported on $I = [-1/2,1/2]$  and with mean zero, and $k>0$.
Then, for every $L^\infty$-normalized bump function $\phi$ adapted to
the square
$[-2^k,2^k]\times [-2^k,2^k]$
with   $\phi \equiv 1$ in a neighborhood of $(0,0)$  
we have the estimate
$$|\langle T_3(1,1)|f\rangle-\Lambda(\phi,f)|\le C_{\gamma,\delta'}2^{-\delta' k} \|f\|_{L^1},$$
where we can take $\delta' = \delta$ if $\delta<1$ and $\delta' = 1^-$ if $\delta=1$.

Moreover, if $\Lambda$ satisfies the restricted boundedness conditions then the estimate above still holds with $\|f\|_1$ replaced by  $\|f\|_2$  (note $|I|^{1/2}= 1$), if the $L^\infty$-normalized bump function $\phi$ adapted to
the square
$[-2^k,2^k]\times [-2^k,2^k]$ satisfies only $\phi(0,0)=1$.  
\end{lemma}
\begin{proof}
Let $\Phi$ be as before.
Considering $\psi_k=D^\infty_{2^k} (\Phi\otimes\Phi) -\phi$, we need to prove
$$|\Lambda(\psi_k\otimes f)|\lesssim 2^{-\delta k}.$$
Let 
$$\Phi_{k\pm l}= D^\infty_{2^{k \pm  l +1}}(\Phi\otimes \Phi)-D^\infty_{2^{k \pm l}}(\Phi\otimes\Phi).
$$
For a large $k_0$ and  $k\gg k_0$, write
$$
\psi_k =  \sum_{l=0}^{\infty} \Psi_{k+l} + \sum_{l=1}^{k-k_0}  \Psi_{k-l} + \tilde \Psi_{k_0},
$$
where
$$\Psi_{k \pm l}= \Phi_{k\pm l}\psi_k =
(D^\infty_{2^{k \pm  l +1}}(\Phi\otimes \Phi)-D^\infty_{2^{k \pm l}}(\Phi\otimes\Phi)) (D^\infty_{2^k} (\Phi\otimes\Phi) -\phi)$$
and 
$$
\tilde \Psi_{k_0}=D^\infty_{2^{k_0}}(\Phi\otimes\Phi)\psi_k.$$

We can apply the reasoning of  Lemma~\ref{definebmo} to $\Psi_{k\pm l} $ to
write $\Lambda(\Psi_{k \pm l}\otimes f)$ as the sum of two integrals, one involving 
$(\nabla \Psi_{k \pm l})\tilde K$ and the other involving 
$\Psi_{k\pm l}(\nabla \tilde K)$, and where the integration in $u$ takes place for $|u|\approx 2^{k\pm l}$.
We then need to estimate $|\nabla \Psi_{k\pm l}|$ and $|\Psi_{k\pm l}|$ and combine them with the Calder\'on-Zygmund estimates on $\tilde K$ (which can be applied if $k_0$ is chosen large enough depending only on $\gamma$). 

For $\Psi_{k+l}$, we can use in the first integral that $\psi_k$ is adapted to a square of side length $2^k$
to get 
$$
|\nabla \Psi_{k+ l}| \leq |\nabla \Phi_{k+l}| |\psi_k| + |\Phi_{k+l}| |\nabla \psi_k|$$$$
\lesssim (2^{-(k+l)} \| \psi_k \|_{L^\infty} + \|\Phi_{k+l} \|_{L^\infty} 2^{-k} (1+ \frac{2^{k+l}}{2^k})^{-N}
\lesssim 2^{-(k+l)}
$$
and in the second integral we simply use that that $\Psi_{k+l}$ is bounded. Together with the Calder\'on-Zygmund estimates this leads to estimates on the integrals
of order $2^{-(k+l)}$ and $2^{-(k+l)\delta}$, respectively, which sum in $l$ to the correct bound.
On the other hand, since
$\phi(0,0)=1$, we can use  for $\Psi_{k-l}$ 
 in the first integral 
 $$
|\nabla \Psi_{k- l}| \leq |\nabla \Phi_{k-l}| |\psi_k| + |\Phi_{k-l}| |\nabla \psi_k|$$$$
\lesssim 2^{-(k-l)}  2^{-k} 2^{(k-l)}+ \|\Phi_{k-l} \|_{L^\infty} 2^{-k}
\lesssim 2^{-k}
$$
and
$$
|\Psi_{k-l}|\lesssim 2^{-k} 2^{(k-l)}\lesssim  2^{-l}$$ in the second integral, which leads to 
estimates of the order of  $2^{-k}$ and $2^{-k \delta} 2^{-l(1-\delta)}$, respectively. Noting that there are less than $k$ terms in  this case, we get to an estimate of the form $2^{-k\delta'}$ with $\delta'$ as in the statement.

Finally, note that the  last term,
$\Lambda(\tilde \Psi_{k_0}\otimes f)$ is zero if  $\phi \equiv 1$ in a neighborhood of $(0,0)$ 
or otherwise it can be controlled by the restricted boundedness property since 
$\tilde \Psi_{k_0}$ is adapted to a cube of side length $2^{k_0}$ with a constant of the order of 
$2^{-k}$.
\end{proof}

\begin{remark}\label{definet11}
For any fixed $y$ and $\Phi$ as in \eqref{Phi},
the function $\tau_{(y ,y)}D^\infty_{2^k} (\Phi \otimes \Phi)$ is an $L^{\infty }$-normalized bump function adapted to
$[-2^k, 2^k]\times [-2^k, 2^k]$ and equal to 1 in a neighborhood of the origin, for sufficiently large $k$. Hence we may also write
$$\Lambda (1,1,f)= \langle T_3(1,1)| f\rangle =\lim_{k\to \infty} \Lambda(\tau_{(y,y)} D^\infty_{2^{k}} (\Phi \otimes \Phi),f).$$
Moreover, an examination of the proof of Lemma~\ref{definebmo}, shows that if  the test function $f$ 
with mean zero satisfies $\supp f \subset I$, then 
$$| \Lambda (1,1,f) - \Lambda(\tau_{(c(I),c(I))} D^\infty_{2^{k}|I|} (\Phi \otimes \Phi),f)| \lesssim 2^{-k\delta} \|f\|_{L^1},$$
for large $k$ depending only on $\gamma$.\\
\end{remark}

We now introduce a weaker continuity property on $\Lambda$ alluded to before.

\begin{definition}
We say that a trilinear form $\Lambda$ satisfies the weak boundedness  property if for any interval $I$
and any $\phi$ that is $L^2$-normalized and adapted to $I\times I\times I$  of order $N$,
\begin{equation}\label{weakbounded}
|\Lambda(\phi)|\le C|I|^{-1/2}.
\end{equation}
\end{definition}

The restricted boundedness conditions with order $N'\gg N$ imply (\ref{weakbounded}) of order $N$. This is immediate for $\phi$ given by a tensor product, while arguments similar to the ones used
in Lemma~\ref{ctos} give the general case. 
We will prove Theorem~\ref{main} assuming only the weak boundedness  property and that the 
distributions  $T_j(1,1)$  given by  \eqref{L} (and its symmetric versions) are in BMO for $j=1,2,3$. Since, as we already saw, all these conditions are implied by the
restricted boundedness, it will follow then yet another formulation of the main theorem.

\begin{theorem}
Let $\Lambda$ be a Calder\'on-Zygmund trilinear form with modulation symmetry in the direction of $\gamma$.
Then $\Lambda$ is bounded in the range
of exponents of Theorem \ref{main} if and only if
$T_j(1,1)\in {\rm BMO}$
for each $j=1,2,3$, and $\Lambda$ satisfies the weak boundedness
property \eqref{weakbounded}.
\end{theorem}
While the Calder\'on-Zygmund condition does not distinguish between
Calder\'on-Zygmund kernels associated with bilinear or trilinear forms, the ${\rm BMO}$
conditions constitute a real difference between distributions
which are bounded bilinear forms and those which are bounded trilinear forms.
For trilinear forms, there are three {\rm BMO} conditions, while
for bilinear forms there are only two. In the section on paraproducts
we will see that the three ${\rm BMO}$ conditions are indeed independent,
and we can adjust the three functions $T_j(1,1)$ independently for each $j$.
In particular, we can take two of these functions  in ${\rm BMO}$ and a
third one not in ${\rm BMO}$ and construct a distribution that provides
a bounded bilinear form but not a bounded trilinear one.

We also observe the following.
\begin{remark}\label{definet11123}
Let $\Lambda$ be a trilinear Calder\'on-Zygmund form  with modulation symmetry in the direction $\gamma$  and satisfying  \eqref{weakbounded}. Then, the forms $\Lambda_j$, $j=1,2,3$, satisfy the  bilinear weak boundedness property of Remark~\ref{equivTcondbmo} and the identities 
$$\Lambda_1(1,\cdot) = \Lambda (1,1,\cdot) = \Lambda_2 (1,\cdot),$$
$$\Lambda_2(\cdot, 1) = \Lambda (\cdot , 1, 1) = \Lambda_3 (\cdot , 1),$$
$$\Lambda_3 (1,\cdot) = \Lambda (1,\cdot , 1) = \Lambda_1(\cdot, 1),$$
hold when the terms are interpreted as tempered distributions modulo constants.
\end{remark}

It is trivial to check the bilinear weak boundedness conditions for the $\Lambda_j$ given the one for $\Lambda$. The above identities are also clear at a formal level. To verify them in a rigorous way, let $f$ be  a test function with mean zero supported in some interval $I$ centered at the origin. For $C$ large enough depending just on $\gamma$ and for $k$ large enough compared to $\log |I|$ we have
$$\Lambda_1(D_{2^k}\Phi,f)=\Lambda(D_{2^k}D_C\Phi,D_{2^k}\Phi,f),$$
because $1\otimes D_{2^k}\Phi \otimes f = D_{2^k}D_C\Phi \otimes D_{2^k}\Phi \otimes f$
on $\gamma^\perp$.
As $k\to\infty$ the left hand side tends to $\Lambda_1(1,f)$ while, by Lemma \ref{altdefT1neededlater}, the right hand side approaches $\Lambda(1,1,f)$ since $D_C\Phi \otimes \Phi \equiv 1$ in a neighborhood of $(0,0)$. A symmetric reasoning gives the other identities.

\section{Bump functions estimates}\label{canccond1}

In this section we prove some estimates on bump functions 
under the additional assumption of $T_j(1,1)=0$ (in the BMO sense) for $j=1,2,3$.
Recall that $T_j(1,1)$ has been defined in Lemma~\ref{uniformBMO}.
We will study the action of $\Lambda$ on triples of bump functions
and obtain good estimates in terms of the localization  in space and
frequency of the bump functions. The estimates will be obtained from
a reduction to the bilinear form case, which we will discuss first.
Lemma~\ref{twobumplemma} below is a slight generalization of some almost orthogonality 
estimates found in the literature.

In what follows, we  will write ${\rm diam}(I_1,I_2)$ for the diameter ${\rm diam}(I_1\cup I_2)$ of the union of the two intervals. We will also use the fact that if $|I_1|>|I_2|$ then
$$1+\frac{{\rm diam}(I_1, I_2)}{|I_1|}\approx 1+\frac{|c(I_1)-c(I_2)|}{|I_1|}.$$

\begin{lemma}\label{molecules}
Let $\phi_J$ and $c_I$ be bump functions $L^2$- adapted of order $N$ to the intervals $J$ and $I$. We have
\begin{equation}\label{01}
  \int |\phi_J (x)|\,|c_I (x)|\, dx  \lesssim  \left(\min(\frac{|I|}{|J|},\frac{|J|}{|I|})\right)^{1/2}(1+(\max(|I|,|J|)^{-1}{\rm diam}(I,J))^{-N}
  \end{equation}
  while
\begin{equation}\label{02}
\left |\int \phi_J (x) c_I (x)\,  dx \right | \lesssim  \left( \frac{|I|}{|J|}\right)^{3/2}(1+|J|^{-1}{\rm diam}(I,J))^{-N},
\end{equation}
as long as
 $|J|\geq |I|$ and $\int c_I(x)  \,dx=0$.
\end{lemma}

Note that the estimate in (\ref{01}) only needs the size estimates on $\phi_J$ and $c_I$ but not the ones on their derivatives. 
Similar estimates are obtained when the bumps $\phi_J$ are replaced by the
ones obtained by the action of Calder\'on-Zygmund operators with certain cancellation.
More precisely one has.

\begin{lemma}\label{waveletstomolecules}
Let $\phi_J$ and $c_I$ be as in the previous lemma with $N\ge 2$.
Assume that $\Lambda$ is a bilinear Calder\'on-Zygmund form associated
with a kernel $K$ with regularity parameter $\delta$. Suppose that $\Lambda$ satisfies the restricted boundedness conditions
\eqref{linearrestricted} (or equivalently that is bounded on $L^2 \times L^2$) and that the linear dual operator\footnote{Recall that this is $\Lambda(f_1,f_2)=\int T(f_1)f_2$} $T$ has the special cancellation condition
$T (1)=0$. 
Then for each $0<\delta'<\delta$,
\begin{equation}\label{03}
\int |T(\phi_J) (x)| \, | c_I (x)|\,dx  \lesssim   \left(\min(\frac{|I|}{|J|},\frac{|J|}{|I|})\right)^{1/2} (1+(\max(|I|,|J|)^{-1}{\rm diam}(I,J))^{-(1+\delta')},
\end{equation} 
provided that $\int \phi_J(x)\, dx=0$, while
\begin{equation}\label{04}
\left |  \int T(\phi_J) (x) c_I (x)\, dx \right | \lesssim \left( \frac{|I|}{|J|}\right)^{1/2+\delta'}(1+|J|^{-1}{\rm diam}(I,J))^{-(1+\delta')}
\end{equation}
if $|J|\geq |I|$, $\int \phi_J(x)\,dx=0$, and $\int c_I(x)\,dx=0$.
\end{lemma}
For a proof of Lemma~\ref{molecules} and Lemma~\ref{waveletstomolecules} see, e.g.,  \cite{M} and \cite{FJ}.  We have now the following generalization.

\begin{lemma}\label{twobumplemma}
Assume again that $\Lambda$ is a bilinear form  associated
with a Calder\'on-Zygmund kernel $K$ with regularity parameter $\delta$. Suppose that $\Lambda$ is bounded  and that the dual linear operator $T$ has the special cancellation condition
$T (1)=0$. 

Let $|I_1|\ge |I_2|$ be intervals and let $\psi\in {\cal S}(\R^2)$ be a bump adapted to
$I_1\times I_2$ of order at least $4$. Assume $\psi$ has mean zero in the second variable, meaning
$$\int\psi(x_1,x_2)\, dx_2=0$$
for all $x_1$. Then for each $0<\delta'<\delta$,
\begin{equation}\label{twobump}
|\Lambda(\psi)|\lesssim \left(\frac{|I_2|}{|I_1|}\right)^{1/2+\delta'}
\left( 1+|I_1|^{-1}{\rm diam}(I_1,I_2) \right)^{-(1+\delta')}.
\end{equation}

\end{lemma}

\begin{proof}

Let $\psi(x_1,x_2)$ be adapted to $I_1 \times I_2$ with $|I_1|\ge |I_2|$ and having mean zero in the second variable. Using wavelets, we can expand it
in the first variable into a a family of bumps (wavelets) $\{\phi_J\} $ which are $L^2$- adapted to the dyadic intervals $J$  and have mean zero.  Furthermore, using the linearity of $\Lambda$  and the fact that $\Lambda$ extends to a continuous functional on $L^2(\R^2)$ (this being a consequence of the hypothesis and of Theorem \ref{classicalt1}), we obtain
$$|\Lambda (\psi)| \leq \sum_J|\Lambda (\langle \psi(\cdot , x_2),\phi_J\rangle\,\phi_J(x_1))| =
 \sum_J|\Lambda (\phi_J(x_1)\,c^J_{I_2}(x_2))| = \sum_J|\langle T\phi_J | c^J_{I_2}\rangle|,$$
where the $c^J_{I_2}$ has mean zero and is adapted to the interval $I_2$ but with implicit bounds
depending on the relative sizes of $I_2$ and $J$ that can be estimated using (\ref{01}) and (\ref{02}). We will use (\ref{03}) and 
(\ref{04}) to estimate $|\langle T\phi_J | c^J_{I_2}\rangle|$. We proceed as follows.

We split  the sum as
$$\sum_J = \sum_{J:\,  |J|<|I_2|\le|I_1|} + \sum_{J:\,  |I_2|\le|J|\le|I_1|} + \sum_{J:\, |I_2|\le |I_1|<|J|} =
\sum_A + \sum_B + \sum_C.$$
Using (\ref{02}) and (\ref{03}) we estimate
$$\sum_A \lesssim  \sum_{J:\, |J|<|I_2|\le|I_1|} \left( \frac{|J|}{|I_2|}\right)^{1/2}
\left( \frac{|J|}{|I_1|}\right)^{3/2} (1+\frac{{\rm diam}(J,I_2)}{|I_2|})^{-1-\delta'}
(1+\frac{{\rm diam}(J,I_1)}{|I_1|})^{-1-\delta'}$$
$$ \lesssim  \sum_{\nu:\, 2^{-\nu}<|I_2|\le|I_1|}
 \left( \frac{2^{-\nu}}{|I_2|}\right)^{1/2}
\left( \frac{2^{-\nu}}{|I_1|}\right)^{3/2}
|I_2|^{1/2}|I_1|^{1/2} 2^{\nu}  \left( \frac{|I_2|}{|I_1|}\right)^{1/2}
 (1+\frac{{\rm diam}(I_1,I_2)}{|I_1|})^{-1-\delta'},
 $$
where we have used a discrete version of estimate (\ref{01}) to sum on all intervals
$J$ of a fixed scale. That creates the normalization factors in the last line. Simplifying, we get
$$\sum_A \lesssim    \frac{|I_2|^{1/2}}{|I_1|^{3/2}}
 (1+\frac{{\rm diam}(I_1,I_2)}{|I_1|})^{-1-\delta'} \sum_{2^{-\nu}<|I_2|} 2^{-\nu}
 \lesssim   \frac{|I_2|^{3/2}}{|I_1|^{3/2}}
 (1+\frac{{\rm diam}(I_1,I_2)}{|I_1|})^{-1-\delta'}
 ,$$
which gives a better estimate than the desired one.

Similarly for $\sum_B$, we first use (\ref{01}) and  (\ref{03}) and then sum at each scale to get
the estimate in \eqref{twobump}.

To estimate $\sum_C$, we use  (\ref{01}) and  (\ref{04}) and  $\delta'<\delta''<\delta$ to compute

$$\sum_C \lesssim  \sum_{J:\, |I_2|\le|I_1|<|J|}
 \left( \frac{|I_2|}{|J|}\right)^{1/2+\delta''}
\left( \frac{|I_1|}{|J|}\right)^{1/2}
(1+\frac{{\rm diam}(J,I_2)}{|J|})^{-1-\delta''}
(1+\frac{{\rm diam}(J,I_1)}{|J|})^{-1-\delta''}
$$
$$
\lesssim  \sum_{\nu:\, |I_2|\le|I_1|<2^{-\nu}}
 \left( \frac{|I_2|}{2^{-\nu}}\right)^{1/2+\delta''}
\left( \frac{|I_1|}{2^{-\nu}}\right)^{1/2}
 (1+\frac{{\rm diam}(I_1,I_2)}{2^{-\nu}})^{-1-\delta''}.
$$
By further splitting the above summation according to the relative size of $2^{-\nu}$ with respect to ${\rm diam}(I_1,I_2)$, we get again the desired bound.
\end{proof}

Next, we formulate the corresponding estimates for trilinear forms. We will use the following notation.

\begin{definition}
For an interval $\omega=(a,b)$ and a constant $c>0$, let  $\widetilde{c\omega} = (ca,cb)$ (that is the dilation of $\omega$ from the origin, not from its center which we denote by $c\omega$).
\end{definition}

\begin{lemma}\label{threebumplemma}
Let
$\Lambda$ be a trilinear Calder\'on-Zygmund form associated with a  kernel $K$ with parameter $\delta$, with modulation symmetry in the direction of $\gamma $
with $\gamma_i\neq 0$,  and which satisfies the weak boundedness  condition \eqref{weakbounded}. Assume also that $\Lambda(1,1, \cdot)= \Lambda( 1, \cdot , 1) = \Lambda(1, 1, \cdot) = 0$.
Let $k_1=k_2\ge k_3$ be three integers. For each $1\le i\le 3$, let $\omega_i'$ be an interval of
length $2^{-k_i}$, and assume that  $\omega_3' \cap \widetilde{\frac{\gamma_3}{\gamma_1}\omega_1'}=\emptyset$ and $\omega_3' \cap \widetilde{\frac{\gamma_3}{\gamma_2}\omega_2'}=\emptyset$.

Let $I_i$ be an interval of length $2^{k_i}$ and assume
$\phi  \in {\cal S}(\R^3)$ is a bump function of order $2N$
adapted to $I_1\times I_2\times I_3$ and with Fourier transform supported in the box
$\omega':=\omega_1'\times \omega_2'\times \omega_3'$.
Then, for each $0<\delta' <\delta $ we have
\begin{equation}\label{threebump}
|\Lambda(\phi)|\lesssim {|I|^{-1/2}}
\left(\frac{|I_3|}{|I|}\right)^{\delta'+1/2}
\left( 1+\frac{{\rm diam}(I_1,I_2,I_3)}{|I|} \right)^{-1-\delta'}
\left( 1+\frac{\big| \sum_{i=1}^3  \gamma_i c(I_i)\big|}{|I|}  \right)^{-N},
\end{equation}
where we have written $\diam (I_1,I_2,I_3)$ for $\diam(\bigcup_{i=1}^{3}I_i)$
and $|I|=|I_1|=|I_2|$.

Similar statements and estimates holds by symmetry for any permutation of the indices $1,2,3$. 
\end{lemma}

\begin{proof}
We assume without loss of generality that ${\rm diam}(I_1,I_2,I_3)$
is comparable to ${\rm diam}(I_2,I_3)$. This can be achieved by
switching the otherwise symmetric roles of $I_1$ and $I_2$ if necessary. We write
$$\psi(x+\beta_2t,x+\beta_3t)=\phi(x+\beta_1t,x+\beta_2t,x+\beta_3t), $$
and note that $\Lambda_1(\psi)=\Lambda(\phi)$.  We plan to apply Lemma \ref{twobumplemma} to $\Lambda_1$ and $\psi$.

First observe that by hypothesis and Remark \ref{definet11123}, $\Lambda_1$ is bounded and satisfies the requirements of Lemma \ref{twobumplemma}. Second, we claim the function $\psi$ is adapted to $I_2\times I_3$ with constant
$|I|^{-1/2}M^{-N}$, where
$$5M= 1+|I|^{-1} \left | \sum \gamma_i c(I_i)\right |.$$
Clearly, it is enough to see this when  $|I|^{-1} \left | \sum \gamma_i c(I_i)\right |  \gg  1$.  We truncate
$$\psi=\psi_1+\psi_2$$
where
$$\psi_1(x_1,x_2)= \psi(x_1,x_2) \Phi_{MI_2}(x_1)\Phi_{MI_3}(x_2)$$
and the functions $\Phi_{MI}$ are the usual $L^\infty$-normalized functions adapted to $MI$.

If $(x_1,x_2):=(x+\beta_2t,x+\beta_3t)$ is in the support of $\psi_1$, then
$|x+\beta_1t-c(I_1)|/|I|\gtrsim M$ because
$$|\gamma_1(x+\beta_1t-c(I_1))|
\geq \left |\sum_{j=1}^3 \gamma_j(x+\beta_jt-c(I_j)) \right | -  4M|I|
=\left |\sum_{j=1}^3 \gamma_j c(I_j) \right | - 4M|I|.$$
By computing derivatives, the claim for $\psi_1$ follows easily. Let us only indicate the bound for
the function itself:
$$
|\psi_1(x_1,x_2)|
=|\phi(x+\beta_1t,x+\beta_2t,x+\beta_3t)|\Phi_{MI_2}(x_2)\Phi_{MI_3}(x_3)
$$
$$
\lesssim |I|^{-1/2}|I_2|^{-1/2}|I_3|^{-1/2}M^{-N}\times
$$
$$
(1+|I|^{-1}|x+\beta_2t-c(I_2)|)^{-N}(1+|I|^{-1}|x+\beta_3t-c(I_3)|)^{-N}.
$$
If $(x+\beta_2t,x+\beta_3t)$ is in the support of $\psi_2$, then
either $|x+\beta_2t- c(I_2)|/|I|$ is of order at least $M$ or
$|x+\beta_3t- c(I_3)|/|I|$ is of order at least $M$,
and the claim for $\psi_2$ is again easy to see.

Finally, it remains to show that $\psi$ has mean zero in the second variable.
Consider first the particular case when  $$\widehat \phi(\eta_1, \eta_2, \eta_3) = \widehat \phi_1(\eta_1)\widehat \phi_2(\eta_2)\widehat \phi_3(\eta_3).$$
A simple change of coordinates 
gives then
$$\psi(u,v)=\phi_1(-\frac{\gamma_2}{\gamma_1}u-\frac{\gamma_3}{\gamma_1}v)\phi_2(u)\phi_3(v).$$
For each $u$, the Fourier transform of $\psi$ in the second variable is supported in 
$\omega_3'-\widetilde{\frac{\gamma_3}{\gamma_1}\omega_1'}$. Since this interval does not contain the origin, the claim follows. 

To obtain the general case expand $\widehat \phi$ as a series of products
of the above, via windowed Fourier series.
\end{proof}

\begin{remark}
\label{equivcond11}
Let $\langle \gamma\rangle$ be the line spanned by $\gamma$, and $\langle\gamma^{(i)}\rangle$ be the projection of the line $\langle\gamma\rangle$ onto the plane $e_i^{\perp}$.
We observe that the conditions on $\omega'$ in the previous lemma are actually equivalent to
\begin{equation}
\label{altpropr1}
\omega'_1 \times \omega'_3 \cap\langle \gamma^{(2)}\rangle=\emptyset,
\end{equation}
\begin{equation}
\label{altpropr2}
\omega'_2 \times \omega'_3 \cap\langle \gamma^{(1)}\rangle=\emptyset,
\end{equation}

In all applications of Lemma \ref{threebumplemma}, the box $\omega'$ will emerge in the following way. We will initially have a box $\omega:=(\omega_1,\omega_2,\omega_3)$ which satisfies 
\eqref{altpropr1} and \eqref{altpropr2}.
We will then choose some appropriate vector $\xi\gamma\in \langle \gamma\rangle$ and define $\omega':=\omega-\xi\gamma$. Regardless of the choice of $\xi$, \eqref{altpropr1} and \eqref{altpropr2} imply for $i\in\{1,2\}$ that $\omega_3'\cap \widetilde{\frac{\gamma_3}{\gamma_i}\omega_i'}=\emptyset$, since $\widetilde{\frac{\gamma_3}{\gamma_i}\omega_i'}=\widetilde{\frac{\gamma_3}{\gamma_i}\omega_i}-\xi\gamma_3$ and $\omega_3'=\omega_3-\xi\gamma_3$.
\end{remark}

\section{Decomposition of $\Lambda$ under the special cancellation conditions.}\label{canccond}

In this section we will express the form $\Lambda$ as a superposition of well localized model operators. Before achieving this goal we first 
 recall how the decomposition was performed in the particular case of the bilinear Hilbert transform, see \cite{LT1}. More generally, assume $K(t)$ is an $x$ independent Calder\'on-Zygmund kernel with enough decay on the derivatives of $\widehat{K}$, and consider a $\Lambda$ which for each $f_1,f_2,f_3\in\S(\R)$ has the representation
$$\Lambda(f_1,f_2,f_3)={\rm p.v}\int f_1(x+\beta_1t)f_2(x+\beta_2t)f_3(x+\beta_3t)K(t)dx\,dt.$$
One first decomposes  $K$ in pieces $K_k$ localized at frequency $\approx 2^{-k}$. Then, the modulation symmetry of $\Lambda$ recommends a wave packet decomposition of the three functions, adapted to the scale of the multiplier. 
The geometry of the Fourier plane in conjunction with the decay in the derivatives of $\widehat{K}$ allows then one to reduce the boundedness of  $|\Lambda(f_1,f_2,f_3)|$ to that of model operators of the form 
$$\sum_{ p \in \P} \, \, a_p \, |I_p|^{-1/2}|\langle f_1,\phi_{p_1}\rangle\langle f_2,\phi_{p_2}\rangle\langle f_3,\phi_{p_3}\rangle|,$$
where  $\P$ is the collection of all multi-tiles in phase space 
$p=p_1\times p_2\times p_3,$ \, $p_i\,=\, I_{p_i}\times\omega_{p_i}$ with $\,|I_{p}|:=|I_{p_1}|=|I_{p_2}|=|I_{p_3}| $ and  $\omega_{p_i}$ pairwise disjoint for each fixed $p$.  The  $\phi_{p_i}$ are wave packets $L^2$- adapted to the time-frequency tile $p_i$  and  $\{a_p\} \in \ell^{\infty}(\P)$.

A fundamental feature of  $\P$ is that it is a one parameter family of multi-tiles, in that each $\omega_{p_i}$ determines uniquely  the other two $\omega_{p_j}$. Another important aspect about $\P$ is the so called `quartile property'; namely the fact that if for two multi-tiles $p$ and $p'$ we have $\omega_{p_i}\cap\omega_{p_i'}\not=0$ then we are guaranteed that $\omega_{p_j}\cap\omega_{p_j'}=0$ for each $j\not=i$. Both of these properties follow as a consequence of the fact that the cubes $\omega_{p_1}\times\omega_{p_2}\times\omega_{p_3}$ are located at some uniform distance from the line $\langle\gamma\rangle$ and they touch the plane $(1,1,1)^{\perp}.$ 

In the case $K$ depends on both $x$ and $t$ we proceed differently.   Consider again a Calder\'on-Zygmund trilinear form $\Lambda$ which is modulation invariant in the direction $\gamma$. We saw that 
we can write 
$$\Lambda (f_1,f_2,f_3) = \langle K(x,t) | f_1 \otimes f_2 \otimes f_3 (x \alpha  + t \beta)\rangle$$
where we still denote by $K$ a distribution in ${\cal S}' (\R^2)$, which agrees with the given kernel
when $t\neq 0$. Note that the Calder\'on-Zygmund conditions on the kernel do not say anything 
about the distribution $K$ for $t=0$. As in the linear case, the weak boundedness property and the conditions $T_j(1,1) \in {\rm BMO}$ are needed to complete, in a certain sense, the control on the distribution $K$. 

To study $\Lambda$ we will perform a particular Whitney decomposition on the frequency domain.
We want to give a heuristic motivation for it.  Based on the experience with multipliers (both linear and multilinear ones),  one  expects that a Mihlin-type of behavior for the frequency representation 
of  $\Lambda$ should play an important role.  That is, the form should be given in the frequency side  by a distribution whose derivatives behave like  the reciprocal of the distance to a particular singular or bad set. Typically one then performs a Whitney decomposition with respect to that set.  

Applying the Fourier inversion formula to each $f_j$ we see that, in the sense of distributions, we can represent 
$\Lambda$ on the frequency domain as
 \begin{equation} \Lambda ( f_1,  f_2, f_3) = \langle \widehat K(-\alpha\cdot \xi, -\beta \cdot \xi) | \widehat f_1 \otimes \widehat  f_2 \otimes \widehat f_3 (\xi)\rangle,
\label{frequencyside0}
\end{equation}
where $\xi=(\xi_1,\xi_2,\xi_3)$. The conditions on the kernel are too weak to conclude any pointwise
kind of behavior for $\widehat K$. This will substantially complicate our analysis but, intuitively,
$\beta^\perp$ should be a bad set for $\widehat \Lambda$.  To further motivate the analysis to be performed,
consider a very particular case of $K$ with compact support, smooth in $x$, and satisfying for $t\neq 0$ the stronger conditions
 $$|\partial^{\alpha} K(x,t)| \lesssim |t|^{-(1+|\alpha|)},$$
for all $|\alpha|\geq 0$. These conditions (together with the weak boundedness property) imply 
 \begin{equation}\label{mihlin}
 |\partial_v^m\widehat K(u,v)|\lesssim |v|^{-m},
 \end{equation}
 for $v\neq 0$ and $m\geq 1$.
The  estimates  say that  the derivatives $\partial_v^m \widehat K$ are only singular at the origin 
(though they still do not say anything about $\widehat K$ itself). This and the representation \eqref{frequencyside0} suggest that $\widehat \Lambda$ may have some singularities on $\beta^\perp$. 

 The representation \eqref{frequencyside0} 
 is not unique in the sense that it depends on $\beta$.  In fact, for duality purposes and to exploit the conditions
 on the operators $T_j$, we can also write by simple changes of variables
 $$\Lambda (f_1,f_2,f_3) = \langle K_3(x,t) | f_1 \otimes f_2 \otimes f_3 (x \alpha  + t \beta^3)\rangle
 =\langle T_3(f_1,f_2)|f_3\rangle$$
 $$\Lambda (f_1,f_2,f_3) = \langle K_2(x,t) | f_1 \otimes f_2 \otimes f_3 (x \alpha  + t \beta^2)\rangle
 =\langle T_2(f_1,f_3)|f_2\rangle$$
  $$\Lambda (f_1,f_2,f_3) = \langle K_1(x,t) | f_1 \otimes f_2 \otimes f_3 (x \alpha  + t \beta^1)\rangle
 =\langle T_1(f_2,f_3)|f_1\rangle,$$
 where  the vectors $\beta^j$ are still perpendicular to $\gamma$ and satisfy that  the component $\beta^j_j$ of them is zero. The $K_j$ are also related to $K$ by a change of variable and still satisfy the same conditions assumed on $K$.
 In other words,
 $$T_3(f_1,f_2)(x)=\int f_1(x+\beta^3_1t) f_2(x+\beta^3_2t) K_3(x,t)\,dt$$
 and similarly for $T_1$ and $T_2$.
 
 We obtain then the three frequency representations
  \begin{equation} \Lambda ( f_1,  f_2, f_3) = \langle \widehat K_j(-\alpha\cdot \xi, -\beta^j \cdot \xi) | \widehat f_1 \otimes \widehat  f_2 \otimes \widehat f_3 (\xi)\rangle.
\label{frequencyside}
\end{equation}
In each of them, away from a plane through $\gamma$  defined by
$$P_j:=(\beta^j)^\perp =  {\rm span}(e_j,\gamma),$$ 
the form $\Lambda$ is given by  a symbol whose  derivatives blow up according
to \eqref{mihlin} when $\xi$ approaches $P_j$.  
 We use a Whitney decomposition that simultaneously resolves all of  the three singular sets $P_j$ independently of which $T_j$ we are using. 
 In a way, the conditions $T_j(1,1) \in {\rm BMO}$ 
 are needed to control  the behavior on the bad set  
 $$S:=\bigcup_{i=1}^3 P_j$$
 and eliminate the potential singularities of $\widehat \Lambda$ on it. 
 
 We see from
 \eqref{frequencyside} that, at least formally,  $T_j(1,1)=0$ translates into
 $$
 \langle \widehat K_j(-\xi_j, 0) | \widehat f_j (\xi_j)\rangle = 0,
 $$
and hence $\widehat K_j(u, 0) =0$, so  $\widehat \Lambda$ vanishes in some sense on $S$.  We will show rigorously that we can perform our analysis in $\R^3\setminus S$ in Lemma~\ref{whitney} below. 
 
 In the case of the bilinear Hilbert transform (or a kernel that is $x$-independent),
 the formula in \eqref{frequencyside0} takes the simpler form
   \begin{equation} \Lambda ( f_1, f_2,  f_3) = \langle  \delta(\alpha\cdot \xi)\widehat K(-\beta \cdot \xi) | \widehat f_1 \otimes \widehat  f_2 \otimes \widehat f_3 (\xi)\rangle,
\end{equation}
which clearly vanishes if $\widehat f_1 \otimes \widehat  f_2 \otimes \widehat f_3$ is supported away
from $\alpha^\perp$ (which also contains $\langle \gamma \rangle$). We see then that $\widehat \Lambda$
is  supported on $\alpha^\perp$ and hence is possibly  singular only  on the line $\langle \gamma \rangle$ independently of the representation used. A one-parameter family of boxes,
i.e. cubes, is then used in this case as mentioned above. In the $x$-dependent case, however,  a two-parameter family of boxes
in $\R^3$ will be used to decompose the complement of the bad set as we will now describe.

We  first produce a Whitney decomposition of the frequency domain $\R^3\setminus S$ into tubes as defined below. This is achieved by combining two-dimensional Whitney decompositions of each $e_j^{\perp}\setminus\langle\gamma^{(j)}\rangle$. 
In a second stage we perform wave packet decompositions of the three functions adapted to such tubes. Finally,  information on the decay of the coefficients associated with various localized pieces will  be provided via the almost orthogonality estimates of the previous section, combining the kernel representation with the weak boundedness condition and the cancellations $T_j(1,1)=0$. 

\begin{definition}
A dyadic box is a parallelepiped $\omega_1\times \omega_2\times \omega_3$
such that all intervals $\omega_i$ -called the sides of the box- are dyadic intervals.  A tube
is a dyadic box where the minimal side-length is attained by
two of the sides. We also allow the maximal side to be all of $\R$. The orientation of a dyadic tube is the direction corresponding to its longest side (cubes have no orientation).
\end{definition}
Observe that a nonempty intersection of two dyadic tubes of possibly different orientations is again a dyadic tube.

Let $C_1\gg  1$ be a sufficiently large constant depending on $\gamma$, whose value will not be specified. Constraints on how large $C_1$ should be will become apparent throughout the paper. The other important constants\footnote{From now all occurrences of $C_1$, $C_2$ and $C_3$ will refer only to this constants} that will appear throughout this work are $C_2:=C_1^{1/2}$ and $C_3:=C_1^2$. 

For each
$j\in \{1,2,3\}$, we decompose the complement of the plane $P_j$ into the collection $\Omega_j$ of
minimal dyadic tubes $\omega$ such that $C_1\omega $  intersects the plane $P_j$. An important property that will be used repeatedly is that for each such tube we have $\frac{C_1}{3}\omega \cap \langle \gamma\rangle=\emptyset$.
These tubes partition the complement of $P_j$, they are infinitely long in direction $e_j$ and
their projection onto
the orthogonal plane to $e_j$ defines a standard Whitney decomposition into squares of $e_j^{\perp}\setminus \langle\gamma^{(j)}\rangle$.

Now let $\Omega $ be the collection of all tubes which are nonempty intersections of three tubes as above, one in each of $\Omega_1$, $\Omega_2$, and $\Omega_3$.
Then $\Omega$ partitions the complement of the bad set. Each of the
tubes $\omega$ in this partition, initially defined as the intersection
of three tubes, is actually determined by the intersection of two of the three
tubes, one defining the two shorter sides of $\omega$ and the other one defining the long side of $\omega$.

For each Whitney tube $\omega \in \Omega $ we consider the box
$3\omega$, which is the box $\omega$ dilated about its center by
a factor $3$.  By a standard argument for Whitney decompositions (e.g.
applied to the Whitney decompositions of each of the three planes)
these boxes have bounded overlap and two overlapping boxes have
comparable side-length in each dimension. Therefore we can find a
partition of unity of the complement of $S$
$$1_{S^c}=\sum_{\omega \in \Omega }\widehat{\phi}_\omega$$
where each $\widehat{\phi}_\omega$ is supported in $3\omega$ and $L^\infty$-adapted to $\omega$.

\begin{definition}
The width of a tube is the length of each of the
shorter sides, its length is the length of its longest side, and its eccentricity is the ratio between its width and length.
\end{definition}

\begin{lemma}\label{whitney}

Assume $\Lambda$ is a trilinear Calderón-Zygmund form with modulation
symmetry in direction $\gamma$ that satisfies the weak boundedness property \eqref{weakbounded} and the special cancellation
conditions $T_i(1,1)=0$ for $1\le i\le 3$.

Let $\Omega^{(k)}$ be the set of all tubes in $\Omega$ with
width at least $2^{-k}$.
Then we have
\begin{equation}\label{whitneysum}
\Lambda(\psi)
=\lim_{k\to \infty} \sum_{ \omega\in \Omega^{(k)}}\Lambda(\psi*\phi_\omega)
\end{equation}
for each Schwartz function $\psi$ with compactly supported Fourier transform.
\end{lemma}

\begin{proof}
We will assume without loss of generality that $\widehat{\psi}$ is supported into the cube $[-1,1]^3$.
For each $k\in \Z$
we construct a set $\tilde{\Omega}^{(k)}$ of tubes of
width $2^{-k}$ such that
the tubes in  $\Omega^{(k)}\cup \tilde{\Omega}^{(k)}$ form a partition of $\R^3$,
their dilates by a factor of 3 have bounded overlap, and which have the property that if $3\omega\cap 3\omega'\not=0$ for some $\omega,\omega'\in\Omega^{(k)}\cup \tilde{\Omega}^{(k)}$ then $\omega$ and $\omega'$ have comparable side-length in every dimension. To achieve this, let $\Omega_{j,k}$ be the subset of $\Omega_j$ consisting of all the tubes of width at least $2^{-k}$. Define also $\tilde{\Omega}_{j,k}$ to be the (uniquely determined) collection of tubes infinitely long in the direction $e_j$ with width $2^{-k}$ such that $\Omega_{j,k}\cup \tilde{\Omega}_{j,k}$ forms a partition of $\R^3$. 

Define now $\tilde{\Omega}^{(k)}$ as the collection of all tubes that arise by intersecting 3 tubes, one in each of $\Omega_{1,k}\cup\tilde{\Omega}_{1,k}$, $\Omega_{2,k}\cup\tilde{\Omega}_{2,k}$ and $\Omega_{3,k}\cup\tilde{\Omega}_{3,k}$, with at least  one of the three tubes in some $\tilde{\Omega}_{i,k}$. Note that ${\Omega}^{(k)}$ consists of all tubes that arise by intersecting 3 tubes, one in each of $\Omega_{1,k}$, $\Omega_{2,k}$ and $\Omega_{3,k}$. It is now an easy exercise to prove that $\Omega^{(k)}\cup \tilde{\Omega}^{(k)}$ is a partition of $\R^3$ that has all the desired properties.  

Using these properties and the standard process of partition of unity, we may define functions $\widehat{\phi}_\omega$
for $\omega\in \tilde{\Omega}^{(k)}$ that are $L^\infty$-adapted to $\omega$ and supported in $3\omega$  so that
$$1=\sum_{\omega \in \Omega^{(k)}} \widehat{\phi}_\omega+
\sum_{\omega\in \tilde{\Omega}^{(k)}}\widehat{\phi}_\omega.$$
Recall that the functions $\phi_\omega$ with $\omega \in \Omega^{(k)}$
have been defined earlier.

It is not hard to observe (see also the rank properties in section \ref{sec6.1}) that 
\begin{equation}
\label{e.e7int}
C_3\omega\cap \langle\gamma\rangle\not=\emptyset
\end{equation} 
for each $\omega\in\Omega^{(k)}\cup\tilde{\Omega}^{(k)}$.
There are only finitely many tubes in
$$\Omega^{(k)}\cup \tilde{\Omega}^{(k)}$$ which intersect the compact support of $\widehat{\psi}$, hence clearly
\begin{equation}\label{exctruncation}
\Lambda(\psi)-
\sum_{\omega\in \Omega^{(k)}}\Lambda(\psi*\phi_\omega)
=\sum_{\omega\in \tilde{\Omega}^{(k)}}\Lambda(\psi*\phi_\omega)
\end{equation}
It then suffices to show that the right hand side tends to $0$ as $k$ tends to $\infty$.  For the rest of the proof it suffices to restrict attention to tubes that intersect the support of $\widehat{\psi}$  and to sufficiently large values of $k$. (In particular, we can assume $k\ge 0$)

We first estimate the contribution coming from the collection $\tilde{\Omega}^{(k,1)}$ of tubes in $\tilde{\Omega}^{(k)}$ whose
eccentricity is $1$ (the cubes). Due to \eqref{e.e7int} there are $O(2^{k})$ tubes in $\tilde{\Omega}^{(k,1)}$.
To estimate the contribution of a tube in $\tilde{\Omega}^{(k,1)}$ to
\eqref{exctruncation}
we use the modulation symmetry of $\Lambda$ to get that $\Lambda(\psi*\phi_\omega)=\Lambda(M_{\gamma \xi}(\psi*\phi_\omega))$, where $\xi\in\R$ is chosen in such a way that the support of the Fourier transform of $M_{\gamma \xi}(\psi*\phi_\omega)$ lies inside the cube centered at the origin with side-length $100C_32^{-k}$. This is possible due to \eqref{e.e7int}. It follows now easily that the Fourier transform of $M_{\gamma \xi}\phi_\omega$ is $L^{\infty}$- adapted to the cube of sidelength $2^{-k}$ centered at the origin. Since $k\ge 0$, the same can be said about the Fourier transform of $M_{\gamma \xi}(\psi*\phi_\omega)$. This easily implies now that $M_{\gamma \xi}(\psi*\phi_\omega)$ is $L^{1}$- adapted to the cube of sidelength $2^{k}$ centered at the origin.
Applying inequality  \eqref{weakbounded} we obtain that
$$|\Lambda(M_{\gamma \xi}(\psi*\phi_\omega))|\lesssim 2^{-2k}.$$
Thus, the contribution of $\tilde{\Omega}^{(k,1)}$ to \eqref{exctruncation} is $O(2^{-k})$.

It remains to consider the set $\tilde{\Omega}^{(k,2)}$ of tubes $\omega\in\tilde{\Omega}^{(k)}$ whose
eccentricity $2^{-\kappa}$ is smaller than 1. Without loss of generality we can assume that the sides of $\omega$ have lengths $2^{-k}$, $2^{-k}$ and $2^{-k+\kappa}$ in this order. It follows that $\omega$ is determined by the intersection of 3 special tubes $\omega':=\R\times\omega_{1,1}\times\omega_{1,2}\in \Omega_{1,k}\subset \Omega_1$, $\omega'':=\omega_{2,1}\times \R\times\omega_{2,2}\in \Omega_{2,k}\subset \Omega_2$ and $\omega''':=\omega_{3,1}\times\omega_{3,2}\times \R\in \tilde{\Omega}_{3,k}$.

First observe that if such a tube $\omega$ produces a nonzero contribution to our sum then its sides have lengths smaller than 1. This is immediate for the smaller sides whose length is $2^{-k}$. Let us now see that the same thing is true for the longer side. We observe that $\omega$ must intersect the cube $[-1,1]^3$, which implies $(\omega_{2,1}\times\omega_{2,2})\cap [-1,1]^2\not=\emptyset$. But since $\omega''\in \subset \Omega_2$, we know that $0\notin 10(\omega_{2,1}\times\omega_{2,2})$. This proves $|\omega_{2,2}|\le 1$. As a consequence, we deduce that the Fourier transform of $\psi*\phi_\omega$ is $L^{\infty}$- adapted to the tube $\omega$.  

Choose now $\xi\in\R$ such that $0\in C_3(\omega-\xi\gamma )$ (this is possible due to \eqref{e.e7int}). Using again the modulation invariance of $\Lambda$ and \eqref{e.e7int}, we get as before that 
$$|\Lambda(\psi*\phi_\omega)|=|\Lambda(M_{\xi\gamma}(\psi*\phi_\omega))|,$$
where $M_{\gamma \xi}(\psi*\phi_\omega)$ is $L^1$- adapted to the box centered at the origin with side-lengths comparable to $(2^k, 2^k, 2^{-\kappa}2^k)$. Moreover, the Fourier transform of the function $M_{\xi\gamma}(\psi*\phi_\omega)$ will be supported in the box $\omega-\gamma \xi$ which is easily seen to satisfy the requirements of Lemma \ref{threebumplemma}, once we prove that $\omega$ satifies the requirements \eqref{altpropr1} and \eqref{altpropr2} in Remark \ref{equivcond11}. But this is immediate since $(\omega_{2,1}\times\omega_{2,2})\cap \langle\gamma^{(2)}\rangle=\emptyset$ and $(\omega_{3,1}\times\omega_{3,2})\cap \langle\gamma^{(3)}\rangle=\emptyset$.

Lemma \ref{threebumplemma} now gives 
$$|\Lambda(M_{\xi\gamma}(\psi*\phi_\omega))|\lesssim 2^{-k/2} 2^{-\kappa(1/2+\delta')} 2^{-3k/2+\kappa/2}$$
where the  factor $2^{-3k/2+\kappa/2}$ adjusts the $L^2$ normalization of $\psi*\phi_\omega$. Finally note that for each $\kappa\ge 0$ there are $O(2^k)$ tubes $\omega$ as above. To see this, we can further assume without loss of generality that $\omega$ is determined by $\omega''$ and $\omega'''$, that is $\omega=(\omega_{3,1},\omega_{3,2},\omega_{2,2})$. Note first that there are $O(2^k)$ tubes $\omega'''=\omega_{3,1}\times\omega_{3,2}\times \R$ of width $2^{-k}$ which intersect $[-1,1]^2$,
since $C_1(\omega_{3,1}\times\omega_{3,2})\cap \langle\gamma^{(3)}\rangle\not=\emptyset$. Given $\omega_{3,1}$, we know that $\omega_{2,1}$ is determined uniquely, and then $\omega_{2,2}$ is determined within finitely many choices by $\omega_{2,1}$, since $C_1(\omega_{2,1}\times\omega_{2,2})\cap \langle\gamma^{(2)}\rangle\not=\emptyset$.
Hence we can estimate
$$\sum_{\omega\in \tilde{\Omega}^{(k,2)}} |\Lambda(\psi*\phi_\omega)|\lesssim
\sum_{\kappa\ge 0}  2^{-k}2^{-\kappa\delta'},$$
which is again an acceptable contribution. 

\end{proof}

As we proceed with the second stage of the decomposition of $\Lambda$, we return to viewing $\Lambda$ as a trilinear form on the triple
product of Schwartz spaces, rather than a distribution on $\R^3$.
Thus we assume $\psi$ in Lemma \ref{whitneysum} is an elementary
tensor $\psi=f_1 \otimes f_2 \otimes f_3$ of three compactly supported
smooth functions. To turn $\phi_\omega$
into a convergent sum of elementary tensors, we invoke Fourier
series.

For each tube $\omega\in \Omega$ we choose functions
 $\widehat{\phi}_{\omega_i}$ for $i=1,2,3$,
$L^\infty$- adapted to $\omega_i$, constant equal to $1$ on $3\omega_i$ and supported on
$5\omega_i$. The dilated tubes $5\omega$ are still disjoint from the bad set $S$
since $C_1\gg  1$. Then
$$\widehat{\phi}_\omega(\xi )=\widehat{\phi}_\omega (\xi )\prod_{i=1}^3
\widehat{\phi}_{\omega_i}(\xi_i)$$
Applying Fourier series on $5\omega $ gives
$$\widehat{\phi}_\omega(\xi )=\sum_{n_1,n_2,n_3\in \Z} c_{\omega,n_1,n_2,n_3}
\prod_{i=1}^3 \widehat{\phi}_{\omega_i}(\xi_i) e^{2\pi i n_i \xi_i /(5|\omega_i|)},$$
and note that for each $M$ the coefficients decay as
$$|c_{\omega,n_1,n_2,n_3}|\lesssim_{M} (1+\max(|n_1|,|n_2|,|n_3|))^{-M}.$$
We also note that since $\widehat{\phi }_{\omega_i}$ is $L^{\infty}$- adapted to $\omega_i$ of any order $M$ and supported in $5\omega_i$, so is the function $M_{n_i/5|\omega_i|}\widehat{\phi}_{\omega_i}$, with a constant that is $O(n_i^M)$. These observations imply that
$$
\Lambda(\psi*\phi_\omega)=\sum_{n_1,n_2,n_3\in \Z} c_{\omega,n_1,n_2,n_3}
\Lambda (f_1*\tau_{n_1/5|\omega_1|}\phi_{\omega_1},f_2*\tau_{n_2/5|\omega_2|}\phi_{\omega_2},
f_3*\tau_{n_3/5|\omega_3|}\phi_{\omega_3}).
$$
Another immediate implication is that it suffices to bound
$$
\sum_{\omega \in \Omega}
|\Lambda(
f_1* \phi_{\omega_1},
f_2* \phi_{\omega_2},
f_3* \phi_{\omega_3}
)|
$$
uniformly over all functions $\phi_{\omega_i}$ such that $\widehat{\phi}_{\omega_i}$ is $L^\infty$- adapted of order -say- $2N$ and supported in  $5\omega_i$.

By Shannon's sampling theorem, we can write for each such function $\widehat{\phi}_{\omega_i}$ and each $f$
$$f*\phi_{\omega_i}= \sum_{I_i} \<f,\phi_{I_i,\omega_i}\> \phi_{I_i,\omega_i}$$
where $I_i$ runs through all dyadic intervals of length $(16|\omega_i|)^{-1}$
and $M_{-c(\omega_i)}\phi_{I_i,\omega_i}$ is an $L^2$- normalized bump function adapted to $I_i$ of order $2N$ such that $\widehat{\phi}_{I_i,\omega_i}$ is supported in $8\omega_i$.

We then estimate $\Lambda(f_1,f_2,f_3)$ by
\begin{equation}\label{modelform}
\sum_{\omega \in \Omega}
\sum_{I_1,I_2,I_3}
|\Lambda( \phi_{I_1,\omega_1},\phi_{I_2,\omega_2},\phi_{I_3,\omega_3})
|\prod_{i=1}^3 |\<f_i,\phi_{I_i,\omega_i}\>|
\end{equation}

Next, observe that for each $\omega \in \Omega$ with width $2^{-k}$ and eccentricity $2^{-\kappa }$ and each $\delta'<\delta$
$$|\Lambda(\phi_{I_1,\omega_1},\phi_{I_2,\omega_2},\phi_{I_3,\omega_3})|$$
\begin{equation}\label{concloflemma}
\lesssim  2^{-k/2}2^{-\kappa(\delta'+1/2)}(1+2^{-k}{\rm diam}(I_1,I_2,I_3))^{-1-\delta'}
(1+2^{-k}|\sum \gamma_i c(I_i)|)^{-N}.
\end{equation}

This will follow from Lemma \ref{threebumplemma}. Indeed, reasoning as before, we can find $\xi\in\R$ such that the tube $\omega-\xi\gamma$ is contained in some tube centered at the origin with width $100C_32^{-k}$ and eccentricity $2^{-\kappa}$. Since $|\xi\gamma_i-c(\omega_i)|=O(|\omega_i|)$ for each $i$, it follows that $M_{-\xi\gamma}(\phi_{I_1,\omega_1}\otimes\phi_{I_2,\omega_2}\otimes\phi_{I_3,\omega_3})$ is $L^2$- adapted to the box $I_1\times I_2\times I_3$. The fact that $\omega$ satisfies the requirements \eqref{altpropr1} and \eqref{altpropr2} (and actually all the other symmetric 4 identities) in Remark \ref{equivcond11} is immediate (see the proof of Lemma \ref{whitney}).

Using these estimates we shall restructure the sum in (\ref{modelform})
and extract the main terms.

First we shall use symmetry to reduce to the case where the sum runs  over
all tubes 
\begin{equation}
\label{trenumberff5slo}
\omega=(\omega_1,\omega_2,\omega_3):=(\omega_{1,1}\times\omega_{1,2}\times\omega_{2,2})
\end{equation}
 such that $\omega_{2,2}$ is the longest side and $\omega$ is determined by intersection of the tubes $\omega'=\omega_{1,1}\times\omega_{1,2}\times \R\in\Omega_3$, $\omega''=\R\times\omega_{2,1}\times\omega_{2,2}\in\Omega_1$ and $\omega'''=\omega_{3,1}\times\R\times\omega_{3,2}\in\Omega_2$.

At the expense of replacing the exponent $\delta$ in
(\ref{concloflemma}) by a slightly smaller $\delta'$ it
suffices to consider only those tubes $\omega$ with a fixed
eccentricity $2^{-\kappa}$ and those triples of intervals for which
\begin{equation}
\label{e.e7diaint1}
2^{m-1}\le \diam(I_1,I_2,I_3)/|I|\le 2^{m}
\end{equation}
\begin{equation}
\label{e.e7diaint2}
2^{m-2}\le \diam(I_2,I_3)/|I|
\end{equation}
for some fixed $m$ and prove summable bounds in $\kappa$ and $m$. Here we use again the notation $|I|:=|I_1|=|I_2|.$

Next, we shall use the rapid decay in the last factor in
(\ref{concloflemma}) to argue similarly to above that one only
needs to consider those terms for which this factor is large.
For fixed $I_3$ and $I_2$ we choose an
interval $\tilde{I}_1$ satisfying \eqref{e.e7diaint1} and \eqref{e.e7diaint2}, for which $\sum_j\gamma_jc(I_j)$ is minimal. For any other interval $I_1$ satisfying \eqref{e.e7diaint1} and \eqref{e.e7diaint2},
we note that the function $M_{-c(\omega_1)}\phi_{I_1,\omega_1}$
is adapted to $\tilde{I_1}$ of order $N$. The constant of adaption increases
like $(1+\dist(c(I_1),c(\tilde{I}_1))/|I|)^{N}$, which is offset by the
last factor in (\ref{concloflemma}) since
$$
|\gamma_1||c(I_1)-c(\tilde{I}_1)|\le\left| \sum_j\gamma_jc(I_j) \right | + O(I).
$$
Thus it suffices to consider only $\tilde{I_1}$. We shall write again $I_1$ for $\tilde{I}_1$ and we shall maintain from the above discussion that we are summing over
a two parameter family of intervals $I_1,I_2,I_3$ such that $I_1$ is determined
by $I_2$ and $I_3$. Likewise we may assume $I_2$ is determined by $I_3$
and $I_1$. Note that we cannot do the same for $I_3$, as $I_3$ is potentially much smaller than $I_1$ and $I_2$ and there may be many intervals $I_3$
(about $2^\kappa$) which maximize the last factor in
(\ref{concloflemma}), for given $I_1$ and $I_2$.

To summarize the above reductions, let $\Omega_{\kappa}^{*}$ denote all tubes as in \eqref{trenumberff5slo},  which have eccentricity $2^{-\kappa}$. Let $\I_m(\omega)$ denote the set of all triples $(I_1,I_2,I_3)$ associated with $\omega$ as above, satisfying  \eqref{e.e7diaint1} and such that both $I_1$ and $I_2$ are uniquely determined by the other two intervals. Then it suffices to get uniform bounds in $m$ and $\kappa$ over all functions $\phi_{I_i,\omega_i}$ such that $M_{-c(\omega_i)}\phi_{I_i,\omega_i}$ is $L^2$- adapted $I_i$ of order $N$, for the following sum
\begin{equation}\label{reducedmodel}
\sum_{\omega\in \Omega_\kappa^{*}}
\sum_{(I_1,I_2,I_3)\in \I_m(\omega)} |I_1|^{-1/2} 2^{-\kappa(1/2+\delta')}2^{-m(1+\delta')}
\prod_{i=1}^3 |\<f_i,\phi_{I_i,\omega_i}\>|,
\end{equation}
where $\delta'<\delta$ will be chosen conveniently (see Theorem \ref{mainmodels}).

It will be convenient to associate with $\omega$ and $(I_1,I_2,I_3)$ as above
\begin{enumerate}
\item For each $i=1,2,3$ a dyadic interval $\omega_{p_i}$ of length $2^{7}|\omega_i|$ which contains
the support $8\omega_i$ of the function $\widehat{\phi_{I_i,\omega_i}}$.
\item For each $i=1,2,3$ a dyadic interval $\omega_{R_i}$ of length $2^{7}|\omega_3|$ which contains
$\omega_{p_i}$. Note that $\omega_{R_3}=\omega_{p_3}$.
\item For each $i=1,2,3$ a dyadic interval $I_{p_i}$ of length $|\omega_{p_i}|^{-1}$ which is contained in $I_i$.
\item A dyadic interval $I_R$ of length $2^7\times2^m|I_1|$ which contains
$I_1$,$I_2$,$I_3$.
\end{enumerate}
It is clear that such intervals exist. There is no deep reason we choose to  modify $I_i$ to $I_{p_i}$ in (3), we only do that so that we have  $|I_{p_i}||\omega_{p_i}|=1$.

These intervals in general might not be standard dyadic intervals.
However, we can choose them to be generalized dyadic intervals. We briefly describe this construction and refer the reader to \cite{DTT} for more details.
Let $q$ be a large prime  and define a generalized
dyadic interval to be one of the form
$[2^n k/q,2^n(k/q+1))$
for integers $k$ and $n$. If $k$ is restricted to a fixed residue class modulo $q$, the collection above forms a grid, in the sense that all intervals of fixed scale form a partition of the real line and any two intervals in the grid
are either disjoint or one contains the other one.

For every interval $I$ , there are $q$ generalized dyadic intervals,
which we call covers, of length strictly between $q|I|$ and $2q|I|$
(thus the length is a uniquely determined power of two) which contain $I$.
More precisely, each grid except for possibly one of the $q$ grids contains
such a cover for $I$.
Therefore, for every collection of less than $q$ intervals,
there is one grid which contains a cover for each interval in the collection.

Using this method, for $q$ large enough we can choose for each $\omega$ and $I_1,I_2,I_3$ the intervals listed above to belong to one of the $q$ grids. We will work with $q=11$. Since there are only 11  grids, it suffices to consider the
sum (\ref{reducedmodel}) over each grid separately. For simplicity of notation we shall
only discuss the standard dyadic grid, which is one of the $11$ grids.
Since we only use the grid properties of the dyadic intervals,
our proof will easily transfer to the case of the other grids.

The above intervals determine rectangles in the phase plane:
$R_i=I_R\times \omega_{R_i}$ and $p_i=I_{p_i}\times \omega_{p_i}$.
We will denote by $R$ the 3-tuple of rectangles $(R_1,R_2,R_3)$ and will call it a multi-rectangle. Similarly, $p$ will denote the 3-tuple of tiles $(p_1,p_2,p_3)$, which will be referred to as a  multi-tile.

We discuss the important properties of the triples $R$ and $p$
other than the obvious containment properties visible in the
figure below.
First, by a further splitting into finitely many collections it suffices to assume that  each of ${R_1}$, ${R_2}$, ${R_3}$ determines the other two.
As all three rectangles have the same spatial interval $I_R$, it suffices to
show this for the frequency intervals. This in turn follows from the fact that the frequency intervals have equal side-length and from the easy\footnote{For similar results, see the rank properties in Section \ref{sec6.1}.} observation that\begin{equation}
\label{e.edilR11}
C_3(\omega_{R_1}\times \omega_{R_2}\times \omega_{R_3})\cap\langle \gamma\rangle\not=0.
\end{equation}

\hskip 150pt
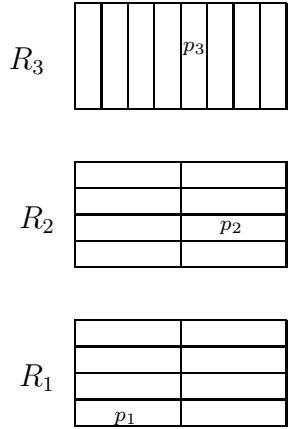
\begin{picture}(100,200)
\put(10,10){\line(1,0){80}}
\put(10,20){\line(1,0){80}}
\put(10,30){\line(1,0){80}}
\put(10,40){\line(1,0){80}}
\put(10,50){\line(1,0){80}}
\put(10,70){\line(1,0){80}}
\put(10,80){\line(1,0){80}}
\put(10,90){\line(1,0){80}}
\put(10,100){\line(1,0){80}}
\put(10,110){\line(1,0){80}}
\put(10,130){\line(1,0){80}}
\put(10,170){\line(1,0){80}}
\put(10,10){\line(0,1){40}}
\put(50,10){\line(0,1){40}}
\put(90,10){\line(0,1){40}}
\put(10,70){\line(0,1){40}}
\put(50,70){\line(0,1){40}}
\put(90,70){\line(0,1){40}}
\put(10,130){\line(0,1){40}}
\put(20,130){\line(0,1){40}}
\put(30,130){\line(0,1){40}}
\put(40,130){\line(0,1){40}}
\put(50,130){\line(0,1){40}}
\put(60,130){\line(0,1){40}}
\put(70,130){\line(0,1){40}}
\put(80,130){\line(0,1){40}}
\put(90,130){\line(0,1){40}}

\put(-15,145){$R_3$}
\put(-15,85){ $R_2$}
\put(-15,25){ $R_1$}
\put(51,152){\tiny $p_3$}
\put(65,84){\tiny $p_2$}
\put(25,12){\tiny $p_1$}
\end{picture}

Similarly, we can assume that  $\omega_{p_i}$ determines both $\omega_{p_j}$ and $\omega_{p_3}$, $i\not=j\in\{1,2\}$. The fact that $\omega_{p_i}$ determines  $\omega_{p_j}$ within finitely many choices follows from the fact that
$$C_1(\omega_{p_1}\times\omega_{p_2})\cap  \langle\gamma^{(3)}\rangle\not=0,$$
a consequence of the special representation $\omega=\omega'\cap\omega''$. Note also that $\omega_{p_j}$ determines $\omega_{R_j}$, which we have proved to determine $\omega_{R_3}=\omega_{p_3}$. We must point out however that $\omega_{p_3}$ only determines $\omega_{p_j}$ within $O(2^{\kappa})$ choices.

On the spatial side, we observe that the intervals $I_{p_3}$ and $I_{p_i}$ determine $I_{p_j}$, this property being reminiscent of the similar property shared by the intervals $I_i$. This implies that in the triple $(p_1,p_2,p_3)$ we can assume that $p_1,p_3$ vary freely and determine $p_2$ and also that $p_2,p_3$ vary freely and determine $p_1$. We recall for comparison the fact that in the case of $x$ independent kernels $K$, each $p_l,\,l\in\{1,2,3\}$ determined uniquely the other two (See the discussion in the beginning of this section).  

The area of any rectangle $R_i$ is
$A=O(2^{\kappa +m})$. If we incorporate the latest reductions, we  denote by $\p$ the family of all multi-tiles $p$ as above.
We will also denote by ${\bf R}$ the family of all multi-rectangles\footnote{From now, we will stop indexing the dependence of various collections like $\p$ and ${\bf R}$ on $\kappa$ and $m$.} $R$ (of fixed area $A$) associated with triples $p\in \p$.
For any fixed triple $R$, we denote by $\p(R)$ the set of all multi-tiles $p\in\p$ that are contained in $R$.

Finally we notice that for any given $R\in {\bf R}$ there are at most $A^2$ multi-tiles $p\in \p(R)$. On the other hand, any given $p\in \p$ determines
a unique $R\in {\bf R}$ such that $p\in \p(R)$. Let us see this latter point. Given $p=(p_1,p_2,p_3)$ we know that $p_3$ determines
$\omega_{R_3}$ which determines both $\omega_{R_1}$ and $\omega_{R_2}$. On the other hand,  both the length and the position of $I_R$ are determined by $I_{p_1}$, since $\kappa$ and $m$ are known a priori.

With these notations we may rewrite (\ref{reducedmodel}) as
\begin{equation}\label{rpmodel}
A^{-\delta'}\sum_{R\in {\bf R}} \sum_{p\in \p(R)}
|I_{R}|^{-1}|I_{p_3}|^{1/2}
\prod_{i=1}^3 |\<f_i,\phi_{p_i}\>|,
\end{equation}
where we can take  $0<\delta'<\delta$ as close to $\delta$ as we want. We will prove bounds for this sum that are then summable  in $\kappa$ and $m$.\\ 

For each measurable subset $E\subset \mathbb R$ with finite measure we define
$$
X(E)=\{ f : |f|\leq 1_{E} \hskip 5pt a.e.\},
$$
$$
X_2(E)=\{ f : |f|\leq |E|^{-1/2}1_{E} \hskip 5pt a.e.\}.
$$
A major subset of a set $E$ is a subset $E_0\subset E$ such that $|E_0|\ge |E|/2.$ 

\begin{definition} Let $\alpha$ be an $3$-tuple of real numbers and assume $\alpha_j \leq 1$ for all $j\in\{1,2,3\}$. A $3$-sublinear form is called of type $\alpha $ if there is
a constant $C$ such that for each finite measure tuple $E = (E_1,E_2,E_3)$, there is an index $j_0$ and a major subset $\tilde{E}_{j_0}$ of $E_{j_0}$ such that for all tuples $f = (f_1,f_2,f_3)$ with $f_j\in X(E_j)$ for all $j\neq j_0$ and $f_{j_0}\in X(\tilde{E}_{j_0})$ we have
$$
|\Lambda (f_1,f_2,f_3)|\lesssim \prod_{j=1}^{3}|E_j|^{\alpha_{j}}.
$$
\end{definition}
We will work with 
$\Lambda (f_1,f_2,f_3)=\sum_{R\in {\bf R}} \sum_{p\in \p(R)}
|I_{R}|^{-1}|I_{p_3}|^{1/2}
\prod_{i=1}^3 |\<f_i,\phi_{p_i}\>|.$ In the next section we will prove the following theorem.
\begin{theorem}
\label{mainmodels}
Let $j_0\in\{1,2,3\}$.
For each $\alpha$ in the triangular region defined by 
$$\alpha_1+\alpha_2+\alpha_3=1$$ 
$$1/2<\alpha_{j_1},\alpha_{j_2}<1,\hbox{ for }j_1,j_2\in\{1,2,3\}\setminus j_{0}$$
$$\max(-\delta,-1/2)<\alpha_{j_0}<0,$$ 
we have that $\Lambda (f_1,f_2,f_3)$ is of type $\alpha$ with bound $C=O(A^{\delta"} )$, for some $\delta" <\delta$ depending only on $\alpha$.
\end{theorem}

By choosing then in \eqref{rpmodel} $\delta'>\delta"$ we obtain that the original form $\Lambda$
is also of type $\alpha$ in the same regions.
Since the convex hull of the three triangular regions in the above theorem is the  region characterized by the restrictions
$$\alpha_1+\alpha_2+\alpha_3=1$$ 
$$\max(-\delta,-1/2)<\alpha_{j}<\min(\delta+1/2,1),\hbox{ for all }j,$$
Theorem \ref{main} follows then by invoking multilinear interpolation as in \cite{MTT1}, \cite{thielelectures}.

\vskip .1in

We  remark that the case $A=1$ in Theorem \ref{mainmodels} corresponds to the situation when the form $\Lambda$ is associated with an $x$ independent kernel $K(t)$ (see also the discussion at the beginning of this section). 

\medskip 

Finally, to prove Theorem \ref{mainmodels} we will assume that 
\begin{equation}
\label{sep00fscales}
R,R'\in{\bf R} \hbox{ and } |I_R|<|I_{R'}|\hbox{ implies } C_3A|I_R|\le |I_{R'}|,
\end{equation}
at the expense of considering $\log (C_3A)$ families in scales.

\section{The boundedness of the model sums}
\label{sec6}
\subsection{Rank, trees and sizes}
\label{sec6.1}

We will prove that if $C_1$ is sufficiently large we have the following:
\\
\\
{\bf Rank properties of $\p$}
\\
Let $p\in\p(R)$ be a multi-tile that is associated with a tube $\omega$ as in \eqref{trenumberff5slo}, according to the procedure described earlier. Let $(\xi^1,\xi^2,\xi^3)\in\langle \gamma \rangle$.

\begin{enumerate}
\item (3-ovelapping implies $j$-lacunary; tile version)
If  $\xi^3\in 2\omega_{p_3}$ then for each $j\in\{1,2\}$ we have $\xi^j\notin C_2\omega_{p_j}$ and $\xi^j\in C_3A\omega_{p_j}$.
\begin{proof}We will consider\footnote{For $j=1$ one has to use the fact that $\omega\subset \omega''':=(\omega_{3,1},\R,\omega_{3,2})$.} the case $j=2$. 
Recall that $\omega\subseteq \omega''=\R\times\omega_{2,1}\times\omega_{2,2}$. If $\xi^2\in C_2\omega_{p_2}$, then this together with $\xi^3\in 2\omega_{p_3}$ would force $\frac{C_1}{4}(\omega_{2,1}\times\omega_{2,2})\cap \langle\gamma^{(1)}\rangle\not=\emptyset$, contradicting the construction of $\omega''$. On the other hand, by construction it follows that there is some $(\tilde{\xi}^2,\tilde{\xi}^3)\in \langle\gamma^{(1)}\rangle\cap C_1(\omega_{2,1}\times\omega_{2,2})$. Note that $|\tilde{\xi}^3-\xi^3|\le  C_1|\omega_{p_3}|$. This implies $|\tilde{\xi}^2-\xi^2|\le  \frac{1}{100}C_3|\omega_{p_3}|$. This together with the fact that $\tilde{\xi}^2\in C_1\omega_{2,1}$, implies $\xi^2\in C_3\omega_{2,1}$. Finally, this together with the fact that $\omega_{2,1}\subset A\omega_{p_2}$ implies $\xi^2\in C_3A\omega_{p_2}$. 
\end{proof}
\item ($j$-ovelapping implies $3$-lacunary)
If $\xi^j\in2\omega_{p_j}$ for some $j\in\{1,2\}$ then we have $\xi^3\notin C_2\omega_{p_3}$ and $\xi^3\in C_3\omega_{p_3}$.
\begin{proof}
A similar argument as for (1) applies here.
\end{proof}
\item ($j$-ovelapping implies $i$-lacunary)
If $\xi^j\in 2\omega_{p_j}$ for some $j\in\{1,2\}$ then for $i=3-j$ we have $\xi^i\notin C_2\omega_{p_i}$ and $\xi^i\in C_3\omega_{p_i}$.
\begin{proof}
If $\xi^i\in C_2\omega_{p_i}$ then this together with $\xi^j\in 2\omega_{p_j}$ forces $(\xi^1,\xi^2)\in \langle\gamma^{(3)}\rangle\cap C_2(\omega_{p_1}\times\omega_{p_2})$, contradicting the construction of $\omega'$. On the other hand, from the construction of $\omega'$ we know that there is some $(\tilde{\xi}^1,\tilde{\xi}^2)\in \langle\gamma^{(3)}\rangle\cap C_1(\omega_{p_1}\times\omega_{p_2})$. Since $|\tilde{\xi}^j-\xi^j|\le C_1|\omega_{p_j}|$ and $\tilde{\xi}^i\in C_1\omega_{p_i}$, it follows that $\xi^i\in C_3\omega_{p_i}$. 
\end{proof}

\item (3-ovelapping implies $j$-lacunary; rectangle version)
If $\xi^3\in 2\omega_{R_3}$ then for each $j\in\{1,2\}$ we have $\xi^j\notin C_2\omega_{R_j}$ and $\xi^j\in C_3\omega_{R_j}$.
\begin{proof}
Similar arguments apply here and for the rest of the properties.
\end{proof}

\item ($j$-overlapping implies 3-lacunary; rectangle version) If $j\in\{1,2\}$ and $\xi^j\in 2\omega_{R_j}$, then $\xi^3\notin C_2\omega_{R_3}$ and $\xi^3\in C_3\omega_{R_3}$.

\item  ($j$-overlapping implies $i$-$C_3$ overlapping; rectangle version)
If $i,j\in\{1,2,3\}$ and $\xi^j\in 2\omega_{R_j}$ then $\xi^i\in C_3\omega_{R_i}$.
\end{enumerate}

\begin{definition}
Let $i\in\{1,2,3\}$.
An $i$-tree $(\T,\p_\T)$ with top $(I_\T,\xi_\T^i)$, where $I_\T$ is a dyadic interval, is a collection $\T\subset{\bf R}$ of multi-rectangles $R$ such that $I_R\subseteq I_\T$, together with a collection $\p_\T=\bigcup_{R\in\T}\p(R,\T)$ of multi-tiles satisfying $\p(R,\T)\subseteq\p(R)$ and  $\xi_\T^i\in 2\omega_{p_i}$ for each $R\in\T$ and each $p\in\p(R,\T)$.

A $0$-tree $(\T,\p_\T)$ with top $(I_\T,\xi_\T^0)$, where $I_\T$ is a dyadic interval, is a collection $\T\subset{\bf R}$ of multi-rectangles $R$ such that $I_R\subseteq I_\T$, together with a collection $\p_\T=\bigcup_{R\in\T}\p(R,\T)$ of multi-tiles satisfying $\p(R,\T)\subseteq\p(R)$ and either

(1) $\xi_\T^0\in 2\omega_{R_1}$ for each $R\in\T$ and  $\xi_\T^0\notin 2\omega_{p_1}$, $\xi_\T^2\notin 2\omega_{p_2}$ for each $p\in\p(R,\T)$, where $(\xi_\T^0,\xi_\T^2)\in\langle\gamma^{(3)}\rangle$ (this will be referred to as $0^{1}$-tree)

or 

(2) $\xi_\T^0\in 2\omega_{R_2}$ for each $R\in\T$ and  $\xi_\T^1\notin 2\omega_{p_1}$, $\xi_\T^0\notin 2\omega_{p_2}$ for each $p\in\p(R,\T)$, where $(\xi_\T^1,\xi_\T^0)\in\langle\gamma^{(3)}\rangle$ (this will be referred to as $0^2$-tree).
\end{definition}

Given a dyadic interval $I$ and a point $\vec{\xi}=(\xi^1,\xi^2,\xi^3)\in\langle\gamma\rangle$, define the saturation $\S(I,\vec{\xi})$ of the pair $(I,\vec{\xi})$ to the set of all multi-tiles which lie in the union of the maximal 1-tree with top $(I,\xi^1)$, the maximal 2-tree with top $(I,\xi^2)$, the maximal 3-tree with top $(I,\xi^3)$, the maximal $0^1$-tree with top $(I,\xi^1)$ and the maximal $0^2$-tree with top $(I,\xi^2)$. Note that actually
\begin{equation}
\label{theahamoment}
\S(I,\vec{\xi})=\bigcup_{i=1}^{3}\bigcup_{R:\xi^i\in2\omega_{R_i}\atop{I_R\subseteq I}}\p(R).
\end{equation}

We state three easy lemmas for future reference.

\begin{lemma}
\label{lljjggrr190gr}
If $\xi^l\in 2\omega_{R_l}$ for some $l\in\{1,2,3\}$ and $I_R\subseteq I$ then $\p(R)\subset \S(I,\vec{\xi})$.
\end{lemma}

\begin{lemma}
\label{lljjggrr190gr1}
If $\xi^i\in 2\omega_{R_i}$, $I_{R'}\subsetneq I$ and  $\omega_{R_j}\subsetneq \omega_{R_j'} $ for some $i,j\in\{1,2,3\}$ then $\p(R')\subset \S(I,\vec{\xi})$.
\end{lemma}
\begin{proof}
We know from  rank property (6) that $\xi^j\in C_3\omega_{R_j}$, and from \eqref{sep00fscales} we deduce that $\xi^j\in 2\omega_{R_j'}$. The conclusion now follows from the previous lemma. 
\end{proof}

The following is an immediate consequence of the rank properties.

\begin{lemma}
\label{lem.sumar322w}
If $i\in\{0,1,2,3\}\setminus \{j\}$, $(\T,\p_\T)$ is an $i$-tree with top $(I_\T,\xi_\T^i)$ and $p\in\p_\T$ then\footnote{Here $\xi_\T^j$ is the $j^{th}$ coordinate of the vector $\vec{\xi}_\T\in\langle\gamma\rangle$ which is uniquely determined by the coordinate $\xi_\T^i$. If $i\not=0$, then $\xi_\T^i$ is the $i^{th}$ coordinate of this vector, while if $i=0$, it is either the first or the second, depending on whether $(\T,\p_\T)$ is a $0^1$-tree or a $0^2$-tree.} $\xi_\T^j\notin 2\omega_{p_j}$ and $\xi_\T^j\in C_3A\omega_{p_j}$.
\end{lemma}

\begin{remark}
\label{r.rem.oveimplacuna}
Lemma  \ref{lem.sumar322w} and \eqref{sep00fscales} imply that for each $i\in\{0,1,2,3\}\setminus \{j\}$, each $i$-tree $(\T,\p_\T)$, each $R,R'\in \T$ with $|I_R|<|I_{R'}|$ and each $p\in\p(R,\T)$, $p'\in\p(R',\T)$ we have $\omega_{p_j}\cap \omega_{p_j'}=\emptyset$. 
\end{remark}
For each subcollection $\p^{*}(R)\subseteq \p(R)$ and each $l\in\{1,2,3\}$ we will denote $\p^{*}_l(R):=\{p_l:\,p\in\p^{*}(R)\}$.
For the simplicity of notation we will sometimes write $\T$ instead of $(\T,\p_\T)$.

\medskip

\begin{definition}
For $l\in \{1,2,3\}$, the tile size $\size_{l,l}$ of a collection $\p^{*}\subset\p$ of multi-tiles with respect to a function $f_l$ is defined as 
$$\size_{l,l}(\p^{*}):=\sup_{R\in{\bf R}}\left(\frac{\sum_{p_l\in \p^{*}_l(R)}|\<f_l,\phi_{p_l}\>|^2}{|I_{R}|}\right)^{1/2},$$
where $\p^{*}(R):=\p(R)\cap\p^{*}$.\footnote{While a given tile $p_l$ may correspond to more multi-tiles $p$ in $\p(R)$ or in $\p^{*}$, it will be counted only once in each summation.}
\end{definition}

\medskip

\begin{definition}
For $i\not=j\in \{1,2,3\}$ and for $j\in \{1,2\}$ and $i=0$, the tree size $\size_{j,i}$ of a collection $\p^{*}\subset\p$ of multi-tiles with respect to a function\footnote{The size will not be indexed by $f_j$, the function with respect to which the size is measured will always be clear from the context.} $f_j$ is defined as 
$$\size_{j,i}(\p^{*}):=\sup_{\T}\left(\frac{\sum_{R\in\T}\sum_{p_j\in \p_j(R,\T)}|\<f_j,\phi_{p_j}\>|^2}{|I_{\T}|}\right)^{1/2},$$
where the supremum is taken over all $i$-trees $(\T,\p_\T)$ with $\p_\T\subseteq \p^{*}$.
\end{definition}

We will estimate the model operator associated with each tree $(\T,\p_\T)$  by first summing in the third variable and then applying the Cauchy-Schwartz inequality in the first two variables. In doing so we recall that $p_i$ and $p_3$ determine $p_j$ uniquely 
\begin{align*}
&\Lambda_{\T}(f_1,f_2,f_3):=\sum_{R\in {\T}} |I_{R}|^{-1}\sum_{p\in \p(R,\T)}|I_{p_3}|^{1/2}
\prod_{i=1}^3 |\<f_i,\phi_{p_i}\>|\\&\le \sum_{R\in {\T}} |I_{R}|^{-1}(\sum_{p_1\in \p_1(R,\T)}|\<f_1,\phi_{p_1}\>|^2)^{1/2}(\sum_{p_2\in \p_2(R,\T)}|\<f_2,\phi_{p_2}\>|^2)^{1/2}\sum_{p_3\in \p_3(R,\T)}|I_{p_3}|^{1/2}|\<f_3,\phi_{p_3}\>|\\&\le \sum_{R\in {\T}} |I_{R}|^{-1/2}(\sum_{p_1\in \p_1(R,\T)}|\<f_1,\phi_{p_1}\>|^2)^{1/2}(\sum_{p_2\in \p_2(R,\T)}|\<f_2,\phi_{p_2}\>|^2)^{1/2}(\sum_{p_3\in \p_3(R,\T)}|\<f_3,\phi_{p_3}\>|^2)^{1/2}
\end{align*}
This estimate is refined as follows, depending on the type of tree we are dealing with.

For a 3-tree $(\T,\p_\T)$ we write
$$
\Lambda_{\T}(f_1,f_2,f_3)\le |I_\T|\left(\frac{\sum_{R\in\T}\sum_{p_1\in \p_1(R,\T)}|\<f_1,\phi_{p_1}\>|^2}{|I_{\T}|}\right)^{1/2}\left(\frac{\sum_{R\in\T}\sum_{p_2\in \p_2(R,\T)}|\<f_2,\phi_{p_2}\>|^2}{|I_{\T}|}\right)^{1/2}$$
$$
\times\left(\sup_{R\in\T}\frac{\sum_{p_3\in \p_3(R,\T)}|\<f_3,\phi_{p_3}\>|^2}{|I_{R}|}\right)^{1/2}$$ $$\le|I_\T|\size_{1,3}(\p_\T)\size_{2,3}(\p_\T)\size_{3,3}(\p_\T).$$
An identical estimate shows that if $(\T,\p_\T)$ is a $0$-tree then 
$$
\Lambda_{\T}(f_1,f_2,f_3)\le|I_\T|\size_{1,0}(\p_\T)\size_{2,0}(\p_\T)\size_{3,3}(\p_\T).$$

For an $i$-tree with $i\in\{1,2\}$ we write
$$\Lambda_{\T}(f_1,f_2,f_3)\le |I_\T|\left(\frac{\sum_{R\in\T}\sum_{p_j\in \p_j(R,\T)}|\<f_j,\phi_{p_j}\>|^2}{|I_{\T}|}\right)^{1/2}\sup_{R\in\T}\left(\frac{\sum_{p_i\in \p_i(R,\T)}|\<f_i,\phi_{p_i}\>|^2}{|I_{R}|}\right)^{1/2}$$
$$\times\left(\frac{\sum_{R\in\T}\sum_{p_3\in \p_3(R,\T)}|\<f_3,\phi_{p_3}\>|^2}{|I_{\T}|}\right)^{1/2}$$ $$\le|I_\T|\size_{j,i}(\p_\T)\size_{i,i}(\p_\T)\size_{3,i}(\p_\T).$$

We thus see that if a collection $\p^{*}$ of multi-tiles is organized as a disjoint union $\F$ of trees, $\p^{*}=\bigcup_{(\T,\p_\T)\in\F}\p_\T$, then
\begin{equation}
\label{orgtreesizes}
|\Lambda_{\p^{*}}(f_1,f_2,f_3)|\le (\sum_{i=0}^{3}\prod_{j=1}^{3}\size_{j,i}(\p^{*}))(\sum_{\T\in\F}|I_\T|),
\end{equation}
where for the purpose of keeping the notation symmetric we denote $\size_{3,0}:=\size_{3,3}$.

This inequality sets up the strategy for the following sections, where we will split $\p^{*}$ into collections which can be organized into trees, with good control over both their sizes and over the $L^1$ norm of the counting function of the tops of the trees.

\subsection{Bessel type inequalities}
\label{sec6.2}

For each $j\in\{1,2,3\}$, each $R\in{\bf R}$ and each $\p^{*}(R)\subseteq \p(R)$ we will use the notation\footnote{When no confusion can arise, we will suppress the  dependence on $\p^{*}(R)$ of $T_R^j(f)$ and $S_R^j(f)$.} 
$$
T_{R,\p^{*}(R)}^j(f)=\sum_{p_j\in \p^{*}_j(R)}\langle f,\phi_{p_{j}}\rangle \phi_{p_{j}}
$$
$$
S_{R,\p^{*}(R)}^j(f)=(\sum_{p_j\in \p^{*}_j(R)}|\langle f,\phi_{p_{j}}\rangle|^2)^{1/2}.
$$
If $R$ belongs to a tree $(\T,\p_{\T})$ and $\p^{*}(R)=\p(R,\T)$, then the notation $T_{R,\T}^j(f)$, $S_{R,\T}^j(f)$ will be preferred.

An immediate consequence of Lemma \ref{molecules}  is the fact that
\begin{equation}
\label{whs73219872}
\|T_{R,\p^{*}(R)}^j(f)\|_2\approx S_{R,\p^{*}(R)}^j(f) \lesssim \|f\|_2
\end{equation}
for each $f\in L^2$, each $R$ and each $\p^{*}(R)\subseteq \p(R)$. Similarly, due to Remark \ref{r.rem.oveimplacuna}, for each $i\in\{0,1,2,3\}$, each $j\in\{1,2,3\}\setminus\{i\}$ and each $i$-tree $(\T,\p_{\T})$
\begin{equation}
\label{whs73219872hhhyt}
\sum_{R\in\T}\|T_{R,\T}^j(f)\|_2\approx \sum_{R\in\T}S_{R,\T}^j(f) \lesssim \|f\|_2.
\end{equation}

\begin{definition}(M-separated rectangles)
We say that a family ${\bf R^{*}}$ of rectangles $R=I_R\times \omega_R$ is $M$-separated if for any $R,R'\in {\bf R^{*}} $, $R\neq R'$, we have that
$(MI_R\times \omega_{R})\cap (MI_{R'}\times \omega_{R'})=\emptyset$.
\end{definition}

\begin{lemma}\label{gettingAseparatedness}
Let ${\bf R^{*}}$ be a finite family of pairwise disjoint rectangles.
Then, for each $M\ge 1$ we can find a subfamily $\tilde{{\bf R}}\subset {\bf R^{*}}$ that is $M$-separated  and satisfies
$$
\sum_{R\in {\bf R^{*}}}|I_R|\leq 3M\sum_{R\in \tilde{{\bf R}}}|I_{R}|.
$$
\end{lemma}

\begin{proof}
Fix $M\ge 1$  and  define $R_{M}=MI_R\times \omega_{R}$.
We select recursively rectangles $R\in {\bf R^{*}}$
with maximal $|I_R|$ and with the property that
$R_{M}\cap R'_{M} =\emptyset $
for all previously selected rectangles $R'$.
When this procedure ends, we get a family of rectangles $\tilde{{\bf R}}=\{ R^1\ldots ,R^n\} \subset {\bf R^{*}}$. We now define
$$
{\bf R}^{i}
=\{ R\in {\bf R^{*}} :R_{M} \cap R^{i}_{M }\neq \emptyset \hskip 3pt ,\hskip 3pt
R_{M} \cap R^{k}_{M} =\emptyset \hskip 5pt {\rm for}\hskip 5pt k<i\}.
$$
Then it is clear that $\tilde{\bf R}$ is $M$-separated and that
$$
{\bf R^{*}}=\bigcup_{i=1}^{n}{\bf R}^{i},
$$
thus
$$
\sum_{R\in {\bf R^{*}}}|I_R|= \sum_{i=1}^n\sum_{R\in {\bf R}^{i}}|I_R|
$$

Moreover, $|I_{R}|\leq |I_{R^i}|$
for any $R\in {\bf R}^{i}$. Otherwise, if
$|I_{R}|>|I_{R^i}|$, by the maximality condition of length and the fact that
$R_{M} \cap R^{k}_{M }=\emptyset $ for all $k<i$,
$R$ should have been chosen instead $R^i$.
This together with the observation that  $MI_R\cap MI_{R^i}\neq \emptyset$ for each  $R\in {\bf R}^{i}$ implies that
$I_R\subset 3MI_{R^i}$. Thus,
\begin{equation}
\label{e.eincl1267}
\bigcup_{R\in {\bf R}^i}I_R\subset 3MI_{R^i}.
\end{equation}

On the other hand for each $R,R'\in{\bf R}^i$ we know that $\omega_{R^i}\subset \omega_{R}\cap \omega_{{R'}}\not=\emptyset$, which together with $R\cap R'=\emptyset$ implies that $I_{R}\cap I_{R'}=\emptyset.$ This together with \eqref{e.eincl1267} implies that
$$
\sum_{R\in {\bf R^{*}}}|I_R|\leq 3M\sum_{i=1}^n|I_{R^i}|
=3M\sum_{R\in \tilde{\bf R}}|I_{R}|.
$$

\end{proof}

\begin{definition}
Let $j\in\{1,2,3\}$.
\begin{enumerate}
\item
 We say that two trees $(\T,\p_{\T})$ and $(\T',\p_{\T'})$ are disjoint if $\T\cap\T'=\emptyset$. 
\item
We say that two trees $(\T,\p_{\T})$ and $(\T',\p_{\T'})$ with tops $(I_\T,\xi_\T)$ and $(I_{\T'},\xi_{\T'})$ are $j$-strongly disjoint if they are disjoint and satisfy the following: 
\\
if $p\in\p(R,\T)$, $p'\in\p(R',\T')$, $R\in \T$, $R'\in \T'$, $\omega_{p_j}\subsetneq \omega_{p_j'}$, then $I_{R'}\cap I_{\T}=\emptyset$
\end{enumerate} 
A family of trees is said to consist of $j$-strongly disjoint trees if any two trees in the family are $j$-strongly disjoint.
\end{definition}

The following lemma will be the main tool in dealing with the tile sizes. 

\begin{lemma}
\label{almostortho}
Let $j\in \{1,2,3\}$, $f\in L^2$ and $\lambda>0$. Let ${\bf R^{*}}\subset {\bf R}$ be a family of multi-rectangles, each of which is associated with a collection of multi-tiles $\p^{*}(R)\subseteq \p(R)$. Assume that the rectangles $(R_j)_{R\in {\bf R^{*}}}$ are pairwise disjoint.

Assume  also that for  each  $R\in {\bf R^{*}}$ we have
\begin{equation}
\label{lowboundfornorm}
S_{R,\p^{*}(R)}^j(f)\geq \lambda |I_R|^{1/2}.
\end{equation}
Then
$$
\sum_{R\in {\bf R^{*}}}|I_R|\lesssim A^{\frac{2}{N-2}}\lambda^{-2}\| f\|_2^2.
$$
\end{lemma}
\begin{proof} 

To simplify notation we will drop the $j$ dependence of the various operators and will index them  only by $R$. Since the rectangles $R_j$ are pairwise disjoint, by Lemma \ref{gettingAseparatedness} it suffices to assume that they are $A^{\epsilon}$-separated with $\epsilon=\frac{2}{N-2}$, and to prove that
$$
\sum_{R\in{\bf R}^{*}}|I_R|\lesssim \lambda^{-2}\| f\|_2^2.
$$
We assume that $\|f\|_2=1$. We may also assume that for each $R\in{\bf R}^{*}$
\begin{equation}\label{scaleslowboundfornorm}
\lambda |I_R|^{1/2}\leq S_{R}(f)\leq 2\lambda |I_R|^{1/2}.
\end{equation}
To see this latter assumption, one can split ${\bf R}^{*}$ into subcollections ${\bf R}^{*}=\bigcup_{k\ge 0}{\bf R}^{*}_k$ such that for each $R\in {\bf R}^{*}_k$ we have $2^{k}\lambda |I_R|^{1/2}\leq S_{R}(f)\leq 2^{k+1}\lambda |I_R|^{1/2}.$ In other words, each subcollection satisfies \eqref{scaleslowboundfornorm} with $\lambda$ replaced by $2^k\lambda$, and applying the estimate to each subcollection and summing a geometric series will prove the general form of the lemma. The proof that follows is a classical instance of the $TT^{*}$ argument.

We have
\begin{align*}
\sum_{R\in{\bf R}^{*}}\| T_{R}(f) \|_2^2&=\sum_{R\in{\bf R}^{*}}\langle T_{R}^*(f),T_{R}^*(f)\rangle
\\&=\sum_{R\in{\bf R}^{*}}\langle T_{R}T_{R}^*(f),f\rangle
\\&\leq \Big\| \sum_{R\in{\bf R}^{*}}T_{R}T_{R}^*(f)\Big\|_2
\end{align*}
which due to \eqref{whs73219872} implies  
\begin{align*}
\Big(\sum_{R\in{\bf R}^{*}}S_{R}(f)^2\Big)^2&\lesssim \sum_{R,R' \in{\bf R}^{*}}\langle T_{R'}T_{R'}^*(f),T_{R}T_{R}^*(f)\rangle
\\&=\sum_{R,R'\in{\bf R}^{*}}\langle T_{R'}^{*}(f),T_{R'}^*T_{R}T_{R}^*(f)\rangle
\\&\leq \sum_{R,R'\in{\bf R}^{*}}\|T_{R'}(f)\|_2\| T_{R'}^*T_{R}\|_{2\to 2} \| T_{R}(f)\|_2.
\end{align*}
Using this, hypothesis (\ref{scaleslowboundfornorm}) and again \eqref{whs73219872} we get
$$
\Big(\sum_{R\in{\bf R}^{*}}\lambda^2|I_R|\Big)^2
\lesssim \lambda^2\sum_{R,R'\in{\bf R}^{*}}|I_R|^{1/2}|I_{R'}|^{1/2}\| T_{R'}^*T_{R}\|_{2\to 2}.
$$
By symmetry and the fact that $T_{R'}^*T_{R}\equiv 0$ if $\omega_{R_j}\cap \omega_{R_j'}=\emptyset$ we have
$$\Big(\sum_{R\in{\bf R}^{*}}|I_R|\Big)^2\lesssim$$
\begin{equation}
\label{diagonalandoffdiagonal}
\lambda^{-2}\Big(\sum_{R\in{\bf R}^{*}}
|I_R|\| T_{R}^*T_{R}\|_{2\to 2}
+\sum_{R\in{\bf R}^{*}}\sum_{R'\in \F_{\rm freq}(R)}
|I_R|^{1/2}|I_{R'}|^{1/2}\| T_{R'}^*T_{R}\|_{2\to 2}\Big)
\end{equation}
where $\F_{\rm freq}(R)=\{R'\in{\bf R}^{*}: R'\neq R, |I_{R'}|\leq|I_R|,
\omega_{R_j}\cap \omega_{R_j'}\neq \emptyset\}$. It suffices to prove now that the term in \eqref{diagonalandoffdiagonal} is $O(\lambda^{-2}\sum_{R\in{\bf R}^{*}}|I_R|)$. To achieve this, we will first estimate the operator norms $\| T_{R'}^*T_{R}\|_{2\to 2}$.

The diagonal term is immediately seen to be $O(\lambda^{-2}\sum_{R\in{\bf R}^{*}}|I_R|)$, due to \eqref{whs73219872}, since
$\| T_{R}^*T_{R}\|_{2\to 2}\leq\| T_{R}\|_{2\to 2}^2\lesssim 1$.
To estimate the off diagonal,  we first note that
\begin{align}
\nonumber
|\langle T_{R'}^*T_{R}(f),g\rangle |
&=|\langle T_{R}(f),T_{R'}(g)\rangle |
\\&\nonumber\leq \sum_{p_j\in \p_j^{*}(R)}\sum_{p_j'\in\p_j^{*}(R') }
| \langle f,\phi_{p_j}\rangle | | \langle g,\phi_{p_j'}\rangle | | \langle \phi_{p_j},\phi_{p_j'}\rangle|
\label{e.eultyu765}
\\&\leq A^2\hspace{-.3cm}\max_{\tiny
\begin{array}{l}
p_j\in\p_j^{*}(R)\\
p_j'\in \p_j^{*}(R')
\end{array}}
\hspace{-.2cm}|\langle \phi_{p_j},\phi_{p_j'}\rangle |\hskip 5pt \| f\|_2\| g\|_2.
\end{align}
From Lemma \ref{molecules} we know the bound
\begin{equation}
\label{hgsdgyet11}
|\langle \phi_{p_j},\phi_{p_j'}\rangle |
\lesssim \Big( \frac{|I_{p_j'}|}{|I_{p_j}|}\Big)^{1/2}(1+|I_{p_j}|^{-1}|c(I_{p_j})-c(I_{p_j'})|)^{-N}.
\end{equation}
If $R'\in \F_{\rm freq}(R)$, it follows that $A^{\epsilon}I_R\cap A^{\epsilon}I_{R'}=\emptyset$ since $R, R'$ are $A^{\epsilon}$ separated. This implies that
\begin{equation}
\label{puntolino}
1+|I_{p_j}|^{-1}|c(I_{p_j})-c(I_{p_j'})|\ge \max\{1+|I_R|^{-1}|c(I_R)-c(I_{R'})|,A^{\epsilon}/2\}.
\end{equation}
We use this and \eqref{hgsdgyet11} to argue that 
$$
|\langle \phi_{p_j},\phi_{p_j'}\rangle|\lesssim \Big( \frac{|I_{R'}|}{|I_{R}|}\Big)^{1/2}(1+|I_{R}|^{-1}|c(I_{R})-c(I_{R'})|)^{-2}A^{\epsilon(2-N)}
$$
Thus
\begin{equation}\label{operatornormbound}
\| T_{R'}^*T_{R}\|_{2\to 2}
\lesssim A^{2-\epsilon(N-2)}\Big( \frac{|I_{R'}|}{|I_R|}\Big)^{1/2}(1+|I_R|^{-1}|c(I_R)-c(I_{R'})|)^{-2}
\end{equation}
We use this inequality to bound the sum corresponding to the off diagonal term in  (\ref{diagonalandoffdiagonal}) by
\begin{equation}\label{freq}
CA^{2-\epsilon(N-2)}\sum_{R\in {\bf R}^{*}}\sum_{R'\in \F_{\rm freq}(R)}
|I_{R'}|(1+|I_R|^{-1}|c(I_R)-c(I_{R'})|)^{-2}.
\end{equation}

Now we fix $R\in {\bf R}^{*}$. We note that if $R'\not=R''$ and $R',R''\in \F_{\rm freq}(R)$ then $I_{R'}\cap I_{R''}=\emptyset.$ With these observations we may write

\begin{align*}
\sum_{R'\in \F_{\rm freq}(R)}
|I_{R'}|(1+|I_R|^{-1}|c(I_R)-c(I_{R'})|)^{-2}&
\lesssim \sum_{R'\in \F_{\rm freq}(R)}
\int_{I_{R'}}\chi_{I_R}(x)^{2}dx\\&\lesssim\int_{I_{R}^{c}}\chi_{I_R}(x)^{2}dx,
\\&\lesssim|I_R|.
\end{align*}

Finally, we get
$$
\Big(\sum_{R\in {\bf R}^{*}}|I_R|\Big)^2
\lesssim\lambda^{-2}\Big(\sum_{R\in {\bf R}^{*}}|I_R|+A^{2+\epsilon(2-N)}\sum_{R\in {\bf R}^{*}}|I_R|\Big)\lesssim \lambda^{-2}\sum_{R\in {\bf R}^{*}}|I_R|.$$
\end{proof}

\begin{definition}(M-separated tiles in a tree)
We say that a tree $(\T,\p_{\T})$ with top $(I_\T,\xi_\T)$ is $M$-separated if $MI_R\subseteq I_\T$ for each $R\in\T$.
\end{definition}

\begin{lemma}
\label{A-separated}
Let $i\in \{0,1,2,3\}\setminus\{j\}$, $M\ge 1$, $f\in L^2$ and $\lambda>0$. Let $\F$ be a collection of $j$-strongly disjoint\footnote{The lemma can be formulated without involving strongly disjointness in either the hypothesis or the conclusion; we choose this formulation since this is how the lemma will be applied.} $i$-trees $(\T,\p_\T)$  satisfying for each $\T\in\F$ and each $R\in\T$
\begin{equation}
\label{upboundfornorm2}
S_{R,\T}^j(f)\le 10^{-2}M^{-1/2}\lambda |I_R|^{1/2} 
\end{equation}
\begin{equation}\label{givenbysize2}
\lambda^2 |I_{\T}|/4\le \sum_{R\in \T}S_{R,\T}^j(f)^{2}\le \lambda^2 |I_{\T}|.
\end{equation}
Assume also that for each subtree $\T'\subseteq\T\in\F$ with top $(I_{\T'},\xi_\T)$ we have 
\begin{equation}
\label{scalesgivenbysize2}
\sum_{R\in {\T'}}S_{R,\T}^j(f)^{2}\le  \lambda^2 |I_{{\T'}}|.
\end{equation}
Then for each $\T\in\F$ we can find a  collection  $\tilde{\T}\subseteq \T$ such that the trees $\{(\tilde{\T}, \p_{\tilde{\T}}:=\bigcup_{R\in\tilde{\T}}\p(R,\tilde{\T})):\,{\T\in\F}\}$ (understood as having the same tops as the original trees) with $\p(R,\tilde{\T})=\p(R,{\T})$ for each $R\in\tilde{\T}$, are $M$-separated,  $j$-strongly disjoint and they satisfy
$$
\frac{1}{20}\lambda^2 |I_{\tilde{\T}}|\le \sum_{R\in \tilde{\T}}S_{R,\T}^j(f)^{2}\le \lambda^2 |I_{\tilde{\T}}|.$$
\end{lemma}

\begin{proof}To get the new trees we proceed by removing those tiles in a tree that are either too close to the edges of the tree top-tile or too close in scale to it. More precisely, for any $i$-tree $\T\in\F$ 
we let $I(\T)=\frac{9}{10}I_\T$ and we define the subcollections
$$
\T_1=\{ R\in \T : I_R<I(\T)\}
$$
$$
\T_2=\{ R\in \T : I_R>I(\T)\}
$$
$$
\T_3=\{ R\in \T : |I_R|>(100M)^{-1}|I_\T|\}
$$
where $I_R<I(\T)$ means $\sup\{x:x\in I_R\}\le \inf\{x:x\in I(\T)\}$. Then, we define
$$
\tilde{\T}=\T\setminus(\T_1\cup \T_2\cup \T_3).
$$
With these definitions we have by (\ref{scalesgivenbysize2})
$$
\sum_{R\in \T_1}S_R^{j}(f)^{2}\leq \lambda^2|I_{\T_1}|\le \frac{1}{20}\lambda^2|I_\T|
$$
and the same for $\T_2$,
while by (\ref{upboundfornorm2})
$$
\sum_{R\in \T_3}S^j_R(f)^{2}\leq 10^{-4}M^{-1}\lambda^2\sum_{R\in T_3}|I_R|\le \frac{1}{10}\lambda^2|I_\T|. 
$$
It easily follows that each tree $\tilde{\T}$ is $M$-separated. The fact that the new collection of trees is $j$-strongly disjoint is inherited from the initial collection. Finally, 
$$
\sum_{R\in \tilde{\T}}S^j_R(f)^{2}
\geq \sum_{R\in \T}S^j_R(f)^{2}-\sum_{i=1,2,3}\sum_{R\in \T_i}S^j_R(f)^{2}
$$
$$
\geq \frac14\lambda^2|I_\T|-\frac{2}{10}\lambda^2|I_\T|\ge\frac1{20}\lambda^2|I_\T|.
$$

\end{proof}

\begin{definition}(Rectangles M-separated in scales)
We say that a collection  ${\bf R^{*}}$ of multi-rectangles is $M$-separated in scales if 
$R\not=R'\in {\bf R^{*}}$ and $|I_{R}|=|I_{R'}|$ imply that $\dist(c(I_R),c(I_{R'}))\ge M|I_R|.$
\end{definition}

We are now ready to prove the analog of Lemma \ref{almostortho} for the case when the trees may consist of more than just one multi-rectangle. This lemma will be the main tool in dealing with the tree sizes.
\begin{lemma}\label{almostortho3}
Let $i\in \{0,1,2,3\}\setminus\{j\}$, $M\ge 1$, $f\in L^2$ and $\lambda>0$. Let ${\bf R^{*}}$ be a collection of multi-rectangles which is $M$-separated in scales. Assume we also have a collection $\F$ of $j$-strongly disjoint $i$-trees $(\T,\p_\T)$ with $\bigcup_{\T\in \F}\T={\bf R^{*}}$ satisfying for each $\T\in\F$ and each $R\in\T$
\begin{equation}
\label{upboundfornorm22}
S_{R,\T}^j(f)\le 10^{-2}M^{-1/2}\lambda |I_R|^{1/2} 
\end{equation}
\begin{equation}\label{givenbysize22}
\lambda^2 |I_{\T}|/4\le \sum_{R\in \T}S_{R,\T}^j(f)^{2}\le \lambda^2 |I_{\T}|.
\end{equation}
Assume also that for each subtree $\T'\subseteq\T\in\F$ with top $(I_{\T'},\xi_\T)$ we have 
\begin{equation}
\label{scalesgivenbysize22}
\sum_{R\in {\T'}}S_{R,\T}^j(f)^{2}\le  \lambda^2 |I_{{\T'}}|.
\end{equation}
Then we have
$$
\sum_{\T\in \F}|I_\T|\lesssim (1+A^2M^{3-N})\lambda^{-2}\| f\|_2^2.
$$
\end{lemma}
\begin{proof}
We will drop the $j$ and $\T$ dependence of various operators.
By normalizing we can assume that $\|f\|_2=1$.

Moreover, by Lemma \ref{A-separated} we can assume the trees in $\F$ are $M$-separated, with some loss in the constants from \eqref{givenbysize22} and \eqref{scalesgivenbysize22}. Indeed, the new trees fabricated by the procedure in Lemma \ref{A-separated} have the same tops as the old ones so that the quantity $\sum_{\T\in \F}|I_\T|$ does not change. 

The proof is another $TT^{*}$ argument, that follows along the lines of the one in Lemma \ref{almostortho}. Due to \eqref{givenbysize22} we get as before 
\begin{align}
\label{e.hh.1kl1}
\Big(\lambda^2\sum_{\T\in \F}|I_\T|\Big)^2
&\lesssim 
\sum_{R,R'\in {\bf R}^{*}\atop{|I_{R'}|=|I_R|}}
\|T_{R'}(f)\|_2\| T_{R'}^*T_{R}\|_{2\to 2}\| T_{R}(f)\|_2
\\&\label{e.hh.1kl2}+\sum_{R,R'\in {\bf R}^{*}\atop{|I_{R'}|<|I_R|}}
\|T_{R'}(f)\|_2\| T_{R'}^*T_{R}\|_{2\to 2} \| T_{R}(f)\|_2.
\end{align}

We estimate each term separately.
By Cauchy-Schwartz, the first sum can be bounded by
\begin{equation}
\label{e.ee.ee.1285t}
\frac{1}2\sum_{R,R'\in {\bf R}^{*}\atop{|I_{R'}|=|I_R|\atop{\omega_{R_j}=\omega_{R_j'}}}}(\|T_{R}(f)\|_2^2+\|T_{R'}(f)\|_2^2)
\| T_{R}^*T_{R'}\|_{2\to 2}.
\end{equation}
By using \eqref{e.eultyu765} and the fact that whenever $p\in\p(R)$, $p'\in\p(R')$ we have
$$|I_{p_j}|^{-1}|c(I_{p_j})-c(I_{p_j'})|\ge |I_R|^{-1}|c(I_R)-c(I_{R'})|.$$
we can derive similar to \eqref{operatornormbound} 
$$
\|T_{R'}^*T_{R}\|_{2\to 2}\lesssim A^2(1+|I_R|^{-1}|c(I_R)-c(I_{R'})|)^{-N},
$$
whenever $|I_{R'}|=|I_R|$. By using this inequality in the case $R'\not= R$ and the fact that $\|T_{R}^*T_{R}\|_{2\to 2}\lesssim 1$, the term in \eqref{e.ee.ee.1285t} can further be bounded by 
\begin{equation}
\label{e.ee.ee.1285t1}
\sum_{R\in {\bf R}^{*}}\|T_{R}(f)\|_2^2(1+\sum_{R'\in {\bf R}^{*}\setminus R\atop{|I_{R'}|=|I_R|\atop{\omega_{R}=\omega_{R'}}}}A^2(1+|I_R|^{-1}|c(I_R)-c(I_{R'})|)^{-N}).
\end{equation}

Finally, using the $M$- separatedness in scales, the term in \eqref{e.ee.ee.1285t1} (and thus the term in \eqref{e.hh.1kl1}) is bounded by
$$\sum_{R\in {\bf R}^{*}}\|T_{R}(f)\|_2^2(1+A^2M^{-N})\lesssim\lambda^2(1+A^2M^{-N})\sum_{\T\in\F}|I_\T|.
$$

We will next concentrate on proving similar bounds for the term in \eqref{e.hh.1kl2}. We start with a few observations. Fix $R$ and $\T$ such that $R\in\T$ and denote 
$${\bf R}^{*}(R)=\{R'\in {\bf R}^{*}:|I_{R'}|<|I_R|,\;T_{R'}^{*}T_{R}\not\equiv 0\}.$$  

The first point we make is that if $R'\in {\bf R}^{*}(R)$ then $R'\notin\T$.
This is an immediate consequence of Remark \ref{r.rem.oveimplacuna}. Using this and the $j$-strongly disjointness we see that  $R'\in {\bf R}^{*}(R)$ implies $I_{R'}\cap I_{\T}=\emptyset.$ 

Secondly, we observe that $R'\not=R''\in {\bf R}^{*}(R)$ implies that $I_{R'}\cap I_{R''}=\emptyset.$ To prove this, we first note that the definition of ${\bf R}^{*}(R)$ guarantees the existence of $\T',\T''\in\F$, $p^{(1)},p^{(2)}\in \p(R,\T)$ and  $p^{(3)}\in \p(R',\T')$, $p^{(4)}\in \p(R'',\T'')$ such that $\omega_{p_j^{(1)}}\subsetneq \omega_{p_j^{(3)}}$ and $\omega_{p_j^{(2)}}\subsetneq \omega_{p_j^{(4)}}$. Due to  \eqref{sep00fscales} and due to the grid structure it follows that $\omega_{R_j}\subseteq \omega_{p_j^{(3)}}\cap\omega_{p_j^{(4)}}.$ In particular, $\omega_{p_j^{(3)}}\cap\omega_{p_j^{(4)}}\not=\emptyset.$
We distinguish two cases.
If $|\omega_{p_j^{(3)}}|=|\omega_{p_j^{(4)}}|$ then clearly $I_{R'}\cap I_{R''}=\emptyset$, since otherwise we would get $R'=R''$. If $|\omega_{p_j^{(3)}}|<|\omega_{p_j^{(4)}}|$ then we first argue as above using Remark \ref{r.rem.oveimplacuna} that $\T'\not=\T''$, and then using the $j$-strongly disjointness that $I_{R''}\cap I_{\T'}=\emptyset$. We conclude again that $I_{R'}\cap I_{R''}=\emptyset$.    

We next estimate $\|T_{R'}^*T_{R}\|_{2\to 2}$ for $R'\in{\bf R}^{*}(R)$. Note that due to the first observation above and the $M$-separatedness of the trees we have $MI_{R}\cap I_{R'}=\emptyset$. By estimating like in \eqref{e.eultyu765} and \eqref{puntolino} we then get 
\begin{align*}
\|T_{R'}^*T_{R}\|_{2\to 2}&\lesssim A^2\Big(\frac{|I_{R'}|}{|I_R|}\Big)^{1/2}(1+|I_R|^{-1}|c(I_R)-c(I_{R'})|)^{-N}\\&\lesssim A^2M^{4-N}\Big(\frac{|I_{R'}|}{|I_R|}\Big)^{1/2}(1+|I_R|^{-1}|c(I_R)-c(I_{R'})|^{-4}.
\end{align*}

We use \eqref{upboundfornorm22} to estimate 
$$\|T_{R'}(f)\|_2\lesssim \|S_{R'}(f)\|_2\lesssim M^{-1/2}\lambda|I_{R'}|^{1/2}$$
and similarly for $R$. If we take into account all these observations, we can bound the term in \eqref{e.hh.1kl2} by
\begin{align*}
A^2M^{3-N}\lambda^2&\sum_{\T\in\F}\sum_{R\in\T}\sum_{R'\in {\bf R}^{*}(R)}|I_{R'}|(1+|I_R|^{-1}|c(I_R)-c(I_{R'})|)^{-4}\\&\lesssim A^2M^{3-N}\lambda^2\sum_{\T\in\F}\sum_{R\in\T}\sum_{R'\in {\bf R}^{*}(R)}\int_{I_{R'}}(1+|I_{R}|^{-1}|c(I_R)-x|)^{-4}dx\\&\lesssim A^2M^{3-N}\lambda^2\sum_{\T\in\F}\sum_{R\in\T}\int_{I_{\T}^c}(1+|I_{R}|^{-1}|c(I_R)-x|)^{-4}dx\\&\lesssim A^2M^{3-N}\lambda^2\sum_{\T\in\F}\sum_{R\in\T}|I_R|(1+|I_{R}|^{-1}\dist(c(I_R),I_\T^c))^{-2}\\&\lesssim A^2M^{3-N}\lambda^2\sum_{\T\in\F}|I_\T|.
\end{align*}
\end{proof}

\subsection{The Peeling Lemma}
\label{sec6.3}

The following lemma will provide a decomposition of the multi-tiles in collections with good control over various sizes.

\begin{lemma}\label{pealinglemma}Let $i\in \{0,1,2,3\}$, $j\in \{1,2,3\}$ and $f\in L^2$. Let $\p^{*}\subset\p$ be a finite collection of multi-tiles and let ${\bf R}^{*}\subset{\bf R} $ be a finite collection of multi-rectangles such that each $p\in\p^{*}$ belongs to some $R\in {\bf R}^{*}$. We assume that the collection ${\bf R}^{*}$ of multi-rectangles
is $M$-separated into scales with $M=A^{\frac{2}{N-2}}$. 
We also assume that 
$
\size_{j,i}(\p^{*})\leq \lambda
$.
Then there is a collection $\F$ of trees $(\T,\p_\T)$ with $\p_{\T}\subset\p^{*}$ such that
$$
\sum_{\T\in \F}|I_\T|\lesssim A^{\frac{4}{N-2}}\lambda^{-2}\| f\|_2^2
$$
$$\size_{j,i}\Big(\p^{*}-\bigcup_{\T\in\F}\p_{\T}\Big)\leq \lambda/2.$$
\end{lemma}
\begin{proof}
We first describe the proof in the case we deal with tree sizes ($i\not=j$), which practically also contains the proof for the much simpler case of tile sizes. This will be briefly mentioned in the end of the proof.

First we eliminate all the multi-rectangles which have big $\size_{j,i}$. We do that in order to set the stage for an application of Lemma \ref{almostortho3}. Consider the collection $\F_0$ of all singleton $i$-trees $(R,\p^{*}(R,\xi))$ with top $(I_R,\xi)$, for some $R\in {\bf R}^{*}$, $\xi\in\R$ and some $\p^{*}(R,\xi)\subseteq\p^{*}(R)$, such that 
$$S^j_{R,\p^{*}(R,\xi)}(f)>10^{-2}M^{-1/2}\lambda|I_R|^{1/2}.$$
Note that a given $R$ may appear in more than just one pair $(R,\p^{*}(R,\xi))$. Define $\F_{00}$ to be the collection of all $R$ that contribute to $\F_0$, that is $(R,\p^{*}(R,\xi_R^i))\in \F_0$ for some $\xi_R^i$.

Start with $\F_1:=\emptyset$ and perform the following algorithm. Select some $R\in\F_{00}$, with the additional property that $I_R$ is maximal with respect to inclusion,  among all such $R$. It does not matter which one is selected, if there is more than one $R$ that qualifies to be selected. Set $\F_1:=\F_1\cup \{R\}$. 

Define the vector $\vec{\xi}_R\in\langle\gamma\rangle$ as the one uniquely determined by the coordinate $\xi_R^i$: if $i\not=0$, then $\xi_R^i$ is the $i^{th}$ coordinate of this vector, while if $i=0$, it is either the first or the second coordinate, depending on whether $(R,\p^{*}(R,\xi))$ is a $0^1$-tree or a $0^2$-tree.

Define $\F_R:=\{R'\in\F_{00}\;:R_j'\cap R_j\not=\emptyset\}$. Set $\F_{00}:=\F_{00}\setminus \F_R$ and restart the algorithm.

Due to the maximality of $I_R$ it follows that $I_{R'}\subseteq I_R$ for each $R\in \F_R$. Let us observe next that
$$\bigcup_{R'\in \F_R}\p^{*}(R')\subseteq \S(I_R,\vec{\xi_R}).$$
This follows from Lemma \ref{lljjggrr190gr} if $R'=R$, and from Lemma \ref{lljjggrr190gr1} if $R'\not=R$. From \eqref{theahamoment} we deduce that $\bigcup_{R'\in \F_R}\p^{*}(R')$ can be split into the union of five trees, each of which with top interval $I_R$.

It suffices to prove that
$$
\sum_{R\in \F_{00}}|I_R|\lesssim A^{\frac{2}{N-2}}M\lambda^{-2}\| f\|_2^2.
$$
But this follows immediately from Lemma \ref{almostortho}, since the rectangles $(R_j)_{R\in \F_{00}}$ are pairwise disjoint. This ends the first stage of the construction.

In the second stage of the construction we perform the following algorithm
\\
Step 0: Initialize $\p_{0}:=\p^{*}\setminus \bigcup_{R\in\F_1}\S(I_R,\vec{\xi_R})$, $\F_i^{(1)}=\emptyset$, $\S_{{\rm aux}}^{(1)}=\emptyset$, $\F_i^{(2)}=\emptyset$, $\S_{{\rm aux}}^{(2)}=\emptyset$.
\\
Step 1: Select an $i$-tree $(\T,\p_\T)$ with top $(I_\T,\xi_\T^i)$, $\T\subseteq {\bf R}^{*}$ and $\p_\T\subseteq\p_{0}$ such that the following requirements are satisfied

(i) $(\T,\p_\T)$ is the maximal tree with the given top $(I_\T,\xi_\T^i)$ that can be constructed out of the multi-tiles that are available.

(ii) $\sum_{R\in \T}S_{R,\T}^j(f)^{2}\ge \frac{1}{8}\lambda^2 |I_{\T}|$

(iii) $\xi_\T^{j}<\omega_{p_j}$ for each\footnote{Recall that $\vec{\xi}_\T\in\langle\gamma\rangle$ is the vector uniquely determined by the coordinate $\xi_\T^{i}$, as in Stage 1.}  $p\in \p_\T$

(iv) $\xi_\T^{j}$ is maximal over all the trees that satisfy (i), (ii) and (iii) above.
\\
If no such tree can be found then go to Step 5. 
\\
Step 2: Put the tree $(\T,\p_\T)$ in the collection $\F_i^{(1)}$ and the multi-tiles $\S(I_\T,\vec{\xi}_\T)$ in $\S_{{\rm aux}}^{(1)}$
\\
Step 3: Upgrade $\p_{0}:=\p_{0}\setminus \S(I_\T,\vec{\xi}_\T)$
\\
Step 4: Go to Step 1
\\
Step 5: Select an $i$-tree $(\T,\p_\T)$ with top $(I_\T,\xi_\T^i)$, $\T\subseteq {\bf R}^{*}$ and $\p_\T\subseteq\p_{0}$ such that the following requirements are satisfied

(i) $(\T,\p_\T)$ is the maximal tree with the given top $(I_\T,\xi_\T^i)$ that can be constructed out of the multi-tiles that are available.

(ii) $\sum_{R\in \T}S_{R,\T}^j(f)^{2}\ge \frac{1}{8}\lambda^2 |I_{\T}|$

(iii) $\xi_\T^{j}>\omega_{p_j}$ for each $p\in \p_\T$

(iv) $\xi_\T^{j}$ is minimal over all the trees that satisfy (i), (ii) and (iii) above.
\\
If no such tree can be found then go to Step 9. 
\\
Step 6:  Put the tree $(\T,\p_\T)$ in the collection $\F_i^{(2)}$ and the multi-tiles $\S(I_\T,\vec{\xi}_\T)$ in $\S_{{\rm aux}}^{(2)}$
\\
Step 7: Upgrade $\p_{0}:=\p_{0}\setminus \S_{{\rm aux}}^{(2)}$
\\
Step 8: Go to Step 5
\\
Step 9: Stop. The algorithm is over.

The first part of Lemma \ref{lem.sumar322w} easily implies that if $\p_0$ denotes the value after the algorithm above ends, then $\size_{j,i}(\p_0)\le \lambda/2$. It suffices now to prove that 
$$
\sum_{\T\in \F^{(1)}_i\cup\F^{(2)}_i}|I_\T|\lesssim A^{2}M^{3-N}\lambda^{-2}\| f\|_2^2.
$$
This will follow from Lemma \ref{almostortho3} once we prove that both $\F^{(1)}_i$ and $\F^{(2)}_i$ consist of $j$-strongly disjoint trees. It suffices to prove this for $\F^{(1)}_i$.

We verify the second requirement of $j$-strongly disjointness.
Assume for contradiction that there are two distinct trees $(\T, \p_\T),(\T',\p_{\T'})\in \F^{(1)}_i$ and $R\in \T$, $R'\in \T'$, $p\in\p(R,\T)$, $p'\in\p(R',\T')$ with $\omega_{p_j}\subsetneq \omega_{p_j'}$ and $I_{R'}\subseteq I_\T$.  By using Lemma \ref{lem.sumar322w}, \eqref{sep00fscales} and (iii) in Step 1 of the construction of both $(\T, \p_\T)$ and $(\T',\p_{\T'})$, we get that $\xi_{\T}^j>\xi_{\T'}^{j}$. Thus, by (iv) in Step 1 we know $(\T, \p_\T)$ was selected before $(\T',\p_{\T'})$.

On the other hand, by the grid properties  we know that $\omega_{R_j}\subsetneq \omega_{R_j'}$. We can then invoke Lemma \ref{lljjggrr190gr1} with $I:=I_\T$ to conclude that $p'\in\S(I_\T,\vec{\xi}_\T)$. But then it is clear that $p'$ was eliminated before the selection of the tree $\T'$, giving rise to a contradiction.

We now verify the first requirement in the definition of $j$-strongly disjointness. Assume by contradiction that $R\in \T\cap \T'$ for some $\T,\T'\in\F^{(1)}_i$, and assume without loss of generality that $\T$ was selected before $\T'$. But then, by Lemma \ref{lljjggrr190gr} with $I=I_\T$ it follows that $\p(R)\subseteq\S(I_\T,\vec{\xi}_\T)$. This means the whole $\p(R)$ was eliminated before the selection of the tree $\T'$, giving rise to a contradiction.

Note that the first stage of this proof is essentially what needs to be done when dealing with tile sizes ($i=j$). Precisely, at each step of the selection algorithm we  search for singleton $j$-trees $(R,\p^{*}(R,\xi_R^i))$ with top $(I_R,\xi_R^i)$ such that 
$$S^j_{R,\p^{*}(R,\xi_R^i)}\ge \lambda|I_R|^{1/2},$$
and such that $I_R$ is maximal with respect to inclusion. 
We eliminate the multi-tiles $\S(I_R,\vec{\xi_R})$ from $\p^{*}$. Let $\F_{00}$ be the collection of the selected multi-rectangles $R$. The fact that
$$
\sum_{R\in \F_{00}}|I_R|\lesssim \lambda^{-2}\| f\|_2^2
$$
is an immediate consequence of Lemma \ref{almostortho} and of the fact that the rectangles $R_i$ with $R\in\F_{00}$ are pairwise disjoint.

As a final observation, we note that -due to \eqref{theahamoment}- at each stage in the Peeling Lemma we eliminate with each multi-tile $p\in\p(R)$ all the multi-tiles in $\p(R)$.
\end{proof}

\subsection{Size estimates}
\label{sec6.4}

In this section we will see how to estimate various sizes. Before we do so, we recall two lemmata that will be used in the sequel with the words `quartiles' or `tri-tiles' replaced by `multi-tiles'. \\

{\bf Lemma A } (Lemma 4.2 in  \cite{MTT4} pg. 410)

{\it Let $\P$ be a finite collection of quartiles, $j=1, 2, 3$ and let $\{a_{P_j}\}_{P\in \P}$ be a sequence of complex numbers. Then 

$$ \sup_{T \in \P} ( \frac{1}{|I_T|} \sum_{P \in T} |a_{P_j}|^2 )^{1/2} \sim  \rm{sup}_{T \in \P} \, \frac{1}{|I_T|}  \Vert   (\sum_{P\in T} |a_{P_j}|^2  \frac{1_{I_P}}{|I_P|}\, )^{1/2} \Vert_{L^{1, \infty}(I_T)} $$
where $T$ ranges over all trees in $\P$ which are $i$-trees for some $i \ne j$. }\\

{\bf Lemma B } (Lemma 6.8 in \cite{MTT3} pg. 443)

{\it Let $j=1, 2, 3$, $E_j$ be a set of finite measure, $f_j$ be a function in $X(E_j)$, and let $\P$ be a finite collection of tri-tiles. Then we have 
$$ \sup_{T \in \P} ( \frac{1}{|I_T|} \sum_{P \in T} |\langle f_j , \phi_{P_j} \rangle |^2 )^{1/2} \, \lesssim \,  \rm{sup}_{P \in \P} \, \frac{\int_{E_j} \chi_{I_P}^{M}}{ |I_P|} $$ for all $M$ with implicit constant depending on $M$. }\\

We can now state our lemmas.
\begin{lemma}
Let $f\in X(E)$ and $1<p<\infty$.
For each $R\in {\bf R}$,  $\xi\in \R$ and $l$ we have
$$\left({|I_{R}|^{-1}}{\sum_{p_l\in \p_l(R)\atop{\xi\in\omega_{p_l}}}|\<f,\phi_{p_l}\>|^2}\right)^{1/2}\lesssim \frac{1}{|I_R|^{1/p}}\|1_E\chi_{I_R}^{N-1}\|_p.$$
For each $i\in\{0,1,2,3\}\setminus\{j\}$ and each $i$-tree $(\T,\p_\T)$ with top $(I_\T,\xi_\T^i)$ we have 
\begin{equation}
\label{eq:diffbound}
\left(\frac{\sum_{R\in\T}\sum_{p_j\in \p_j(R,\T)}|\<f,\phi_{p_j}\>|^2}{|I_{\T}|}\right)^{1/2}\le C_A\sup_{I\in\I_\T}\frac{1}{|I|^{1/p}}\|1_E\chi_{I}^{N-1}\|_p,
\end{equation}
where 
$$\I_\T=\{I\;\operatorname{dyadic: }I_R\subseteq I\subseteq I_{\T}\hbox{ for some }R\in\T\},$$
$C_A=O(A^{1/2})$ if $i\in\{0,3\}$ and $j\in\{1,2\}$, and $C_A=O(1)$ if $i\in\{1,2\}$.
\end{lemma}

\begin{proof}
This is a version of Lemma B above. Here we prove only the cases that are a bit different from the case in Lemma B. Namely  we will only prove the second part of our lemma, and only in the case when $i=3$ and $(i=0,j\in\{1,2\})$. These are the ``worst case scenarios'' because of the presence of $A$ in the corresponding type of lacunarity.

Define $a_{R_j}=(\sum_{p_j\in \p_j(R,\T)}|\<f,\phi_{p_j}\>|^2)^{1/2}$. We first focus on the case $i=3$. We note that
\begin{align*}
a_{R_j}^2&\le A\sup_{p\in\p(R,\T)} |\<f,\phi_{p_j}\>|^2\\&=A\sup_{p\in\p(R,\T)} |\langle f*\check{1}_{\omega_{R_j}},\phi_{p_j}\rangle|^2\\&\lesssim A|I_R|\inf_{x\in I_R}M_1^2(f*\check{1}_{\omega_{R_j}})(x).
\end{align*}
By Lemma A above, it suffices\footnote{This reduction is made possible by and explains the presence of the collection $\I_\T$ in \eqref{eq:diffbound}.} to show that for each subtree $\T'$ of $T$
\begin{equation}
\label{pagh77ntz1}
\|\left(\sum_{R\in\T'}a_{R_j}^2\frac{\chi_{I_R}}{|I_R|}\right)^{1/2}\|_{p}\lesssim A^{1/2}\|1_E\chi_{I_{\T'}}^{N-1}\|_p.
\end{equation}
To simplify notation, we will continue to write $\T$ rather than $\T'$.
and by writing $f=f_1+f_2$ with $f_1=f1_{2I_\T}$, it further suffices to prove \eqref{pagh77ntz1} for both $f_1$ and $f_2$. 

In the case of $f_2$ we use the decay of $\phi_{p_j}$ to write 
$$a_{R_j}\lesssim A^{1/2}\left(\frac{|I_R|}{|I_\T|}\right)^{N-\frac12}|I_\T|^{-1/2}\int_E\chi_{I_\T}^{N},$$
which by  summation and H\"older's inequality proves \eqref{pagh77ntz1}.

To deal with $f_1$ we apply the Fefferman-Stein inequality first and then invoke the Littlewood-Paley theory and the rank property (4) to get
\begin{align*}
\|\left(\sum_{R\in\T}a_{R_j}^2\frac{\chi_{I_R}}{|I_R|}\right)^{1/2}\|_{p}&\lesssim A^{1/2}\|(\sum_{R\in\T}M_1^2(f_1*\check{1}_{\omega_{R_j}})(x))^{1/2}\|_p\\&\lesssim A^{1/2}\|(\sum_{R\in\T}(f_1*\check{1}_{\omega_{R_j}})^2(x))^{1/2}\|_p\\&\lesssim A^{1/2}\|1_E\|_p,
\end{align*}
where $M_1(f)$ denotes the Hardy-Littlewood maximal function of $f$. 

Let us now briefly see the case $i=0$. The argument is very similar to above. Denote by $\xi_\T^j$ the $j^{th}$ component of the vector $\vec{\xi}_\T$ associated with $\xi_\T^i$ as before. By modulation symmetry it suffices to assume that $\xi_\T^j=0$. For each $R\in\T$ and each $p\in\p(R,\T)$ let $\omega(p)$ be an interval of the form $[\epsilon 2^k,\epsilon 2^{k+3}]$ ($\epsilon\in\{-1,1\}$, $k\in\Z$) such that $\omega_{p_j}\subset \omega(p)$. This is possible due to Lemma \ref{lem.sumar322w}. We will estimate as before the term corresponding to $f_2$, and then write for $f_1$ 
$$a_{R_j}^2\lesssim A \sup_{p\in\p(R)}|I_R|\inf_{x\in I_R}M_1^2(f_1*\check{1}_{\omega(p)})(x).$$
By Lemma \ref{lem.sumar322w} and \eqref{sep00fscales} we know that for each interval $\omega:=[\epsilon2^k,\epsilon2^{k+3}]$ there is at most one scale $|I_R|$ such that $\omega(p)=\omega$. The proof follows as in the previous case, by applying the Fefferman-Stein inequality and the Littlewood-Paley theory.
\end{proof}

\begin{lemma}
For each $R\in {\bf R}$, $\xi\in \R$ and $l$  we have
$$\left({|I_{R}|^{-1}}{\sum_{p_l\in \p_l(R)\atop{\xi\in\omega_{p_l}}}|\<f,\phi_{p_l}\>|^2}\right)^{1/2}\lesssim \frac{1}{|I_R|^{1/2}}\|f\chi_{I_R}^{N-2}\|_2.$$
For each $i$-tree $(\T,\p_\T)$ with $i\in\{0,1,2,3\}\setminus\{j\}$  we have 
$$\left(\frac{\sum_{R\in\T}\sum_{p_j\in \p_j(R,\T)}|\<f,\phi_{p_j}\>|^2}{|I_{\T}|}\right)^{1/2}\lesssim \frac{1}{|I_\T|^{1/2}}\|f\chi_{I_\T}^{N-2}\|_2.$$
\end{lemma}
\begin{proof}
Use  \eqref{whs73219872} and  \eqref{whs73219872hhhyt} with $f:=f\chi_{I_\T}^{N-2}$ and $\phi_{p_j}:=\phi_{p_j}\chi_{I_\T}^{-N+2}$, and note that  $(M_{-c(\omega_{p_j})}\phi_{p_j})\chi_{I_\T}^{-N+2}$ is $L^2$ adapted of order 2 to $I_{p_j}$ and has the same frequency support as  $M_{-c(\omega_{p_j})}\phi_{p_j}$. 
\end{proof}

By interpolating between the previous two lemmas we get 
\begin{corollary}
\label{treft690ksahdjhsa}
Let $f\in X(E)$ and $1<p\le 2$ and $\epsilon>0$.
For each $R\in {\bf R}$,  $\xi\in \R$  we have
\begin{align*}
\left({{|I_{R}|^{-1}}\sum_{p_l\in \p_l(R)\atop{\xi\in\omega_{p_l}}}|\<f,\phi_{p_l}\>|^2}\right)^{1/2}&\lesssim \frac{1}{|I_R|^{1/p}}\|1_E\chi_{I_R}^{N-2}\|_p\\&\lesssim (\sup_{x\in E}\chi_{I_R}(x))^{\frac{N-4}{p}}(\inf_{x\in I_R}M_1(1_E)(x))^{1/p}.
\end{align*}
For each $i$-tree $(\T,\p_\T)$ with $i\in\{0,1,2,3\}\setminus\{j\}$  we have 
\begin{align*}
\left(\frac{\sum_{R\in\T}\sum_{p_j\in \p_j(R,\T)}|\<f,\phi_{p_j}\>|^2}{|I_{\T}|}\right)^{1/2}&\lesssim C\sup_{I\in\I_\T}\frac{1}{|I|^{1/p}}\|1_E\chi_{I}^{N-2}\|_p\\&\lesssim C(\sup_{I\in\I_\T\atop{x\in E}}\chi_{I}(x))^{\frac{N-4}{p}}(\sup_{I\in\I_\T}\inf_{x\in I}M_1(1_E)(x))^{1/p},\end{align*}
where $C=A^{\epsilon+\frac1p-\frac12}$ if $i\in\{0,3\}$ and $j\in\{1,2\}$, and $C=1$ if $i\in\{1,2\}$.
\end{corollary}

\subsection{Proof of Theorem \ref{mainmodels}}
\label{sec6.6}
We may assume ${\bf R}$ is finite, and get bounds independent of ${\bf R}$.
At the expense of losing a factor of $A^{O(\frac1N)}$ in the bounds, it suffices to assume that the collection ${\bf R}$ of multi-rectangles
is $M$-separated into scales with $M=A^{\frac{2}{N-2}}$. Since $N$ can be taken arbitrarily large, all the factors of the form $A^{O(\frac1N)}$ contributing to various bounds may and will be tolerated.

By scaling invariance we may assume that $|E_{j_0}|=1$. Define 
$$\Omega=\bigcup_{j=1}^{3}\{x: M_1(1_{E_j})(x)>100|E_j|\}$$ 
and $\tilde{E}_{j_0}:=E_{j_0}\setminus \Omega$, and note that $|\tilde{E}_{j_0}|>\frac12|{E}_{j_0}|$.

Let $f_j\in X_2(E_j)$, $j\not=j_0$ and $f_{j_0}\in X_2(\tilde{E}_{j_0})$. All sizes $\size_{j,i}$ are understood with respect to $f_j$. We need to show that for some $\delta"<\delta$
\begin{equation}
\label{finalkretgh56}
|\Lambda(f_1,f_2,f_3)|\lesssim A^{\delta"}|E_1|^{\gamma_1}|E_2|^{\gamma_2}|E_3|^{\gamma_3},
\end{equation}
where $\gamma_i:=\alpha_i-\frac12.$ We note that due to our restrictions, we have that $0<\gamma_{j_1},\gamma_{j_2}<\frac12$ and $\gamma_{j_1}+\gamma_{j_2}=-\alpha_{j_0}<\frac12$. Thus we can find $0<\beta_1,\beta_2<\frac12$ sufficiently close to $\frac12$ such that
$a_{j_1}:=1-\frac{\gamma_{j_1}}{\beta_1}>0$, $a_{j_2}:=1-\frac{\gamma_{j_2}}{\beta_2}>0$ and $a_{j_0}:=2-a_{j_1}-a_{j_2}\in (0,1).$

We shall make the assumption that either 

(1) $I_R \cap (\R\setminus\Omega)\not=\emptyset$ for all\footnote{The collection of multi-tiles is also appropriately restricted.}  $R\in {\bf R}$, or

(2) $I_R=I$ for all  $R\in {\bf R}$, for some (fixed) dyadic $I\subset \Omega$ with 
\begin{equation}
\label{eq:1297}
2^l< 1+\frac{\text{dist}(I,\R\setminus\Omega)}{|I|}\le 2^{l+1},\hbox{ for some (fixed) $l\ge 0$},
\end{equation}
and prove \eqref{finalkretgh56} for both case (1) and (2), with an additional  multiplicative factor of $2^{-l}|I|$ in the bound, in case (2). If we can prove these special cases with the indicated gain, the general case  follows by summation in $l$ and $I$, since $|\Omega|\lesssim 1.$ We present the argument for case (1), and then will indicate how to modify it for case (2).

Define $p_{j_1},p_{j_2}\in (1,2)$ such that $\beta_1=\frac1{p_{j_1}}-\frac12$, $\beta_2=\frac1{p_{j_2}}-\frac12$ and define $p_{j_0}=2$.
Note that  by Corollary ~\ref{treft690ksahdjhsa} it easily follows that for each $\epsilon>0$
\begin{equation}
\label{eq:11qwlfgt1}
\max_{i}\size_{j_1,i}({\p})\lesssim |E_{j_1}|^{\beta_1}A^{\beta_1+\epsilon}
\end{equation}
\begin{equation}
\label{eq:11qwlfgt2}
\max_{i}\size_{j_2,i}({\p})\lesssim |E_{j_2}|^{\beta_2}A^{\beta_2+\epsilon}
\end{equation}
\begin{equation}
\label{eq:11qwlfgt3}
\max_{i}\size_{j_0,i}({ \p})\lesssim 1.
\end{equation}
Fix $j\in\{1,2,3\}$.
We successively use the Pealing Lemma \ref{pealinglemma}, simultaneously for each $i\in\{0,1,2,3\}$, to decompose $\p=\bigcup_{k=-\infty}^{\max_i\size_{j,i}(\p)}\p_{k}^{(j)}$ such that $\p_k^{(j)}$ consists of the union of a family $\F_k^{(j)}$ of trees $(\T,\p_\T)$, $\p_k^{(j)}=\bigcup_{\T\in\F_k^{(j)}}\p_{\T}$, satisfying
\begin{equation}
\label{19koskmiuiwkew6711}
\max_{i}\size_{j,i}(\p_k^{(j)})\lesssim 2^{k}
\end{equation}
\begin{equation}
\label{19koskmiuiwkew67}
\sum_{\T\in \F_k^{(j)}}|I_\T|\lesssim A^{O(\frac1N)}2^{-2k}.
\end{equation}

We get
$$|\Lambda(f_1,f_2,f_3)|\le \sum_{k_1,k_2,k_3}\sum_{p\in \p^{(1)}_{k_1}\cap \p^{(2)}_{k_2}\cap\p^{(3)}_{k_3}}|I_R|^{-1}|I_{p_3}|^{1/2}\prod_{i=1}^{3}|\langle f_j,\phi_{p_j}\rangle|,$$ where we implicitly assume that 
$$2^{k_j}\le \max_{i}\size_{j,i}(\p).$$

By symmetry we may restrict ourselves to the case $k_{j_*}=\max_{j}k_j$, for some $j_*\in\{1,2,3\}$. 
We can further estimate the sum above by 
\begin{equation}
\label{eq:15}
\sum_{k_1,k_2,k_3}\sum_{(\T,\p_\T)\in \F_{k_{j_*}}^{(j_*)}}\sum_{p\in \tilde{\p}_\T}|I_R|^{-1}|I_{p_3}|^{1/2}\prod_{i=1}^{3}|\langle f_j,\phi_{p_j}\rangle|,
\end{equation}
where $\tilde{\p}_\T:=\p_\T\cap \p_{k_{j}}^{(j)}\cap\p_{k_{j'}}^{(j')}$, $j,j'\in\{1,2,3\}\setminus j_*$. Note that $\T$ gets partitioned by the intersection with various trees from $\F_{k_{j_*^1}}^{(j_*^1)}$ and $\F_{k_{j_*^2}}^{(j_*^2)}$, where $\{j_*^1,j_*^2\}=\{1,2,3\}\setminus j_*$. By the final observation in the proof of the Peeling Lemma, two different such trees in some $\F_{k_{j_*^1}}^{(j_*^l)}$, $l\in\{1,2\}$, will not share any multi-rectangle. It follows that we have the following natural partition
$$\T=\bigcup_{r\in R(\T)}\T_r$$
$$\tilde{\p}_\T=\bigcup_{r\in R(\T)} \p_{\T_r}$$
where $(\T_r,\p_{\T_r})$ is a subtree of $(\T,\p_{\T})$ which arises by an intersection, as described above. Due to the elimination of the saturations $\S(I_{\T'},\vec{\xi}_{\T'})$ in Step 2 and Step 6 of the algorithm in the Peeling Lemma (here $\T'$ is a generic tree in  $\F_{k_{j_*^1}}^{(j_*^l)}$, $l\in\{1,2\}$), it easily follows that the trees $(\T_r,\p_{\T_r})$ can be assigned tops $I_{\T_r}\subseteq I_\T$ which are pairwise disjoint for $r\in R(\T)$. In particular,
$$\sum_{r\in R(\T)}|I_{\T_r}|\le |I_\T|.$$  

We also note that each subtree $(\T_r,\p_{\T_r})$ satisfies $\max_i\size_{j,i}(\p_{\T_r})\lesssim 2^{k_j}$ for each $j\in\{1,2,3\}$. Using these observations and then invoking \eqref{orgtreesizes} and \eqref{19koskmiuiwkew67} with $j=j_*$, we may estimate ~\eqref{eq:15} by

$$
\sum_{k_1,k_2,k_3}\sum_{(\T,\p_\T)\in \F_{k_{j_*}}^{(j_*)}}\sum_{r\in R(\T)}\sum_{p\in \p_{\T_r}}|I_R|^{-1}|I_{p_3}|^{1/2}\prod_{i=1}^{3}|\langle f_j,\phi_{p_j}\rangle|\le
\sum_{k_1,k_2,k_3}\sum_{(\T,\p_\T)\in \F_{k_{j_*}}^{(j_*)}}2^{k_1+k_2+k_3}|I_\T|$$
$$\le \sum_{k_1,k_2,k_3}\sum_{(\T,\p_\T)\in \F_{k_{j_*}}^{(j_*)}}2^{k_1+k_2+k_3}2^{-2k_{j_*}}\le \sum_{k_1,k_2,k_3}2^{k_1(1-a_1)}2^{k_2(1-a_2)}2^{k_3(1-a_3)},$$
where $a_i$ have been defined in the beginning of the argument.
Finally, by invoking \eqref{eq:11qwlfgt1} - \eqref{eq:11qwlfgt3} we can further estimate the above by 
$$\sum_{k_1,k_2,k_3\le 0}2^{k_1(1-a_1)}2^{k_2(1-a_2)}2^{k_3(1-a_3)}(|E_{j_1}|^{\beta_1}A^{\beta_1+\epsilon})^{1-a_1}(|E_{j_2}|^{\beta_2}A^{\beta_2+\epsilon})^{1-a_2}$$
$$\lesssim |E_{j_1}|^{\gamma_1}|E_{j_2}|^{\gamma_2}A^{\gamma_1+\gamma_2+\epsilon a_{j_0}}\lesssim |E_{j_1}|^{\gamma_1}|E_{j_2}|^{\gamma_2}A^{\epsilon a_{j_0}-\alpha_{j_0}}.$$

This ends the proof in case (1), since $\alpha_{j_0}>-\delta$. To deal with case (2) we make the following modifications. Redefine $f_j:=f_j\chi_{I}^{2}$ and $\phi_{p_j}:=\phi_{p_j}\chi_{I}^{-2}$, and note that the new functions $f_j$ have the same properties as the old ones. Moreover $(M_{-c(\omega_{p_j})}\phi_{p_j})\chi_{I_\T}^{-2}$ is $L^2$ adapted of order $N-2$ to $I_{p_j}$ and has the same frequency support as  $M_{-c(\omega_{p_j})}\phi_{p_j}$. Then run the same argument as in case (1). As we said earlier, we expect the presence of $|I|$ and $2^{-l}$ in the bound for \eqref{finalkretgh56}. The presence of $|I|$ is explained by the localized estimate
$$\|f_j\|_2^2\lesssim |I|\inf_{x\in I}|E_j|^{-1}M_1(1_{E_j})(x)$$
which becomes effective in the application of the Pealing Lemma in \eqref{19koskmiuiwkew67}.
The decay in $l$ is due to the classical estimates
$$\inf_{x\in I}M_1(1_{E_j})(x)\lesssim 2^{l}\inf_{x\in 2^{l+1}I}M_1(1_{E_j})(x)\lesssim 2^l,$$
$$\sup_{x\in \tilde{E}_{j_0}}\chi_I(x)\lesssim 2^{-l},$$
which become effective in the application of Corollary \ref{treft690ksahdjhsa} in estimating sizes.


\section{The general case: modulation invariant paraproducts}
\label{sec7}
Given any function $f$ in $\rm BMO$,
we shall construct a trilinear form $\Lambda$
satisfying the bounds in the conclusion of 
Theorem \ref{main} (and thus also its assumptions) and also
$$\Lambda(1,1,.)=f$$
$$\Lambda(1,.,1)=0$$
$$\Lambda(.,1,1)=0$$
Such form and its symmetric
counterparts under permutation of the three arguments will be called modulation invariant paraproducts.
By subtracting three paraproducts we can reduce every trilinear
Calderón-Zygmund form as in Theorem \ref{main} to another one
satisfying the special cancellation conditions $\Lambda(1,1,.)=\Lambda(1,.,1)=\Lambda(.,1,1)=0$.
Thus the construction of the paraproduct
will finish the proof of Theorem \ref{main}.

Define $\phi_1,\phi_2,\phi_3$ such
that $\widehat {\phi_1}=\widehat {\phi_2}=\Phi$ and $\widehat {\phi_3}=\tau_{10}\Phi+\tau_{-10}\Phi$, where $\Phi$ was introduced in \eqref{Phi}. 
Consider
$$\psi(x) = \int
\phi_1(x+(\beta_1-\beta_3)t)\phi_2(x+(\beta_2-\beta_3)t)
\phi_3(x)\, dt.$$
By using Fourier transforms it follows immediately that $\psi$ is a non zero Schwartz function.
The function
\begin{equation}\label{integrand}
\phi_1(x+(\beta_1-\beta_3)t)\phi_2(x+(\beta_2-\beta_3)t)\phi_3(x)
\end{equation}
is easily seen to have zero integral in $x$ for each $t$. 
Hence $\psi$ itself has mean zero.

By Calderón's reproducing formula, we have (with some constant $c$)
$$f= c \int_0^{+\infty}  f * \psi_t * \psi_t  \, \frac{dt}{t}$$
This well-know formula in $L^2$ also extends to distributions modulo polynomials provided
$\widehat \psi$ vanishes in a neighborhood of the origin. In particular, for $f$ in BMO
the formula holds in the sense
$$f=\lim_{\epsilon\to 0}c\int_{\epsilon}^{1/\epsilon}f * \psi_t * \psi_t  \, \frac{dt}{t}$$
at least in the distributional sense when tested against bump functions with mean zero.

Defining
$$\psi_{k,n}(x):=2^{-k/2}\psi(2^{-k} x- n),$$
a simple change of coordinates gives (with a new constant $c$)
\begin{equation}
\label{BMOdistrCRF}
f=\lim_{\kappa\to +\infty}c\int_{|k|\le \kappa}\int_\R \< f,\psi_{k,n}\> \psi_{k,n} \, dk \, dn
\end{equation}
with the equality holding in the sense of distributions when tested again functions with mean zero.

If we define $\phi_{i,k,n}$ by translation and dilation in the analogous
manner, then we have
$$\psi_{k,n}(x) = \int
\phi_{1,k,n}(x+(\beta_1-\beta_3)t)\phi_{2,k,n}(x+(\beta_2-\beta_3)t)
\phi_{3,k,n}(x)\, dt$$
Set
$$c_{k,n}:= c \<f,\psi_{k,n}\>$$
Then for each $\kappa\in\R$ and each $f_i\in\S(\R)$ we define

$$\Lambda_{\kappa}(f_1,f_2,f_3)=
\int_{|k|\le\kappa} \int_\R \int_\R \int_\R
c_{k,n} \left[\prod_{i=1}^3 f_i(x+\beta_i t) \phi_{i,k,n}(x+\beta_i t)\right]
\, dx \, dt \, dn \, dk$$
\begin{equation}
\label{intreprtrunc}
=\int_\R \int_\R\prod_{i=1}^3 f_i(x+\beta_i t)K_{\kappa}(x,t)dxdt,
\end{equation}
where 
$$K_{\kappa}(x,t)=\int_{|k|\le\kappa} \int_\R c_{k,n}\prod_{i=1}^3\phi_{i,k,n}(x+\beta_i t)\, dn \, dk.$$
Since $f$ is in BMO, we have that $|c_{k,n}|\lesssim 2^{k/2}$ (this is a particular instance of \eqref{Carcondhh567}). By discretizing the integral representing $K_{\kappa}$, we can write it as an average over $[0,1]^2$ of sums of the form
$$\tilde{K_{\kappa}}(x,t)=\sum_{2^{-\kappa}\le |I|\le 2^{\kappa}\atop{I\;\operatorname{dyadic}}}c_I\prod_{i=1}^3\phi_{i,I}(x+\beta_i t),$$
with $\phi_{i,I}$ being $L^2$-adapted to $I$ and $|c_I|\lesssim |I|^{1/2}.$ It is now an easy exercise to conclude that each $\tilde{K_{\kappa}}$ satisfies \eqref{e.e1Cal1} and \eqref{e.e1Cal2} with $\delta=1$ and with uniform constants $C_{\tilde{K_{\kappa}}}=O(1)$. Moreover, $\tilde{K_{\kappa}}$ is locally integrable since in particular $|\tilde{K_{\kappa}}(x,t)|\lesssim \min\{2^{\kappa},|t|^{-1}\}$. These facts will easily prove that the integral in \eqref{intreprtrunc} is indeed convergent for arbitrary Schwartz functions.

Let now $f_1,f_2,f_3$ be three bump functions $L^2$- adapted of order 2 to some interval $J$. We next prove an inequality that will justify some of our claims. Note first that for each $x$, $f_2(x+(\beta_2-\beta_1)t)f_3(x+(\beta_3-\beta_1)t)$ is $L^2$- adapted of order 2 (as a function of $t$) to some interval of size similar to $|J|$ with implicit constant $O(|J|^{-1/2})$. Similarly, $\phi_{2,I}(x+(\beta_2-\beta_1)t)\phi_{3,I}(x+(\beta_3-\beta_1)t)$ is $L^2$- adapted of order 2  to some interval of size similar to $|I|$ with implicit constant $O(|I|^{-1/2})$. Moreover, this latter function also has mean zero with respect to $t$. By applying \eqref{02} to these functions and \eqref{01} to $f_1$ and $\phi_{1,I}$ we conclude that
$$
\sum_{I\; \operatorname{dyadic}}|c_I|\left|\int_\R \int_\R
\left[\prod_{i=1}^3 f_i(x+\beta_i t) \phi_{i,I}(x+\beta_i t)\right]
\, dx \, dt\right|$$
$$
=\sum_{I\; \operatorname{ dyadic}}|c_I|\left|\int_\R f_1(x) \phi_{I,1}(x)\int_\R
\left[\prod_{i=2}^3 f_i(x+(\beta_i-\beta_1) t) \phi_{i,I}(x+(\beta_i-\beta_1) t)\right]\, dt\, dx \right|\lesssim $$
$$\lesssim\sum_{I\; \operatorname{ dyadic}}|c_I||I|^{-1/2}|J|^{-1/2}(1+(\max(|I|,|J|))^{-1}\operatorname{dist}(I,J))^{-4}\min(\left(\frac{|I|}{|J|}\right)^2,\left(\frac{|J|}{|I|}\right)^2)\lesssim |J|^{-1/2}.$$

Define now for each $f_i\in\S(\R)$
$$\Lambda(f_1,f_2,f_3)=\lim_{\kappa\to +\infty}\Lambda_{\kappa}(f_1,f_2,f_3).$$
Since each Schwartz function is adapted to the unit interval centered at the origin, the above computations show that the above limit exists. Moreover, due to the earlier estimates for $\tilde{K_\kappa}$, the form $\Lambda$ is associated with a kernel $K$ satisfying \eqref{e.e1Cal1} and \eqref{e.e1Cal2} with $\delta=1$.

Since the function
$$\phi_1(x)\phi_2(x+(\beta_2-\beta_1)t)\phi_3(x+(\beta_3-\beta_1)t)$$
has mean zero in $t$ for every fixed $x$, it is easy to verify using the kernel representation and the definition of $\Lambda_{\kappa}(.,1,1)$ that $\Lambda_{\kappa}(.,1,1)=0$. Then, by invoking \eqref{L} and \eqref{deferrorb12} we conclude that $\Lambda(.,1,1)=0$. Likewise we see $\Lambda(1,.,1)=0$.

To see that $\Lambda(1,1,.)=f$, we replace
the integration variable $x$ by $y=x+\beta_3 t$, then
execute the integration in $t$ to obtain for each compactly supported $f_3\in\S(\R)$ with mean zero
$$\Lambda_{\kappa}(1,1,f_3)=
\int_{\R} \left[\int_{|k|\le \kappa} \int_\R
c_{k,n} \psi_{k,n}(y)
\, dn \, dk \,\right] f_3(y)dy.$$
By \eqref{BMOdistrCRF}, the limit on the right hand side is $\int_{\R}f(y)f_3(y) dy$, and so we conclude that $\Lambda(1,1,.)$ is $f$ as tempered distributions modulo constants.

It remains to prove that $\Lambda$ is bounded as in the conclusion of Theorem \ref{main}. 
Using that 
$$ \int f|_{\gamma^\perp}\,d\gamma^\perp = \int \widehat f |_{\langle \gamma \rangle} d\langle \gamma \rangle,$$
where $d\gamma^\perp$ and $d\langle \gamma \rangle$ are, respectively, the normalized 
Lebesgue measures on $\gamma^\perp$ and $\langle \gamma \rangle$,
we note that for each $f_i\in\S(\R)$, $\Lambda(f_1,f_2,f_3)$ coincides up to some universal constant with 
$$\lim_{\kappa\to +\infty}\int_{|k|\le \kappa}\int_{\R^2}
c_{k,n} 2^{-k} \left[
\prod_{i=1}^3 \int f_i(x) \phi_{i,k,n}(x)
e^{2\pi i \gamma_i 2^{-k} l x} \, dx
\right]\,dn \, dl\, dk. $$
Hence $\Lambda$ is up to a universal constant an average
of forms of the type
\begin{equation}\label{paramodel}
\sum_{k,n,l\in \Z}
c_{k,n} 2^{-k} \prod_{i=1}^3 \<f_i, {\phi}_{i,k,n,l}\>,
\end{equation}
where for some $k_0,n_0,l_0\in [0,1]$
$${\phi}_{i,k,n,l}(x)=\overline{\phi_{i,k+k_0,n+n_0}(x)}
e^{-2\pi i \gamma_i 2^{-(k+k_0)}(l+l_0) x}$$

The form (\ref{paramodel}) is our basic model form that we wish to
estimate.
The function $\phi_{i,k,n,l}$ is an $L^2$- normalized
bump function adapted to the interval
$$[2^{k+k_0} (n+n_0)) , 2^{k+k_0}(n+n_0+1))$$
By changing the bump function constants mildly, we can assume that
the function is adapted to the dyadic interval
$I_{k,n}= [2^kn,2^k(n+1))$.

The function  $\widehat{\phi}_{i,k,n,l}$
is supported in an interval $\omega_{i,k,l}$ of length
$2^{-k+2}$ and we may assume the following properties
(in case $i=3$ we split the generating function $\phi_3$
into a sum of two generating functions, one with Fourier support
contained in $[-12,-8]$ and the other with
Fourier support contained in $[8,12]$.
Without loss of generality we may replace $\phi_3$ by one of the two):
$$\omega_{1,k,l}= \omega_{2,k,l}$$
for all $k,l$, and if
$$\omega_{i,k,l}\cap \omega_{i,k',l'}\not=\emptyset$$
for some $k,k',l,l'$, then for some universal constants $1\ll c_{2}\ll c_{1}$
$$c_{2} \omega_{j,k,l}\cap c_{2} \omega_{j,k',l'}= \emptyset$$
and
$$c_{1} \omega_{j,k,l}\cap c_{1} \omega_{j,k',l'}\neq \emptyset$$
whenever $i\neq j$ and at least one of $i$ and $j$ is equal
to $3$.
By pigeonholing into finitely many summands if necessary, we also
may assume that
$$\omega_{i,k,l}\cap \omega_{i,k,l'}=\emptyset$$
if $n\neq n'$.
We adopt the usual geometric picture that the parameter tuple $(k,n,l)$
is identified with a triple $p=(p_1,p_2,p_3)$ of tiles
$$p_i= I_p\times \omega_{p,i}= I_{k,n}\times \omega_{i,k,l}$$
Compared to the theory of the bilinear Hilbert transform, the new element here is that two tiles in this
triple are equal. This lack
of separation is offset by better estimates for the coefficients $c_{k,n}$
than in the model forms for the bilinear Hilbert transform. Since $f$ is in
BMO, the coefficients $c_{k,n}$ satisfy a Carleson sequence condition
\begin{equation}
\label{Carcondhh567}
\sum_{I_{k,n}\subset J} |c_{k,n}|^2\le C_f|J|.
\end{equation}

This is what we need to know about the basic model operator (\ref{paramodel})
we have to estimate. The proof runs parallel to the proof of boundedness of the bilinear
Hilbert transform, e.g. in \cite{thielelectures}. The only difference concerns
the estimate on an individual tree for which both $p_1$ and $p_2$ are overlapping
and only $p_{3}$ are disjoint. These trees are estimated by using \eqref{Carcondhh567} as follows:

$$\sum_{p\in \T} |I_p|^{-1} c_{I_p} \prod_{j=1}^3 |\<f_j,\phi_{p_j}\>|$$
$$\le \left(\sum_{p\in T} |c_{I_p}|^2\right)^{1/2}
\left(\sum_{p\in \T} |\<f_3,\phi_{p_3}\>|^2\right)^{1/2}
\prod_{j=1}^2\left( \sup_{p\in \T} \frac{|\<f_j,\phi_{p_j}\>|}{|I_p|^{1/2}}\right)$$
$$\lesssim |I_\T|
\left(\frac 1{|I_\T|}\sum_{p\in \T} |\<f_3,\phi_{p_3}\>|^2\right)^{1/2}
\prod_{j=1}^2 \left(\sup_{p\in \T} \frac{|\<f_j,\phi_{p_j}\>|}{|I_p|^{1/2}}\right).$$
The factors in the last expression are as in \cite{thielelectures} estimated by the
tree sizes defined there. The rest of the proof is identical to the one in
\cite{thielelectures} and one obtains the same bounds as for the
bilinear Hilbert transform.

We close this section by mentioning an interesting application of the constructions we have performed above. For each dyadic interval $I$ with length at least 1, let $c_I$ be a coefficient selected in such a way that $|c_I|\le |I|^{1/2}$ and such that 
\begin{equation}
\label{dkhddhhdjh112}
\sup_{J}\frac{1}{|J|}\sum_{I\subseteq J\atop{J\; \operatorname{ dyadic}}}|c_I|^2=\infty.
\end{equation}

It can be easily seen by using \eqref{02} that 
\begin{equation}
\label{BMOtypeconv}
\lim_{\kappa\to\infty}\int \sum_{|I|\le \kappa}c_I\psi_I(y) f_3(y)dy
\end{equation}
exists for each Schwartz function $f_3$ with mean $0$. Define
 
$$
\Lambda(f_1,f_2,f_3)=\sum_{I\; \operatorname{ dyadic}}c_I\int_\R \int_\R
\left[\prod_{i=1}^3 f_i(x+\beta_i t) \phi_{i,I}(x+\beta_i t)\right]
\, dx \, dt.$$
By reasoning as before, it is easy to check that the bilinear form $\Lambda(.,.,1)$ satisfies the weak  boundedness condition, and moreover that $\Lambda(.,1,1)=\Lambda(1,.,1)=0$. Remark \ref{equivTcondbmo} shows that these imply that $\Lambda(.,.,1)$ is a bilinear form bounded on $L^2\times L^2$. Also, by \eqref{BMOtypeconv} it easily follows that the action on $H^1(R)$ atoms of $\Lambda(1,1,.)$ coincides with that of $\sum_{I}c_I\psi_I$. Due to \eqref{dkhddhhdjh112}, $\Lambda(1,1,.)$ can not be identified with a BMO function. By invoking Remark \ref{equivTcondbmo} again, it follows that the bilinear form $\Lambda(1,.,.)$ is not bounded, in spite of being completely represented by the same kernel as $\Lambda(.,.,1)$. Moreover, the trilinear form is itself unbounded, since otherwise $\Lambda(1,1,.)$ would necessarily have to be a BMO function. 
Hence $\Lambda(.,.,1)$ is a bounded bilinear Calderón-Zygmund
form associated with a Calderón-Zygmund kernel that is not the restriction of a bounded trilinear form with the given parameter $\beta$ and associated with the same $K$.


\begin{thebibliography}{99}

\bibitem{BNT} B\'enyi \'A., Nahmod A.R., and Torres R.H., {\em Sobolev
space estimates and symbolic calculus for bilinear pseudodifferential operators},
J. Geom. Anal. {\bf 16.3}, pp. 431-453, [2006].

\bibitem{calderon1} Calder\'on A.P., {\em Commutators of singular
integral operators}, Proc. Natl. Acad. Sci. USA {\bf 53}, pp.
1092-1099, [1977].

\bibitem{calderon2} Calder\'on A.P., {\em Cauchy integrals on Lipschitz curves and
related operators}, Proc. Natl. Acad. Sci. USA {\bf 74}, pp.
1324-1327, [1977].

\bibitem{carleson} Carleson L., {\em On convergence and growth of
partial sums of Fourier series}, Acta\ Math. {\bf 116}, pp. 135-157,
[1966].


\bibitem{CJ} Christ M. and Journ\'{e} J.-L., {\em Polynomial growth estimates for multilinear singular
integral operators}, Acta\ Math. {\bf 159}, pp. 51-80, [1987].

\bibitem{CM1} Coifman R.R. and  Meyer Y., {\em Commutateurs d' integrales singuli\`eres et
op\'erateurs multilin\'eaires}, Ann.  Inst. Fourier (Grenoble) {\bf
28}, pp. 177-202, [1978].

\bibitem{CM2} Coifman R.R. and Meyer Y., {\em Fourier analysis of
multilinear convolutions, Calder\'on's theorem and analysis of Lipschitz
curves}, Euclidean harmonic analysis (Proc. Sem. Univ. Maryland, College
Park, Md.), pp. 104-122, Lecture Notes in Math. {\bf 779}, pp. 104-122, [1979].

\bibitem{CM3} Coifman R.R. and Meyer Y., {\em Ondelettes and op\'erateurs III, Operat\'eurs
multilin\'eaires},  Actualit\'es Mathematiques, Hermman, Paris, [1991].

\bibitem{DJ} David G. and Journé J.L., {\em A boundedness criterion for generalized Calder\'on-Zygmund operators},
Ann. of Math. {\bf 120}, pp. 371-397, [1984].


\bibitem{DTT} Demeter C., Tao, T and Thiele C. {\em Maximal Multilinear Operators}, to appear in TAMS.

\bibitem{D} Duoandikoetxea J., {\em An\'alisis de Fourier.}, Addison-Wesley/Univ. Aut. de Madrid, [1995].

\bibitem{fefferman} Fefferman C., {\em Pointwise convergence of
Fourier series}, Ann. of Math. {\bf 98}, pp. 551-571, [1973].

\bibitem{FJ} Frazier M. and Jawerth B., {\it A discrete transform and decompositions of distribution spaces,}
J. Funct. Anal. {\bf 93.1}, pp. 34-170, [1990].

\bibitem{GN1} Gilbert J. and  Nahmod A., {\em Boundedness of bilinear
operators with non-smooth symbols}, Math. Res. Lett. {\bf 7}, pp. 767-778, [2000].

\bibitem{GN2} Gilbert J. and Nahmod A., {\em Bilinear operators with non-smooth symbols. I},
J. Fourier Anal. Appl. {\bf 5}, pp. 435-467, [2001].

\bibitem{GN3} Gilbert J. and Nahmod A., {\em $L^p$-boundedness for time-frecuency paraproducts. II},
J. Fourier Anal. Appl. {\bf 8}, pp. 109-172, [2002].

\bibitem{GL1} Grafakos L. and Li X., {\em Uniform bounds for the bilinear  Hilbert transform I },
Ann. of Math. {\bf 159.3}, pp. 889-993, [2004].

\bibitem{GT} Grafakos L. and Torres R.H., {\em Multilinear Calder\'on-Zygmund theory }, Adv. in Math. {\bf
165}, pp. 124-164, [2002].


\bibitem{LT1} Lacey M. and Thiele C., {\em $L^p$ bounds on the bilinear Hilbert
transform for $2<p<\infty $}, Ann. of Math. {\bf 146}, pp. 693-724, [1997].

\bibitem{LT2} Lacey M. and Thiele C., {\em On Calder\'on's conjecture.},
Ann. of Math. {\bf 149.2}, pp. 475-496, [1999].

\bibitem{GL2} Li X., {\em Uniform bounds for the bilinear  Hilbert transform II}, Rev. Mat. Iberoamer., to appear.

\bibitem{M} Meyer, Y., {\it Les nouveaux op\'erateurs de Calder\'on-Zygmund}, Colloquium in honor of Laurent Schwartz,
Vol. 1 (Palaiseau, 1983), Ast\'erisque {\bf 131}, pp. 237-254, [1985].

\bibitem{MTT1} Muscalu C., Tao T., and Thiele C., {\em Multilinear operators given by singular
multipliers}, J. Amer. Math. Soc. {\bf 15}, pp. 469-496, [2002].

\bibitem{MTT2} Muscalu C., Tao T., and Thiele C., {\em Uniform estimates
on multi-linear operators with modulation symmetry}, J. Anal. {\bf 88}, pp. 255-307, [2002].

\bibitem{MTT4} Muscalu, C., Tao, T. and Thiele, {\em $L\sp p$ estimates for the biest. I. The Walsh case.}  Math. Ann. {\bf 329}  (2004),  no. 3,  401-426.

\bibitem{MTT3} Muscalu, C., Tao, T. and Thiele, {\em $L\sp p$ estimates for the biest. II. The Fourier case.}  Math. Ann. {\bf 329}  (2004),  no. 3, 427-461.

\bibitem{ST} Stein E. M., {\em Harmonic Analysis: real-variable methods,
orthogonality and oscilatory integrals}, Princeton Univ. Press, [1993].

\bibitem{thielelectures} Thiele C., {\em Wave packet analysis}, CBMS {\bf 105}, [2006].

\bibitem{Th1} Thiele C., {\em A uniform estimate}  Ann. of Math. (2)  {\bf 156}  (2002),  no. 2, 519-563.

\bibitem{To} Torres R.H., {\em Boundedness results for operators with singular kernels on distribution spaces},  Mem. Amer. Math. Soc. {\bf 90} , no. 442 , [1991].
\end{thebibliography}
\end{document}